\documentclass[11pt,a4paper,leqno]{amsart}
\usepackage{amssymb,amsmath,amsthm}

\usepackage[mathscr]{eucal} 
\usepackage{graphicx}
\usepackage[colorlinks=true,linkcolor=Red,citecolor=Green,urlcolor=Fuchsia]{hyperref}
\usepackage{pdfsync}
\usepackage{enumerate} 
\usepackage{dsfont}
\usepackage{esint}
\usepackage[utf8]{inputenc}
\usepackage[english]{babel}
\usepackage{csquotes}
\MakeOuterQuote{"}
\usepackage[dvipsnames]{xcolor}
\usepackage{mathtools}

\allowdisplaybreaks
\theoremstyle{plain}
\newtheorem{theorem}{Theorem}[section]
\newtheorem{definition}[theorem]{Definition}
\newtheorem{assumption}[theorem]{Assumption}
\newtheorem{lemma}[theorem]{Lemma}
\newtheorem{corollary}[theorem]{Corollary}
\newtheorem{proposition}[theorem]{Proposition}
\theoremstyle{remark}
\newtheorem{remark}[theorem]{Remark}

\newtheorem{claim}[theorem]{Claim}

\numberwithin{equation}{section}

\def\dd{\,\mathrm  d}
\newcommand{\loc}{\rm loc}
\renewcommand{\div}{\operatorname{div}}
\newcommand{\curl}{\operatorname{curl}}

\newcommand{\supp}{\operatorname{supp}}

\renewcommand{\leq}{\leqslant}
\renewcommand{\geq}{\geqslant}
\renewcommand{\epsilon}{\varepsilon}
\renewcommand{\phi}{\varphi}
\newcommand{\norm}[1]{\left\lVert#1\right\rVert}

\def\Tend#1#2{\mathop{\longrightarrow}\limits_{#1\rightarrow#2}}

\def\wstarcv#1#2{\underset{#1 \rightarrow #2}{\xrightharpoonup{\; \; \ast \; \;}}}

\def\jbracket#1{<\! #1 \!>}

\date\today

\author[M. Ménard]{Matthieu Ménard}

\email{matthieu.menard@univ-grenoble-alpes.fr}

\address{Univ. Grenoble Alpes, CNRS, Institut Fourier, F-38000 Grenoble, France.}

\title{Mean-field limit of point vortices for the lake equations.}

\begin{document}
\begin{abstract}
In this paper we study the mean-field limit of a system of point vortices for the lake equations. These equations model the evolution of the horizontal component of the velocity field of a fluid in a lake of non-constant depth, when its vertical component can be neglected.  As for the axisymmetric Euler equations there are non-trivial self interactions of the vortices consisting in the leading order of a transport term along the level sets of the depth function. 

If the self-interactions are negligible, we show that the system of point vortices converges to the lake equations as the number of points becomes very large. If the self-interactions are of order one, we show that it converges to a forced lake equations and if the self-interactions are predominant, then up to time rescaling we show that it converges to a transport equation.

The proof is based on a modulated energy approach introduced by Duerinckx and Serfaty in (Duke Math. J., 2020) that we adapt to deal with the heterogeneity of the lake kernel.
\end{abstract}

\maketitle

\section{Introduction}

\subsection{Lake equations}

The purpose of this article is to investigate the mean-field limit of point vortices (which are dirac masses in the vorticity field of a fluid) in a lake of non-constant depth. Namely we want to establish the convergence of an empirical distribution of point vortices to a continuous density solving the lake equations, as the number of vortices becomes very large. These equations describe the evolution of the horizontal velocity field of an incompressible fluid in a lake, when:
\begin{itemize}
\item The depth is small with respect to the lengthscale of horizontal variations of the fluid velocity.
\item The surface of the fluid is almost flat (small Froude number).
\item The vertical velocity is small with respect to the horizontal velocity.
\end{itemize}

For a rigorous derivation of these equations from the shallow water system we refer to the work of Bresch, Gisclon and Lin in \cite{BreschGisclonLin}. A more general introduction to depth-averaged models can be found in \cite[Chapter~5]{Greenspan} and a discussion on the three upper hypothesis can be found in \cite{Richardson}.

These equations are similar to the planar Euler equations, but they take into account the depth of the lake, given by a positive function $b$. If $b$ is constant, then one recovers the usual planar Euler equations. The well-posedness of the lake equations on bounded domains have been first investigated by Levermore, Oliver and Titi in \cite{LevermoreOliverTiti}. In this paper they studied an analogue of the Yudovich theorem for Euler equations (see \cite{Yudovich}). This result was extended later by Bresch and Métivier in \cite{BreschMetivier} to include the case where the depth function $b$ vanishes at the boundary and by Lacave, Nguyen and Pausader in \cite{LacaveNguyenPausader} to deal with the case of rough bottoms. The existence and uniqueness of global classical solutions have been established by Al Taki and Lacave in \cite{AlTakiLacave}.

In this paper we will study the case of an infinite lake modeled by the whole plane $\mathbb{R}^2$. We are interested in the following vorticity form of the equations:

\begin{equation}\label{lake_equation_vorticity}
\left\{
\begin{aligned}
& \partial_t \omega + \div\left(\left(u-\alpha\frac{\nabla^\bot b}{b}\right)\omega\right) = 0 \\
& \div(bu) = 0 \\
& \curl(u) = \omega
\end{aligned}\right.
\end{equation}
where 
\begin{itemize}

\item $\bot$ denotes the rotation by $\displaystyle{\frac{\pi}{2}}$ (that is $(x_1,x_2)^\bot := (-x_2,x_1)$).
\item $\alpha \in [0,+\infty)$ is a forcing parameter.
\item $b : \mathbb{R}^2 \longrightarrow (0,+\infty)$ is the depth function satisfying Assumption~\ref{assumption_b} below.
\item $u : [0,+\infty)\times \mathbb{R}^2 \longrightarrow \mathbb{R}^2$ is the velocity field of the fluid.
\item $\omega : [0,+\infty)\times \mathbb{R}^2 \longrightarrow \mathbb{R}$ is the vorticity field of the fluid, defined by 
\begin{equation*}
\omega = \curl(u) := \partial_1 u_2 - \partial_2 u_1. 
\end{equation*}
\end{itemize} 

The true lake equations have no forcing term ($\alpha=0$), but we will study this more general model as it could arise as a mean-field limit of point vortices (in the regime where the self-interaction of the vortices are not negligible). It is a particular case of a model studied by Duerinckx and Fischer (see \cite[Equation~(1.9)]{DuerinckxFischer}). In this work the authors proved the global existence and uniqueness of weak solutions and the local well-posedness of strong solutions. We will consider the following definition of weak solutions:

\begin{definition}\label{definition_weak_solution}
Let $T > 0$ and $\omega_0 \in L^\infty(\mathbb{R}^2)$ with compact support. We say that $(\omega,u)$ is a weak solution of \eqref{lake_equation_vorticity} on $[0,T]$ with initial condition $\omega_0$ if $\omega \in L^1([0,T],L^\infty(\mathbb{R}^2,\mathbb{R}^2)) \cap \mathcal{C}^0([0,T],L^\infty(\mathbb{R}^2)-w^\ast)$ with compact support in space for all $t \in [0,T]$, $u \in L^2_{\loc}([0,T]\times\mathbb{R}^2,\mathbb{R}^2)$, for almost every $t \in [0,T]$, $\div(bu) = 0$ and $\curl(u) = \omega$ distributionally and for all $\phi$ smooth with compact support in $[0,T)\times\mathbb{R}^2$ and $t \in [0,T)$, 
\begin{equation}\label{formulation_faible}
\iint_{[0,t]\times\mathbb{R}^2} \partial_t \phi \omega + \nabla \phi \cdot\left(u-\alpha\frac{\nabla^\bot b}{b}\right)\omega = \int_{\mathbb{R}^2} \phi(t)\omega(t)-\int_{\mathbb{R}^2} \phi(0)\omega_0.
\end{equation}
\end{definition}

In the regime where the self-interaction of the point vortices is predominant, the system of point vortices will converge in an accelerated timescale to a transport equation along the level sets of the topography:
\begin{equation}\label{transport_equation}
\partial_t \overline{\omega} - \div\left(\frac{\nabla^\bot b}{b}\overline{\omega}\right) = 0.
\end{equation}

For this equation we will use the following definition of weak solutions:
\begin{definition}\label{definition_weak_solution_transport}
Let $T > 0$ and $\overline{\omega}_0 \in L^\infty(\mathbb{R}^2)$ with compact support. We say that is a weak solution of \eqref{transport_equation} on $[0,T]$ with initial condition $\overline{\omega}_0$ if $\overline{\omega}\in L^1([0,T],L^\infty(\mathbb{R}^2,\mathbb{R}^2))\cap  \mathcal{C}^0([0,T],L^\infty(\mathbb{R}^2)-w^\ast)$ with compact support  in space for all $t \in [0,T]$ and for all $\phi$ smooth with compact support in $[0,T)\times\mathbb{R}^2$ and $t \in [0,T)$,
\begin{equation}\label{formulation_faible_transport}
\iint_{[0,t]\times\mathbb{R}^2} \partial_t \phi \overline{\omega}- \nabla \phi \cdot\frac{\nabla^\bot b}{b}\overline{\omega} = \int_{\mathbb{R}^2} \phi(t)\overline{\omega}(t)-\int_{\mathbb{R}^2} \phi(0)\overline{\omega}_0.
\end{equation}
\end{definition}

\subsection{Point vortices for the lake equations}

The forced lake equations \eqref{lake_equation_vorticity} have been derived as the mean-field limit of complex Ginzburg-Landau vortices with forcing and pinning effects by Duerinckx and Serfaty in \cite{DuerinckxSerfaty}. The dynamics of these vortices comes from the physics of supraconductors or superfluids and is very different from the dynamics of vortices in a lake. In this paper we are interested in deriving Equations \eqref{lake_equation_vorticity} as the mean-field limit of a model introduced by Richardson in \cite{Richardson}. In that work he established by a formal computation the equation followed by the center of vorticity $q(t)$ of a small vortex of size $\epsilon$ in a lake of depth $b$. To leading order in $\epsilon$, this equation gives

\begin{equation}\label{equation_pvs_richardson}
\dot q(t) \approx -\frac{\Gamma|\ln(\epsilon)|}{4\pi}\frac{\nabla^\bot b(q(t))}{q(t)}
\end{equation}
where $\Gamma$ is the intensity of vorticity (that is $\displaystyle{\Gamma =\int_{B(q(0),\epsilon)} \omega}$).

This means that to leading order in $\epsilon$, a very small vortex follows the level lines of the topography without seeing the interaction with other vortices remaining far from him. The latter equation was rigourously justified by Dekeyser and Van Schaftingen in \cite{DekeyserVanSchaftingen} for the motion of a single vortex and this result was extended later to the case of a finite number of vortices by Hientzsch, Lacave and Miot in \cite{HientzschLacaveMiot}. 

We want to investigate the behavior of $N$ point  vortices of intensity $N^{-1}$ as $N$ becomes large. We will see in Section~\ref{section:2} that the elliptic problem
\begin{equation*}
\left\{
\begin{aligned}
& \div(bu) = 0 \\
& \curl(u) = \omega
\end{aligned}\right.
\end{equation*}
has a unique solution given by the kernel
\begin{equation*}
g_b(x,y) := \sqrt{b(x)b(y)}g(x-y) + S_b(x,y)
\end{equation*}
where $S_b$ is a function solving a certain elliptic equation (see Equation \eqref{definition_S_b}) and $\displaystyle{g(x) := -\frac{1}{2\pi}\ln|x|}$ is the opposite of the Green kernel of the Laplacian on the plane $\mathbb{R}^2$. More precisely, we have
\begin{equation*}
u(x) = -\frac{1}{b(x)}\int_{\mathbb{R}^2} \nabla_x^\bot g_b(x,y)\omega(y)\dd y.
\end{equation*}
Recall that a point vortex is asymptotically represented by a dirac mass of vorticity. Therefore using the kernel $\nabla_x^\bot g_b$ we can compute the velocity field generated by $N-1$ vortices $\delta_{q_j}$ of intensity $ \displaystyle{\frac{1}{N}}$ on a vortex $\delta_{q_i}$: 
\begin{equation*}
-\frac{1}{N}\underset{j\neq i}{\sum_{j=1}^N}\frac{1}{b(q_i)}\nabla^\bot_x g_b(q_i,q_j).
\end{equation*}
This term correspond to the term $u_{reg}$ given by Richardson in \cite[Equation~(2.90)]{Richardson}. Combining this equation with the self-interaction term of \eqref{equation_pvs_richardson} we get the model of point vortices we will study in this paper:
\begin{equation}\label{equation_pv}
\dot q_i = -\alpha_N\frac{\nabla^\bot b(q_i)}{b(q_i)}
-\frac{1}{N}\underset{j\neq i}{\sum_{j=1}^N}\frac{1}{b(q_i)}\nabla^\bot_x g_b(q_i,q_j)
\end{equation}
where we have denoted
\begin{equation*}
\alpha_N := \frac{|\ln(\epsilon_N)|}{4\pi N}
\end{equation*}
where $\epsilon_N$ is the size of the vortices.

\begin{remark}
Up to now there is no mathematical justification of Equation \eqref{equation_pv}: We do not even expect this equation to describe precisely the motion of a fixed number of small vortices as we have neglected all self-interaction terms of order smaller than $|\ln(\epsilon)|$. However Theorem~\ref{MFL_theorem} will justify that this simplified model is statistically relevant when $N$ becomes very large.
\end{remark}

\begin{remark}
There are several works establishing approximate analytical trajectories of vortices in a lake for some specific depth profiles, and also other numerical and experimental results related to vortex dynamics in lakes. For more details we refer to the results of \cite{Richardson} and the associated bibliography.
\end{remark}

Two quantities will be of interest for the study of this system. The interaction energy
\begin{equation*}
E_N(t) := \frac{1}{N^2}\sum_{i=1}^N\underset{j \neq i}{\sum_{j=1}^N} g_b(q_i(t),q_j(t))
\end{equation*}
and the moment of inertia
\begin{equation*}
I_N(t) := \frac{1}{N}\sum_{i=1}^N|q_i(t)|^2.
\end{equation*}
One could prove that the total energy
\begin{equation*}
E_N^{tot} := E_N + \frac{\alpha_N}{N}\sum_{i=1}^N b(q_i)
\end{equation*}
is a conserved quantity for the point vortex system \eqref{equation_pv} or that if $\omega$ is a solution of \eqref{lake_equation_vorticity} with enough regularity and decay, the quantity 
\begin{equation*}
\iint_{\mathbb{R}^2\times\mathbb{R}^2} g_b(x,y)\omega(t,x)\omega(t,y)\dd x \dd y + \alpha \int_{\mathbb{R}^2} b(x)\omega(t,x) \dd x
\end{equation*}
is conserved by the flow. The moment of inertia $I_N$ and the interaction energy $E_N$ are not conserved quantity but they are bounded in time, and this will be useful both for our mean-field limit result and for the well-posedness of System \eqref{equation_pv} (see Section~\ref{section:3}).

If $\alpha_N \Tend{N}{+\infty} +\infty$ the self-interactions are predominant. In order to study this regime we will study an accelerated timescale as it was done in \cite{DekeyserVanSchaftingen} and \cite{HientzschLacaveMiot} to study the motion of a finite number of vortices. Therefore we define:
\begin{equation*}
\overline{q_i}(t) := q_i(\alpha_N^{-1}t).
\end{equation*}
This gives
\begin{equation}\label{equation_pv_rescaled}
\dot{\overline{q_i}} = -\frac{\nabla^\bot b(\overline{q_i})}{b(\overline{q_i})}
-\frac{1}{N\alpha_N}\underset{j\neq i}{\sum_{j=1}^N}\frac{1}{b(\overline{q_i})}\nabla^\bot_x g_b(\overline{q_i},\overline{q_j})
\end{equation}

We also define the rescaled interaction energy
\begin{equation*}
\overline{E_N}(t) := E_N(\alpha_N^{-1}t)
\end{equation*}
and the rescaled moment of inertia
\begin{equation*}
\overline{I_N}(t) := I_N(\alpha_N^{-1}t).
\end{equation*}

\subsection{Mean-field limits}

Mean-field limits consist in studying the convergence of a system of ordinary differential equations modeling the evolution of a finite number of particles
\begin{equation}\label{edo_generic_order_1}
\dot x_i = \frac{1}{N}\sum_{i=1}^NK(x_i - x_j)
\end{equation}
to a Euler-like equation modeling the evolution of a continuous density $\mu(t,x)$:
\begin{equation}\label{edp_generic_order_1}
\partial_t \mu + \div((K*\mu)\mu) = 0
\end{equation}
when the number of particles becomes large (here $K : \mathbb{R}^d \longrightarrow \mathbb{R}^d$ is an interaction kernel). For systems of order two we are interested in the convergence of a system of particles following Newton's second law
\begin{equation*}
\ddot x_i = \frac{1}{N}\sum_{i=1}^NK(x_i - x_j)
\end{equation*}
to a Vlasov-like equation modeling the evolution of a continuous density $f(t,x,v)$:
\begin{equation}\label{edp_generic_order_2}
\left\{
\begin{aligned}
& \partial_t f + \div_x(fv) + \div_v((K*\mu)f) = 0 \\
& \mu(t,x) = \int_{\mathbb{R}^d} f(t,x,v)\dd v.
\end{aligned}\right.
\end{equation}

A mean-field limit result consists in proving that if at time zero, the empirical distribution of the particles 
\begin{equation*}
\frac{1}{N}\sum_{i=1}^N \delta_{x_i(t)} \qquad \left(\text{respectively} \quad \frac{1}{N}\sum_{i=1}^N \delta_{(x_i(t),\dot x_i(t))}\right)
\end{equation*}
converges to the continuous density $\mu(t,x)$ solution of \eqref{edp_generic_order_1} (respectively $f(t,x,v)$ solution of \eqref{edp_generic_order_2}) then the convergence also holds for any finite time.

When $K$ is Lipschitz the mean-field limit of the upper system was established by compactness arguments in \cite{BraunHepp,NeuzertWick} or by optimal transport theory and Wasserstein distances by Dobrushin in \cite{Dobrushin}. If $K$ is singular there are numerous results establishing the mean-field limit of systems of order one:

Schochet proved in \cite{Schochet1} the mean-field convergence of the point vortex system (that is $\displaystyle{K = \frac{1}{2\pi}\frac{x^\bot}{|x|^2}}$ in dimension 2) to a measure-valued solution of Euler equations up to a subsequence, using arguments previously developed in \cite{Delort} and \cite{Schochet2} to prove existence of such solutions.

For sub-coulombic interactions, that is $|K(x)|, |x||\nabla K(x)| \leq C|x|^{-\alpha}$ with $0 < \alpha < d-1$, the mean-field limit of \eqref{edo_generic_order_1} was proved by Hauray in \cite{Hauray} assuming $\div(K) = 0$ and using a Dobruschin-type approach (following the idea of \cite{HaurayJabin2,HaurayJabin1}). It was also used by Carillo, Choi and Hauray to study with the mean-field limit of some aggregation models in \cite{CarrilloChoiHauray}.

In \cite{Duerinckx} Duerinckx gave another proof of the mean-field limit of several Riesz interaction gradient flows using a "modulated energy" that was introduced by Serfaty in \cite{Serfaty2017}.

In \cite{Serfaty}, Serfaty used this modulated energy approach to prove the mean-field convergence of such systems where $K$ was a kernel given by Coulomb, logarithmic or Riesz interaction, that is $K = \nabla g$ for $g(x) = |x|^{-s}$ with $\max(d-2,0) \leq s < d$ for $d \geq 1$ or $g(x) = -\ln|x|$ for $d = 1$ or $2$. For this purpose $K*\mu$ was supposed to be Lipschitz.

Rosenzweig proved in \cite{Rosenzweig} the mean-field convergence of the point vortex system without assuming Lipschitz regularity for the limit velocity field, using the same energy as in \cite{Serfaty} with refined estimates. Remark that it ensures that the point vortex system converges to any Yudovich solutions of the Euler equations (see \cite{Yudovich}). This result was extended later for higher dimensional systems ($d \geq 3$) in \cite{Rosenzweig4} by the same author.

In \cite{NguyenRosenzweigSerfaty} Nguyen, Rosenzweig and Serfaty extended the modulated energy approach to a more general class of potentials $g$ using the commutator structure of the equations.

With a modulated energy approach, Bresch, Jabin and Wang defined a modulated entropy functionnal which allowed them to prove mean-field limit of interacting particles with noise in  \cite{BreschJabinWang,BreschJabinWang2,BreschJabinWang3}. This method was used later to obtain uniform in time convergence for Riesz-type flows by Rosenzweig and Serfaty in \cite{RosenzweigSerfaty} and by Rosenzweig, Serfaty and Chodron de Courcel in \cite{ChodronDeCourcelRosenzweigSerfaty}.

For systems of order two, the mean-field limit has been established for several singular kernels:

In \cite{HaurayJabin2,HaurayJabin1}, Hauray and Jabin dealt with the case of some sub-coulombian interactions (or more precisely $|K(x)| \leq c|x|^{-s}$ with $0 < s < 1$) by using a Dobrushin-type approach.

In \cite{JabinWang2,JabinWang1}, Jabin and Wang studied the case of bounded and $W^{-1,\infty}$ gradients.

In \cite{BoersPickl,HuangLiuPickl,Lazarovici,LazaroviciPickl} the same kind of results is proved with some cutoff of the interaction kernel.

In the appendix of \cite{Serfaty}, Duerinckx and Serfaty studied the case of particles interacting with a Coulomb or a Riesz interaction kernel to the Vlasov equation in the monokinetic regime, that is the pressureless Euler-Poisson equations. The same method have been used to study the mean-field limit of more general models coming from quantum physics, biology or fluid dynamics (see for example \cite{CarilloChoi,M1preprint,BenPorat}).

In \cite{HanKwanIacobelli}, Han-Kwan and Iacobelli proved the mean-field limit of particles following Newton's second law to the Euler equation in a quasineutral regime or in the gyrokinetic limit. This result was extended later by Rosenzweig in \cite{Rosenzweig2} to allow a larger choice of scaling between the number of particles and the coupling constant.

Recently, Bresch, Jabin and Soler were able in \cite{BreschJabinSoler} to prove the mean-field limit derivation of the Vlasov-Fokker-Planck equation with the true Coulomb interactions using the BBGKY hierarchy and the diffusivity in the velocity variables to get estimates on the marginals. 

Numerous other mean-field limit results were proved for interacting particles with noise with regular or singular kernels. See for example \cite{BermanOnheim,BolleyChafaiJoaquin,CarilloFerreiraPrecioso,FournierHaurayMischler,JabinWang2,JabinWang1,Lacker,LiLiuYu,NguyenRosenzweigSerfaty,Osada}. For a more complete bibliography on the mean-field limit of interacting particles with noise we refer to the bibliography of \cite{ChodronDeCourcelRosenzweigSerfaty}.

For a general introduction to the subject of mean-field limits we refer to the reviews \cite{Golse,Jabin}.

\subsection{Notations and assumptions}

\subsubsection{Notations}

\begin{itemize}

\item For $u \in L^1_{\loc}(\mathbb{R}^2,\mathbb{R}^2)$, we denote $\curl(u) = \partial_1 u_2 - \partial_2 u_1$.

\item For $h \in \dot H^1(\mathbb{R}^2)$, we denote
\begin{equation}\label{definition_SE}
[h,h]_{i,j} := 2\partial_i h \partial_j h - |\nabla h|^2\delta_{i,j}.
\end{equation}
It is the stress-energy tensor  used in \cite{Serfaty} to prove the mean-field limit of several singular ODE's. Remark that for $h$ smooth enough, we have
\begin{equation*}
\div [h,h] = 2\Delta h \nabla h.
\end{equation*}

\item We denote $\jbracket{x} = (1+|x|^2)^\frac{1}{2}$.

\item $g$ is the opposite of the Green function of the laplacian: 
\begin{equation*}
g(x) := -\frac{1}{2\pi}\ln|x|.
\end{equation*}

\item $|\cdot|_{\mathcal{C}^{0,s}}$ is the semi-norm associated to the Hölder space $\mathcal{C}^{0,s}$:
\begin{equation*}
|f|_{\mathcal{C}^{0,s}} = \sup_{x \neq y} \frac{|f(x) - f(y)|}{|x-y|^s}.
\end{equation*}

\item When $1 \leq p \leq +\infty$, $p'$ denotes the dual exponent of $p$.

\item If $\nu$ is a probability measure on $\mathbb{R}^2$, we will denote $\nu^{\otimes 2} := \nu \otimes \nu$.

\item $C$ is a generic constant. We will denote $C_{A,B}$ when a constant depends on some quantities $A$ and $B$.

\item $\mathcal{P}(\mathbb{R}^2)$ is the space of probability measures on $\mathbb{R}^2$.

\item For $Q_N = (q_1,...,q_N) \in (\mathbb{R}^2)^N$ we denote $\displaystyle{I(Q_N) = \frac{1}{N}\sum_{i=1}^N |q_i|^2}$. 

\end{itemize}

\subsubsection{Assumptions}

We will make the following assumption on the depth function $b$:
\begin{assumption}\label{assumption_b}
We assume that $b$ is a smooth function, $\inf{b} > 0$, $\sup{b} < +\infty$ and that there exists $\gamma > 0$ such that 
\begin{equation*}
(1+|x|)^{4+\gamma}(|\nabla b(x)| + |D^2 b(x)|) < +\infty.
\end{equation*}
\end{assumption}

We will consider regular solutions of \eqref{lake_equation_vorticity} and \eqref{transport_equation} in the following sense:
\begin{assumption}\label{assumption_omega}
We say that a function $\omega(t,x)$ satisfies Assumption~\ref{assumption_omega} if $\omega \in L^{\infty}([0,T],L^\infty(\mathbb{R}^2)\cap \mathcal{P}(\mathbb{R}^2))\cap \mathcal{C}^0([0,T],L^\infty(\mathbb{R}^2)-w^\ast)$, if there exists a compact $K$ such that for every $t \in [0,T]$, $\supp(\omega(t)) \subset K$ and if $\nabla G_b[\omega] \in L^\infty([0,T],W^{1,\infty})$ where $G_b$ is the operator defined by Equation \eqref{existence_noyau_green}. 
\end{assumption}

\begin{remark}
A weak solution of \eqref{lake_equation_vorticity} in the sense of Definition \ref{definition_weak_solution} (or a weak solution of \eqref{transport_equation} in the sense of  Definition \ref{definition_weak_solution_transport}) does not necessarily verify Assumption \ref{assumption_omega} because of the regularity we ask for the velocity field $\nabla G_b[\omega]$. This assumption will be crucial to apply Proposition \ref{controle_terme_principal_gronwall} and prove the mean-field limit Theorem \ref{MFL_theorem}. The existence and uniqueness of sufficiently regular solutions of \eqref{lake_equation_vorticity} locally in time is ensured by \cite[Theorem~2]{DuerinckxFischer}. One could also prove that $\omega \in L^\infty([0,T],\mathcal{C}^{0,s})$ is sufficient to have $\nabla G_b[\omega] \in L^\infty([0,T],W^{1,\infty})$.
\end{remark}

\subsection{Main result and plan of the paper}

The main result of this paper is the following theorem which gives the mean-field limit of the point vortex system \eqref{equation_pv} and its rescaled version \eqref{equation_pv_rescaled} (we recall that the kernel $g_b$ is defined by \eqref{definition_g_b}):

\begin{theorem}\label{MFL_theorem}

Assume that $b$ satisfies Assumption~\ref{assumption_b}. We have mean-field convergence of the point-vortex system in the two following regimes:

\begin{enumerate}

\item Let $\omega$ be a solution of \eqref{lake_equation_vorticity} with initial datum $\omega_0$ in the sense of Definition~\ref{definition_weak_solution}, satisfying Assumption~\ref{assumption_omega} and $(q_1,...,q_N)$ be a solution of \eqref{equation_pv}. Assume that:

\begin{itemize}

\item $(I_N(0))_N$ is bounded.

\item $\displaystyle{\frac{1}{N}\sum_{i=1}^N \delta_{q_i^0}} \wstarcv{N}{+\infty} \omega_0$ for the weak-$\ast$ topology of probability measures.

\item $\alpha_N \Tend{N}{+\infty} \alpha$.

\item  $\displaystyle{\frac{1}{N^2}\sum_{1 \leq i \neq j \leq N} g_b(q_i^0,q_j^0) \Tend{N}{+\infty} \iint_{\mathbb{R}^2\times\mathbb{R}^2} g_b(x,y)\omega_0(x)\omega_0(y)\dd x \dd y}$.

\end{itemize}

Then for all $t \in [0,T]$, $\displaystyle{\frac{1}{N}\sum_{i=1}^N \delta_{q_i(t)}} \wstarcv{N}{+\infty} \omega(t)$ for the weak-$\ast$ topology of probability measures and
\begin{equation*}
\frac{1}{N^2}\sum_{1 \leq i \neq j \leq N} g_b(q_i(t),q_j(t)) \Tend{N}{+\infty} \iint_{\mathbb{R}^2\times\mathbb{R}^2} g_b(x,y)\omega(t,x)\omega(t,y)\dd x \dd y.
\end{equation*} 

\item Let $\overline{\omega}$ be a solution of \eqref{transport_equation} with initial datum $\omega_0$ in the sense of Definition~\ref{definition_weak_solution_transport}, satisfying Assumption~\ref{assumption_omega} and $(\overline{q_1},...,\overline{q_N})$ be a solution of \eqref{equation_pv_rescaled}. Assume that:
\begin{itemize}

\item $(\overline{I_N}(0))_N$ is bounded.

\item $\displaystyle{\frac{1}{N}\sum_{i=1}^N \delta_{q_i^0}} \wstarcv{N}{+\infty}  \omega_0$ for the weak-$\ast$ topology of probability measures.

\item $\alpha_N \Tend{N}{+\infty} +\infty$.

\item  $\displaystyle{\frac{1}{N^2}\sum_{1 \leq i \neq j \leq N} g_b(q_i^0,q_j^0)  \Tend{N}{+\infty} \iint_{\mathbb{R}^2\times\mathbb{R}^2} g_b(x,y)\omega_0(x)\omega_0(y)\dd x \dd y}$.

\end{itemize}
Then for all $t \in [0,T]$, $\displaystyle{\frac{1}{N}\sum_{i=1}^N \delta_{\overline{q_i}(t)}} \wstarcv{N}{+\infty} \overline{\omega}(t)$ for the weak-$\ast$ topology of probability measures and
\begin{equation*}
\frac{1}{N^2}\sum_{1 \leq i \neq j \leq N} g_b(\overline{q_i}(t),\overline{q_j}(t)) \Tend{N}{+\infty} \iint_{\mathbb{R}^2\times\mathbb{R}^2} g_b(x,y)\overline{\omega}(t,x)\overline{\omega}(t,y)\dd x \dd y.
\end{equation*} 

\end{enumerate}
\end{theorem}

Remark that in the case $\alpha_N \Tend{N}{+\infty} 0$ we recover the classical lake equations (\eqref{lake_equation_vorticity} with $\alpha = 0$).

The boundedness of $(I_N(0))$ is a technical assumption made to ensure that not too much vorticity is going to infinity. This assumption was not necessary in  the original papers of Duerinckx in \cite{Duerinckx} and of Serfaty in \cite{Serfaty} but we will need it to deal with the heterogeneity of the kernel $g_b$ (defined in \eqref{definition_g_b}).

The convergence of the interaction energy and the weak$-\ast$ convergence of $(\omega_N)$ to $\omega$ ensure the convergence of $(\omega_N)$ to $\omega$ in a stronger sense: We will prove in Corollary~\ref{corollary_weak_star_cv} that provided certain technical assumptions are satisfied, it is equivalent to the convergence to zero of a "modulated energy" functionnal. For an empirical measure of point vortices $(q_1,...,q_N)$ and a vorticity field $\omega \in L^\infty$ with compact support, this modulated energy is defined by:

\begin{multline}\label{definition_modulated_energy}
\mathcal{F}_b(Q_N,\omega) := \\ \iint_{(\mathbb{R}^2\times\mathbb{R}^2)\backslash\Delta} g_b(x,y)\dd\left(\frac{1}{N}\sum_{i=1}^N \delta_{q_i} - \omega\right)(x)\dd\left(\frac{1}{N}\sum_{i=1}^N \delta_{q_i} - \omega\right)(y)
\end{multline}
where
\begin{equation*}
\Delta := \{(x,x) \; ; \; x \in \mathbb{R}^2\}.
\end{equation*}

We will use this energy to control the distance between solutions $\omega$ and $Q_N$ of \eqref{lake_equation_vorticity} and \eqref{equation_pv} or solutions $\overline{\omega}$ and $\overline{Q_N}$ of \eqref{transport_equation} and \eqref{equation_pv_rescaled} at any given time $t$:

\begin{equation}\label{definition_F_b_N}
\mathcal{F}_{b,N}(t) := \mathcal{F}_{b}(Q_N(t),\omega(t))
\end{equation}
and
\begin{equation}\label{definition_F_b_N_rescaled}
\overline{\mathcal{F}}_{b,N}(t) := \mathcal{F}_{b}(\overline{Q_N}(t),\overline{\omega}(t)).
\end{equation}

The proof of Theorem~\ref{MFL_theorem} relies on Grönwall-type estimates on these two quantities. The paper is organised as follows:

\begin{itemize}

\item In Section~\ref{section:2} we prove the well-posedness of the elliptic problem linking a velocity field satisfying $\div(bu) = 0$ and its vorticity, the existence of a Green kernel for this elliptic problem and we establish several regularity estimates.

\item In Section~\ref{section:3} we prove that the point-vortex system is well-posed and give some estimates on the interaction energy and on the moment of inertia of the system that we will need in Section~\ref{section:7}.

\item In Section~\ref{section:4} we compute the time derivative of $\mathcal{F}_{b,N}$ and of $\overline{\mathcal{F}}_{b,N}$.

\item In Section~\ref{section:5} we state several properties of the modulated energy. We prove that it controls the convergence in $H^s$ for $s < -1$ (see Corollary~\ref{corollary_coercivity}) and that having the convergence of the modulated energy is equivalent to have weak-$\ast$ convergence of the point vortex system and convergence of its interaction energy (see Corollary~\ref{corollary_weak_star_cv}).

\item In Section~\ref{section:6} we bound the main term appearing in the derivatives of the modulated energies.

\item In Section~\ref{section:7} we use the results of the other sections to prove Theorem~\ref{MFL_theorem}.

\end{itemize}

The modulated energy $\mathcal{F}_b$ is similar to the modulated energy defined in \cite[Equation~(1.16)]{Serfaty} and the proofs of Sections \ref{section:4} to \ref{section:7} follow the same global ideas. The main difference between Theorem~\ref{MFL_theorem} and other mean-field limit results using modulated energies is that the kernel $g_b$ is not of the form $a(x,y) = a(x-y)$. Most of the difficulties adressed by this paper consist in dealing with the heterogeneity of the kernel $g_b$.

\subsection*{Funding}

This work is supported by the French National Research Agency in the framework of the project “SINGFLOWS” (ANR-18-CE40-0027-01).

\section{Velocity reconstruction}\label{section:2}

There exists a Biot-Savart type law to reconstruct a velocity field $u$ satisfying $\div(bu) = 0$ from its vorticity. In this section we prove several results concerning this reconstruction. In Subsection~\ref{subsection:21} we prove that the elliptic equations linking $u$ with its vorticity are well-posed. In Subsection~\ref{subsection:22} we prove some results related to the asymptotic behavior of the velocity field as $|x| \longrightarrow \infty$. In Subsection~\ref{subsection:23}, we give an analogue of the Biot-Savart law for a velocity field satisfying System \eqref{elliptic_problem_velocity}. Finally, in Subsection~\ref{subsection:24} we define some regularisations of the Coulomb kernel and of the dirac mass that we will need in Sections~\ref{section:5} and \ref{section:6}.

\subsection{Well-posedness of the elliptic problem}\label{subsection:21}

In this subsection we justify the well-posedness of the elliptic equations satisfied by the velocity field:
\begin{equation}\label{elliptic_problem_velocity}
\left\{
\begin{aligned}
& \div(bu) = 0 \\
& \curl(u) = \omega.
\end{aligned}\right.
\end{equation}
As we will write $\displaystyle{u = -\frac{1}{b}\nabla^\bot \psi}$ we will also consider the "stream function" formulation of the upper system:
\begin{equation}\label{elliptic_problem_stream}
-\div\left(\frac{1}{b}\nabla \psi\right) = \omega.
\end{equation}
For this purpose we will consider the following weighted Sobolev spaces:
\begin{definition}
For $1 < p < \infty$ we consider the Banach space $W^{2,p}_{-1}(\mathbb{R}^2)$ defined by
\begin{equation*}
W^{2,p}_{-1}(\mathbb{R}^2) := \{u \in \mathcal{D}'(\mathbb{R}^2) \; ; \; \forall \alpha \in \mathbb{N}^2, |\alpha| \leq 2, \jbracket{\cdot}^{|\alpha|-1}D^\alpha u \in L^p(\mathbb{R}^2)\}
\end{equation*}
and equipped with its natural norm
\begin{equation*}
\norm{u}_{W^{2,p}_{-1}} := \left(\sum_{|\alpha| \leq 2} \norm{\jbracket{\cdot}^{|\alpha|-1}D^\alpha u}^p_{L^p}\right)^\frac{1}{p}.
\end{equation*}
\end{definition}

These weighted spaces were first introduced by Cantor in \cite{Cantor} and have been investigated to study elliptic equations on unbounded domains. For a more precise study of these spaces and further references we refer to \cite{LockhartMcOwen,McOwen79,McOwen80}. The following proposition is a straightforward consequence of \cite[Theorem~2]{LockhartMcOwen} (which is the combination of two theorems proved in \cite{McOwen79} and \cite{McOwen80}) and states that Equations \eqref{elliptic_problem_velocity} and \eqref{elliptic_problem_stream} are well-posed.

\begin{proposition}\label{proposition_well_posedness_elliptic}
Let $2 < p < +\infty$, assume that $\jbracket{\cdot}\omega \in L^p(\mathbb{R}^2)$, then there exists a unique solution $\psi$ of \eqref{elliptic_problem_stream} in $W^{2,p}_{-1}(\mathbb{R}^2) \slash \mathbb{R}$. Morever if $u \in L^p(\mathbb{R}^2,\mathbb{R}^2)$ is a solution of \eqref{elliptic_problem_velocity} in the sense of distributions, then
\begin{equation*}
u = -\frac{1}{b}\nabla^\bot \psi.
\end{equation*}
\end{proposition}

\begin{proof}
We can rewrite Equation \eqref{elliptic_problem_stream} as
\begin{align*}
-\Delta \psi - b\nabla\left(\frac{1}{b}\right)\cdot\nabla \psi = b\omega.
\end{align*}
We have that:
\begin{itemize}

\item $-\Delta$ is an elliptic operator with constant coefficients and homogeneous of degree $2$.

\item $\displaystyle{b\nabla\left(\frac{1}{b}\right) \in \mathcal{C}^0}$ and
\begin{equation*}
\underset{|x|\rightarrow+\infty}{\lim} \left|\jbracket{x}^{2-1+0}b(x)\nabla\left(\frac{1}{b}\right) (x)\right| = 0 
\end{equation*} 
since $b$ satisfies Assumption~\ref{assumption_b}.

\item $\jbracket{\cdot} b\omega \in L^p$.

\item $\displaystyle{-1\leq -\frac{2}{p}}$ and $\displaystyle{1-\frac{2}{p} \notin \mathbb{N}}$.

\end{itemize}
Therefore by \cite[Theorem~2]{LockhartMcOwen}, there exists a unique solution $\psi$ (up to a constant) of Equation \eqref{elliptic_problem_stream} in $W^{2,p}_{-1}(\mathbb{R}^2)$.

Now if $u \in L^p$ is a solution of \eqref{elliptic_problem_velocity}, then
\begin{align*}
\norm{\jbracket{\cdot}\curl(bu)}_{L^p} &= \norm{\jbracket{\cdot}b\omega}_{L^p}+\norm{\jbracket{\cdot}\nabla^\bot b\cdot u} \\
&\leq C_b(\norm{\jbracket{\cdot}\omega}_{L^p}+\norm{u}_{L^p})
\end{align*}
since $b$ satisfies Assumption~\ref{assumption_b}. Let us consider $\pi \in W^{2,p}_{-1}(\mathbb{R}^2)$ to be the unique solution (up to a constant) of $-\Delta \pi = \curl(bu)$ given by \cite[Theorem~1]{LockhartMcOwen}. Then $bu + \nabla^\bot \pi$ is a div-curl free vector field in $L^p$ so it is zero. Moreover,
\begin{align*}
-\div\left(\frac{1}{b}\nabla \pi\right) &= -\curl\left(\frac{1}{b}\nabla^\bot \pi\right) 
= \curl(u) = \omega
\end{align*} 
so $\nabla \pi = \nabla \psi$ by uniqueness of solutions of \eqref{elliptic_problem_stream} in $W^{2,p}_{-1}(\mathbb{R}^2)/\mathbb{R}$.
\end{proof}

Now we state several estimates for solutions of Equation \eqref{elliptic_problem_stream}, proved by Duerinckx in \cite{DuerinckxFischer}:   

\begin{lemma}\label{lemma_link_solutions_duerinckx}[From  \cite[Lemma~2.6]{DuerinckxFischer}]
Let $p>2$, $\omega$ be such that $\jbracket{\cdot}\omega \in L^p(\mathbb{R}^2)$. If $\psi \in W^{2,p}_{-1}(\mathbb{R}^2)$ is the solution of \eqref{elliptic_problem_stream} given by Proposition~\ref{proposition_well_posedness_elliptic}, then: 
\begin{enumerate}

\item There exists $p_0 >2$ depending only on $b$ such that for all $2 < p \leq p_0$,
\begin{equation*}
\norm{\nabla \psi}_{L^p} \leq C_p\norm{\omega}_{L^\frac{2p}{p+2}}.
\end{equation*}

\item For all $0 < s < 1$, 
\begin{align*}
|\nabla \psi|_{\mathcal{C}^{0,s}} \leq C_s \norm{\omega}_{L^\frac{2}{1-s}}.
\end{align*}

\item $\norm{\nabla \psi}_{L^\infty} \leq C\norm{\omega}_{L^1\cap L^\infty}$. 

\end{enumerate}
\end{lemma}

\begin{remark}
In \cite{DuerinckxFischer}, this lemma was stated for any solution of \eqref{elliptic_problem_stream} with decreasing gradient (which is the case for a solution given by Proposition~\ref{proposition_well_posedness_elliptic} since its gradient is in $W^{1,p}$) and for $\omega$ smooth with compact support but by density it can be extended to all $\omega$ such that $\jbracket{\cdot}\omega \in L^p(\mathbb{R}^2)$ and such that the upper inequalities make sense. 
\end{remark}

\subsection{Asymptotic behavior of the velocity field}\label{subsection:22}

The main result of this subsection is the following proposition giving the asymptotic behavior of a velocity field satisfying \eqref{elliptic_problem_velocity}.

\begin{proposition}\label{theorem_asymptotic_velocity_field_lake}
Let $\omega \in L^\infty$ with compact support and $\displaystyle{u = -\frac{1}{b}\nabla^\bot \psi}$ where $\psi$ is the solution of \eqref{elliptic_problem_stream} given by Proposition~\ref{proposition_well_posedness_elliptic}. There exists $C > 0$ depending only on $b$ and $\omega$ such that for all $x \in \mathbb{R}^2\backslash\{0\} $,

\begin{equation}\label{asymptotic_bound_velocity_field}
\left|u(x) - \frac{1}{2\pi}\left(\int_{\mathbb{R}^2} \omega\right) \frac{x^\bot}{|x|^2}\right| \leq \frac{C}{|x|^2}.
\end{equation}
Moreover there exists $\delta \in (0,1)$ and $C$ such that
\begin{equation}\label{bound_stream_function}
|\psi(x)| \leq C(1+|x|^\delta).
\end{equation}

\end{proposition}

To prove this proposition we will need to use the following result about the asymptotic behavior of a velocity field given by the usual Biot-Savart law:

\begin{lemma}\label{theorem_asymptotic_velocity_field_euler}
Let us assume that $\mu$ is a measurable function such that $\mu \in L^1((1+|x|^2)\dd x)$ and $|\cdot|^2\mu \in L^p$ for some $p > 2$. Then there exists $C,R >0$ depending only on $\mu$ such that for all $x \in \mathbb{R}^2\backslash\{0\}$,

\begin{equation*}
\left|\int_{\mathbb{R}^2}\frac{x-y}{|x-y|^2}\dd \mu(y)-\left(\int_{\mathbb{R}^2} \dd \mu(y)\right)\frac{x}{|x|^2}\right| 
\leq \frac{C}{|x|^2}.
\end{equation*}

In particular if $\displaystyle{\int_{\mathbb{R}^2} \dd\mu = 0}$, then

\begin{equation*}
\int_{\mathbb{R}^2}\frac{x-y}{|x-y|^2}\dd \mu(y) = \underset{|x|\rightarrow+\infty}{\mathcal{O}}(|x|^{-2}).
\end{equation*}

\end{lemma}
This lemma is a classical result in fluid dynamics (see for example \cite[Proposition~3.3]{MajdaBertozzi}) that we will prove for the sake of completeness. 

\begin{proof}
If $x \neq 0$, we have

\begin{equation*}
\int_{\mathbb{R}^2} \mu(y)\left(\frac{x-y}{|x-y|^2}-\frac{x}{|x|^2}\right)\dd y
= \frac{1}{|x|^2}\int_{\mathbb{R}^2} \mu(y)\frac{|x|^2(x-y) - x|x-y|^2}{|x-y|^2}\dd y.
\end{equation*}
Now remark that
\begin{align*}
|x|^2&(x-y) - x|x-y|^2 \\ &= |x|^2(x-y) - (x-y)(|x|^2 + |y|^2 -2x\cdot y) -y|x-y|^2 \\
&= (x-y)(|y|^2 -2(x-y)\cdot y -2|y|^2)  -y|x-y|^2 \\
&= -y|x-y|^2 - 2[(x-y)\cdot y](x-y) -|y|^2(x-y).
\end{align*}
Thus
\begin{align*}
\left|\int_{\mathbb{R}^2} \mu(y)\left(\frac{x-y}{|x-y|^2}-\frac{x}{|x|^2}\right)\dd y\right| \leq& \frac{C}{|x|^2}\bigg(\int_{\mathbb{R}^2} |y||\mu(y)|\dd y \\
&+ \int_{\mathbb{R}^2} \frac{|y|^2|\mu(y)|}{|x-y|}\dd y\bigg).
\end{align*}
Now we have that for any $p > 2$,
\begin{align*}
\int_{\mathbb{R}^2} \frac{|y|^2|\mu(y)|}{|x-y|}\dd y &\leq \norm{|\cdot|^2\mu}^\frac{p-2}{2p-2}_{L^1}\norm{|\cdot|^2\mu}_{L^p}^\frac{p}{2p-2}
\end{align*}
(see for example \cite[Lemma~1]{Iftimie}) and therefore we get the proof of Lemma~\ref{theorem_asymptotic_velocity_field_euler}.
\end{proof}

With this result we can now study the asymptotic behavior of a velocity field satisfying System \eqref{elliptic_problem_velocity}:
\begin{proof}[Proof of Proposition~\ref{theorem_asymptotic_velocity_field_lake}]
We write
\begin{equation}\label{expression_mu_bu}
\mu := \div(u) = \div\left(\frac{1}{b}bu\right) = \nabla\left(\frac{1}{b}\right)\cdot bu = -\frac{\nabla b\cdot u}{b}.
\end{equation}
By Helmholtz decomposition we can write
\begin{equation}\label{helmoltz_u}
u = -\nabla g\ast\mu -\nabla^\bot g\ast\omega.
\end{equation}
Let $2 < p < +\infty$, then by Assumption~\ref{assumption_b},
\begin{align*}
\int_{\mathbb{R}^2} (1+|y|^2) |\mu(y)|\dd y &\leq C_b\int_{\mathbb{R}^2} \frac{1+|y|^2}{(1+|y|)^{4+\gamma}}|u(y)|\dd y \\
&\leq C_b \norm{(1+|\cdot|)^{-(2+\gamma)}}_{L^{p'}}\norm{u}_{L^p} < +\infty
\end{align*}
and
\begin{align*}
\int_{\mathbb{R}^2} |y|^{2p}|\mu(y)|^p\dd y &\leq C_b\int_{\mathbb{R}^2} |y|^{2p}(1+|y|)^{-p(4+\gamma)}|u(y)|^p \dd y \\
&\leq C_b \int_{\mathbb{R}^2} |u(y)|^p \dd y < + \infty.
\end{align*}
If we apply Lemma~\ref{theorem_asymptotic_velocity_field_euler} on each term of \eqref{helmoltz_u} we only need to show that $\displaystyle{\int \mu = 0}$ to obtain \eqref{asymptotic_bound_velocity_field}. We define 
\begin{equation*}
b_\infty := \underset{|x| \rightarrow +\infty}{\lim} b(x).
\end{equation*}
Remark that the existence of this limit is guaranteed by Assumption~\ref{assumption_b}. Let us prove by induction that for any integer $n$,

\begin{equation}\label{hyp_recurrence_asymptotique}
\sum_{k=0}^n \frac{\ln^k(b_\infty)}{k!}\int_{\mathbb{R}^2} \mu = \frac{1}{n!}\int_{\mathbb{R}^2} \ln^n(b)\mu.
\end{equation}

If $n = 0$ then this equality reduces to $\displaystyle{\int \mu = \int \mu}$. Now let us assume that it holds for some $n \geq 0$. Using Equation \eqref{expression_mu_bu}, we get
\begin{equation*}
\ln^n(b)\mu = -\frac{1}{n+1}\nabla \ln^{n+1}(b) \cdot u.
\end{equation*}
Inserting Equation \eqref{helmoltz_u}, we get
\begin{equation*}
\ln^n(b)\mu = \frac{1}{n+1}\nabla \ln^{n+1}(b)\cdot(\nabla g\ast\mu + \nabla^\bot g\ast\omega).
\end{equation*}
Integrating over a ball of center $0$ and radius $R$ and integrating by parts we get
\begin{equation}\label{ipp_recurrence}
\begin{aligned}
\int_{B(0,R)} \ln^n(b)\mu
=& \frac{1}{n+1}\bigg(\int_{\partial B(0,R)} \ln^{n+1}(b)(\nabla g\ast\mu + \nabla^\bot g\ast\omega)\cdot \dd \vec{S}  \\
&- \int_{B(0,R)}\ln^{n+1}(b)\div(\nabla g\ast\mu + \nabla^\bot g\ast\omega)\bigg) \\
=&\frac{1}{n+1}\bigg(\int_{\partial B(0,R)} \ln^{n+1}(b)(\nabla g\ast\mu + \nabla^\bot g\ast\omega)\cdot \dd \vec{S}  \\
&+ \int_{B(0,R)}\ln^{n+1}(b)\mu\bigg) \\
\end{aligned}
\end{equation}
where $\dd \vec{S}(x) = 2\pi x \dd\sigma(x)$ and $\sigma$ is the uniform probability measure on $\partial B(0,R)$. Using Lemma~\ref{theorem_asymptotic_velocity_field_euler}, we get that for $x \in \partial B(0,R)$,
\begin{align*}
(\nabla g\ast\mu &+ \nabla^\bot g\ast\omega)(x)\cdot x \\
&= -\frac{1}{2\pi}\left(\left(\int_{\mathbb{R}^2} \mu\right)\frac{x}{|x|^2}+\left(\int_{\mathbb{R}^2} \omega\right)\frac{x^\bot}{|x|^2} + \mathcal{O}(R^{-2})\right)\cdot x \\
&= -\frac{1}{2\pi}\left(\int_{\mathbb{R}^2} \mu\right)+ \mathcal{O}(R^{-1}).
\end{align*}
Thus we get that
\begin{equation*}
\frac{1}{n+1}\int_{\partial B(0,R)} \ln^{n+1}(b)(\nabla g\ast\mu + \nabla^\bot g\ast\omega)\cdot \dd \vec{S} \Tend{R}{+\infty} -\frac{\ln^{n+1}(b_\infty)}{n+1}\int_{\mathbb{R}^2} \mu.
\end{equation*}
Combining the upper equality with Equations \eqref{hyp_recurrence_asymptotique} and \eqref{ipp_recurrence} we get that 
\begin{equation*}
\sum_{k=0}^{n+1} \frac{\ln^k(b_\infty)}{k!}\int_{\mathbb{R}^2} \mu = \frac{1}{(n+1)!}\int_{\mathbb{R}^2} \ln^{n+1}(b)\mu
\end{equation*}
which ends the proof of Equality \eqref{hyp_recurrence_asymptotique}. Now if $n$ goes to infinity, this gives
\begin{equation*}
e^{\ln(b_\infty)}\int_{\mathbb{R}^2} \mu = 0
\end{equation*}
and thus
\begin{equation*}
\int_{\mathbb{R}^2} \mu = 0.
\end{equation*}
Now by Lemma~\ref{lemma_link_solutions_duerinckx} and Morrey's inequality (see for example \cite[Theorem~9.12]{Brezis}), for any $2 < p \leq p_0$,
\begin{align*}
|\psi(x)| &\leq |\psi(x)- \psi(0)| + |\psi(0)| \\
&\leq C_p\norm{\nabla \psi}_{L^p}|x|^{1-\frac{2}{p}} + |\psi(0)|.
\end{align*}
Taking $\displaystyle{\delta=1-\frac{2}{p}}$ we obtain \eqref{bound_stream_function}.
\end{proof}

\subsection{Construction of the Green kernel}\label{subsection:23}

The main result of this subsection is a Biot-Savart type law for the lake equations, given by Proposition~\ref{proposition_bs_law}. Let us begin by giving the definition and some estimates on the function $S_b$ that appears in the definition of the kernel $g_b$ (see Equation \eqref{definition_g_b}):

\begin{lemma}\label{estimees_S_b_not_necessarily_symmetric}

For $y \in \mathbb{R}^2$, let $S_b(\cdot,y)$ be a solution of
\begin{equation}\label{definition_S_b}
-\div\left(\frac{1}{b}\nabla S_b(\cdot,y)\right) = -g(\cdot-y)\sqrt{b(y)}\Delta\left(\frac{1}{\sqrt{b}}\right)
\end{equation}
given by Proposition~\ref{proposition_well_posedness_elliptic} applied to $\displaystyle{\omega = -g(\cdot-y)\sqrt{b(y)}\Delta\left(\frac{1}{\sqrt{b}}\right)}$ and $\psi = S_b(\cdot,y)$. Then:
\begin{enumerate}

\item For any $y \in \mathbb{R}^2$ and $2 < p \leq +\infty$, $\nabla_x S_b(\cdot,y) \in L^p$ and
\begin{equation*}
\norm{\nabla_x S_b(\cdot,y)}_{L^p} \leq C_{b,p}(1+|y|).
\end{equation*}

\item There exists $s_0 \in (0,1)$ such that for all $0 < s < s_0$,
\begin{align*}
|\nabla_x S_b(x,\cdot)|_{\mathcal{C}^{0,s}(B(y,1))} &\leq C_{b,s}(1+|y|) \\
|\nabla_x S_b(\cdot,y)|_{\mathcal{C}^{0,s}(\mathbb{R}^2)} &\leq C_{b,s}(1+|y|).
\end{align*}

\end{enumerate}

\end{lemma}

\begin{proof}
For any $p$ such that $1 \leq p < + \infty$, we have
\begin{equation*}
\norm{\sqrt{b(y)}\jbracket{\cdot}\Delta\left(\frac{1}{\sqrt{b}}\right)g(\cdot-y)}_{L^p}
\leq \norm{b}_{L^\infty}^\frac{1}{2}\norm{g(\cdot-y)\jbracket{\cdot}\Delta\left(\frac{1}{\sqrt{b}}\right)}_{L^p}
\end{equation*}
and
\begin{align*}
&\norm{\jbracket{\cdot} g(\cdot-y)\Delta\left(\frac{1}{\sqrt{b}}\right)}^p_{L^p} \\
\leq& \int_{B(y,1)} \jbracket{x}^p|g(x-y)|^p\left|\Delta\left(\frac{1}{\sqrt{b}}\right)(x)\right|^p  \dd x \\
&+ \int_{B(y,1)^c} \jbracket{x}^p|g(x-y)|^p\left|\Delta\left(\frac{1}{\sqrt{b}}\right)(x)\right|^p  \dd x \\
\leq& C\norm{g}^p_{L^p(B(0,1))}\norm{\jbracket{\cdot} \Delta\left(\frac{1}{\sqrt{b}}\right)}^p_{L^\infty} \\
&+ \int_{B(y,1)^c} (1+|x|^2)^\frac{p}{2}(|x|+|y|)^p\left|\Delta\left(\frac{1}{\sqrt{b}}\right)(x)\right|^p  \dd x.
\end{align*}
By Assumption~\ref{assumption_b}, we have that
\begin{align*}
\int_{B(y,1)^c} (1+|x|^2)^\frac{p}{2}(|x|+|y|)^p&\left|\Delta\left(\frac{1}{\sqrt{b}}\right)(x)\right|^p  \dd x \\
&\leq  \int_{\mathbb{R}^2} \frac{(1+|x|^2)^\frac{p}{2}(|x|+|y|)^p}{(1+|x|)^{(4+\gamma)p}}\dd x \\
&\leq C_b(1+|y|)^p.
\end{align*}
Therefore we can apply Proposition~\ref{proposition_well_posedness_elliptic} to show that there exists a solution $S_b(\cdot,y)$ of \eqref{definition_S_b} in $W^{2,p}_{-1}(\mathbb{R}^2)$, unique up to a constant. Since $\jbracket{x} \geq 1$ we also have that
\begin{align*}
\norm{\sqrt{b(y)}g(\cdot-y)\Delta\left(\frac{1}{\sqrt{b}}\right)}_{L^p} \leq C_{b,p}(1+|y|).
\end{align*}
By Lemma~\ref{lemma_link_solutions_duerinckx}, there exists $p_0$ such that for any $2 < p \leq p_0$ and $0 < s < 1$:
\begin{align*}
\norm{\nabla_x S_b(\cdot,y)}_{L^p} &\leq \norm{\sqrt{b(y)}\Delta\left(\frac{1}{\sqrt{b}}\right)g(\cdot-y)}_{L^\frac{2p}{p+2}} \\
&\leq  C_{b,p}(1+|y|)
\end{align*}
and
\begin{align*}
|\nabla_x S_b(\cdot,y)|_{\mathcal{C}^{0,s}} &\leq C_s \norm{\sqrt{b(y)}\Delta\left(\frac{1}{\sqrt{b}}\right)g(\cdot-y)}_{L^{\frac{2}{1-s}}} \\
&\leq C_{b,s}(1+|y|)
\end{align*}
that is the second inequality of Claim (2). Using that
\begin{equation*}
\norm{\cdot}_{L^\infty} \leq C(\norm{\cdot}_{L^p} + |\cdot|_{\mathcal{C}^{0,s}})
\end{equation*}
(see for example the proof of Morrey's embedding theorem in \cite[Theorem~9.12]{Brezis}), we get the bound we want on $\nabla_x S_b$:
\begin{align*}
\norm{\nabla_x S_b(\cdot,y)}_{L^\infty} \leq C_b(1+|y|).
\end{align*}
If we interpolate the inequalities on $\norm{\nabla_x S_b(\cdot,y)}_{L^\infty}$ and $\norm{\nabla_x S_b(\cdot,y)}_{L^p}$ for $2 < p \leq p_0$ we find that for any $p > 2$,
\begin{align*}
\norm{\nabla_x S_b(\cdot,y)}_{L^p} \leq C_{b,p}(1+|y|).
\end{align*}

For the first inequality of Claim (2), let us consider $z$ such that $|z|$ is small and remark that $S_b(x,y+z) - S_b(x,y)$ solves
\begin{align*}
\div&\left(\frac{1}{b}(\nabla_x S_b(\cdot,y+z) -\nabla_x S_b(\cdot,y))\right) \\
&= \left(\sqrt{b(y+z)}g(y+z-\cdot) - \sqrt{b(y)}g(y-\cdot)\right)\Delta\left(\frac{1}{\sqrt{b}}\right).
\end{align*}
Let us find a bound for the second member in $L^p$:
\begin{align*}
\bigg(&\sqrt{b(y+z)}g(y+z-x) - \sqrt{b(y)}g(y-x)\bigg)\Delta\left(\frac{1}{\sqrt{b}}\right)(x) \\
=& (\sqrt{b(y+z)} - \sqrt{b(y)})g(y-x) \Delta\left(\frac{1}{\sqrt{b}}\right)(x) \\
&+ \sqrt{b(y+z)}(g(y+z-x) - g(y-x))\Delta\left(\frac{1}{\sqrt{b}}\right)(x).
\end{align*}
For the first term,
\begin{align*}
\left|(\sqrt{b(y+z)} - \sqrt{b(y)})g(y-x) \Delta\left(\frac{1}{\sqrt{b}}\right)(x)\right|
&\leq C_b|z|\left|g(y-x)\Delta\left(\frac{1}{\sqrt{b}}\right)\right|
\end{align*}
and we can bound its $L^p$ norms by $C_b(1+|y|)|z|$ as in the proof of Claim $(1)$. For the second term,
\begin{align*}
\int_{\mathbb{R}^2}& \left|\sqrt{b(y+z)}(g(y+z-x) - g(y-x))\Delta\left(\frac{1}{\sqrt{b}}\right)(x)\right|^p \dd x \\
\leq& C_b \int_{\mathbb{R}^2} \left|(g(x+z) - g(x))\Delta\left(\frac{1}{\sqrt{b}}\right)(y-x)\right|^p \dd x \\
\leq& C_b\int_{B(0,|z|^\alpha)}|g(x+z)-g(x)|^p\dd x \\
&+ C_b\int_{B(0,|z|^\alpha)^c}|g(x+z)-g(x)|^p\left|\Delta\left(\frac{1}{\sqrt{b}}\right)(y-x)\right|^p\dd x
\end{align*}
for any $0 < \alpha < 1$. Now, if $|z|$ is small enough,
\begin{align*}
\int_{B(0,|z|^\alpha)}|g(x+z)-g(x)|^p\dd x 
&\leq C\int_{B(0,|z|^\alpha)}g(x+z)^p+g(x)^p\dd x. 
\end{align*}
Now we use a classical rearrangement procedure to bound
\begin{align*}
\int_{B(0,|z|^\alpha)}g(x+z)^p &- \int_{B(0,|z|^\alpha)}g(x)^p\dd x \\
=& \int_{B(z,|z|^\alpha)}g(x)^p - \int_{B(0,|z|^\alpha)}g(x)^p\dd x \\
=& \int_{B(0,|z|^\alpha)}g(x)^p(\mathbf{1}_{B(z,|z|^\alpha)}(x) - 1)\dd x \\
&+ \int_{B(0,|z|^\alpha)^c\cap B(z,|z|^\alpha)} g(x)^p \dd x
\end{align*}
Now remark that for $x \in B(0,|z|^\alpha)$, $\displaystyle{g(x)^p \geq -\frac{1}{2\pi}\ln^p(|z|^\alpha)}$ and therefore
\begin{align*}
\int_{B(0,|z|^\alpha)}&g(x)^p(\mathbf{1}_{B(z,|z|^\alpha)}(x) - 1)\dd x \\ 
&\leq -\frac{1}{2\pi}\ln^p(|z|^\alpha)\int_{B(0,|z|^\alpha)}(\mathbf{1}_{B(z,|z|^\alpha)}(x) - 1)\dd x \\
&\leq -\frac{1}{2\pi}\ln^p(|z|^\alpha)(|B(0,|z|^\alpha)\cap B(z,|z|^\alpha)| - |B(0,|z|^\alpha)|) 
\end{align*}
and on $B(0,|z|^\alpha)^c$, $g(x) \leq -\frac{1}{2\pi}\ln(|z|^\alpha)$ so
\begin{align*}
\int_{B(0,|z|^\alpha)^c\cap B(z,|z|^\alpha)} g(x)^p \dd x \leq -\frac{1}{2\pi}\ln^p(|z|^\alpha)|B(0,|z|^\alpha)^c\cap B(z,|z|^\alpha)|.
\end{align*}
We get
\begin{align*}
\int_{B(0,|z|^\alpha)}g(x+z)^p - \int_{B(0,|z|^\alpha)}g(x)^p\dd x \leq 0
\end{align*}
and therefore
\begin{align*}
\int_{B(0,|z|^\alpha)}|g(x+z)-g(x)|^p\dd x 
&\leq 2\int_{B(0,|z|^\alpha)}g(x)^p\dd x \\
&\leq C|z|^{2\alpha}\int_{B(0,1)}g(|z|^{\alpha}y)^p\dd y \\
&\leq C|z|^{2\alpha}\int_{B(0,1)}(\alpha g(z) + g(y))^p \dd y \\
&\leq C_b|z|^{2\alpha}g(z)^p.
\end{align*}
Now if $|z|$ is small enough,
\begin{align*}
C_b&\int_{B(0,|z|^\alpha)^c}|g(x+z)-g(x)|^p\dd x \left|\Delta\left(\frac{1}{\sqrt{b}}\right)(y-x)\right|^p \\
&\leq C_b\bigg(|z|\frac{C}{|z|^\alpha}\bigg)^p\int_{\mathbb{R}^2} \left|\Delta\left(\frac{1}{\sqrt{b}}\right)(y-x)\right|^p \dd x  \\
&\leq C_b|z|^{p(1-\alpha)}
\end{align*}
by Assumption~\ref{assumption_b}. Finally, using Lemma~\ref{lemma_link_solutions_duerinckx} as for the first claim, we get that for any $0 < \alpha < 1$ and some $p > 2$,
\begin{align*}
|\nabla_x S_b(x,y+z) - \nabla_x S_b(x,y)| \leq C_b(1+|y|)|z| + C_b(|z|^{\frac{2\alpha}{p}}g(z) + |z|^{1-\alpha}).
\end{align*}
Dividing both sides by $|z|^s$ for $s$ small enough proves the first inequality of Claim $(2)$.
\end{proof}

With this lemma we are now able to construct the lake kernel. The construction is similar to the one established in \cite[Proposition~3.1]{DekeyserVanSchaftingen} for bounded domains.

\begin{proposition}\label{proposition_bs_law}
There exists a symmetric solution $S_b$ of Equation \eqref{definition_S_b} such that $S_b(0,0) = 0$. We define $g_b$ as
\begin{equation}\label{definition_g_b}
g_b(x,y) := \sqrt{b(x)b(y)}g(x-y) + S_b(x,y).
\end{equation}
Let $\omega \in L^\infty$ with compact support. We define
\begin{equation}\label{existence_noyau_green}
G_b[\omega](x) = \int_{\mathbb{R}^2} g_b(x,y)\dd\omega(y).
\end{equation}
Then $G_b[\omega]$ is a distributional solution of \eqref{elliptic_problem_stream}.

Moreover for $2 < p <+\infty$, $G_b[\omega]$ is the unique solution (up to a constant) of \eqref{elliptic_problem_stream} in $W^{2,p}_{-1}(\mathbb{R}^2)$ given by Proposition~\ref{proposition_well_posedness_elliptic}.
\end{proposition}

\begin{proof}[Proof of Proposition~\ref{proposition_bs_law}]
Let us first define
\begin{equation*}
g_b(x,y) := \sqrt{b(x)b(y)}g(x-y) + S_b(x,y)
\end{equation*}
where $S_b$ is a solution of Equation \eqref{definition_S_b} given by Proposition~\ref{proposition_well_posedness_elliptic} (not necessarily symmetric). Then we have the following result:

\begin{claim}\label{g_b_is_a_kernel_in_x}
If $\phi$ is smooth with compact support, then
\begin{equation*}
-\int_{\mathbb{R}^2} g_b(x,y)\div\left(\frac{1}{b}\nabla \phi\right)(x)\dd x = \phi(y).
\end{equation*}
\end{claim}

\begin{proof}[Proof of the Claim]
We have
\begin{align*}
-\int_{\mathbb{R}^2} &g_b(x,y)\div\left(\frac{1}{b}\nabla \phi\right)(x)\dd x \\
=& -\int_{\mathbb{R}^2} \sqrt{b(x)b(y)}g(x-y)\div\left(\frac{1}{b}\nabla \phi\right)(x)\dd x \\ 
&- \int_{\mathbb{R}^2} S_b(x,y)\div\left(\frac{1}{b}\nabla \phi\right)(x)\dd x \\
=:& T_1 + T_2.
\end{align*}
We have
\begin{align*}
T_1 =& - \sqrt{b(y)}\int_{\mathbb{R}^2}\sqrt{b(x)}g(x-y)\div\left(\frac{1}{b}\nabla \phi\right)(x)\dd x \\
=& \sqrt{b(y)}\int_{\mathbb{R}^2} g(x-y)\frac{1}{2b(x)\sqrt{b(x)}}\nabla b(x)\cdot \nabla \phi(x) \dd x \\
&+ \sqrt{b(y)}\int_{\mathbb{R}^2}\frac{1}{\sqrt{b(x)}}\nabla g(x-y)\cdot \nabla \phi (x) \dd x \\
=:& L_1 + L_2.
\end{align*}
Integrating by parts in the first integral we get
\begin{align*}
L_1 
=&- \sqrt{b(y)}\int_{\mathbb{R}^2} \phi(x) \frac{1}{2b(x)\sqrt{b(x)}}\nabla g(x-y)\cdot \nabla b(x) \dd x \\
&- \sqrt{b(y)}\int_{\mathbb{R}^2} \phi(x) g(x-y) \div\left(\frac{1}{2b\sqrt{b}}\nabla b\right)(x)\dd x.
\end{align*}
For $L_2$, we use 
\begin{equation*}
\nabla\left(\frac{1}{\sqrt{b}}\phi\right) = \frac{1}{\sqrt{b}}\nabla \phi - \phi\frac{1}{2b\sqrt{b}}\nabla b
\end{equation*}
to get
\begin{align*}
L_2
=& \sqrt{b(y)} \int_{\mathbb{R}^2} \phi(x) \frac{1}{2b(x)\sqrt{b(x)}}\nabla b(x) \cdot \nabla g(x-y) \dd x \\
&+ \sqrt{b(y)}\int_{\mathbb{R}^2} \nabla\left(\frac{1}{\sqrt{b(x)}}\phi(x)\right)\cdot\nabla g(x-y)\dd x \\
=&  \sqrt{b(y)} \int_{\mathbb{R}^2} \phi(x) \frac{1}{2b(x)\sqrt{b(x)}}\nabla b(x) \cdot \nabla g(x-y) \dd x \\
&+ \phi(y)
\end{align*}
since $-\Delta_x g(x-y) = \delta_y$ distributionally.
Now let us compute $T_2$:
\begin{align*}
T_2 &= -\int_{\mathbb{R}^2} S_b(x,y) \div\left(\frac{1}{b}\nabla \phi\right)(x)\dd x \\
&= -\int_{\mathbb{R}^2} \div\left(\frac{1}{b}\nabla_x S_b(\cdot,y)\right)(x)\phi(x)\dd x \\
&= -\sqrt{b(y)}\int_{\mathbb{R}^2} g(x-y)\Delta\left(\frac{1}{\sqrt{b}}\right)(x) \phi(x) \dd x
\end{align*}
where we used that $S_b$ is a solution of \eqref{definition_S_b} in the last line. Now just remark that
\begin{equation*}
\Delta\left(\frac{1}{\sqrt{b}}\right) = -\div\left(\frac{1}{2b\sqrt{b}}\nabla b\right)
\end{equation*}
and thus adding $L_1$ and $L_2$ we get
\begin{equation*}
-\int_{\mathbb{R}^2} g_b(x,y)\div\left(\frac{1}{b}\nabla \phi\right)(x)\dd x = \phi(y)
\end{equation*}
and we get the proof of Claim~\ref{g_b_is_a_kernel_in_x}.
\end{proof}

Now let $\omega \in L^\infty(\mathbb{R}^2)$ with compact support. We have
\begin{align*}
-\int_{\mathbb{R}^2}&\left(\int_{\mathbb{R}^2} g_b(x,y)\omega(y) \dd y\right) \div\left(\frac{1}{b}\nabla\phi\right)(x) \dd x \\
&= -\int_{\mathbb{R}^2} \left(\int_{\mathbb{R}^2} g_b(x,y) \div\left(\frac{1}{b}\nabla\phi\right)(x) \dd x\right) \omega(y) \dd y\\
&= \int_{\mathbb{R}^2} \phi(y) \omega(y) \dd y
\end{align*}
where we used Claim~\ref{g_b_is_a_kernel_in_x} in the last equality. Therefore $G_b[\omega]$ is a distributional solution of \eqref{elliptic_problem_stream}.

Now we prove that with this kernel we recover solutions in the sense of Proposition~\ref{proposition_well_posedness_elliptic}:
\begin{claim}\label{link_solutions_elliptic_kernel}
Let $\omega \in L^\infty$ with compact support, then for all $p \in (2,+\infty)$, we have that $\nabla G_b[\omega] \in L^p$. Moreover if $\psi$ is the solution of \eqref{elliptic_problem_stream} given by Proposition \eqref{proposition_well_posedness_elliptic}, then $\psi = G_b[\omega]$ up to a constant.
\end{claim}

\begin{proof}[Proof of the claim]
We have:
\begin{align*}
\nabla G_b[\omega](x) 
=& \int_{\mathbb{R}^2} \frac{\nabla b(x)}{2\sqrt{b(x)}}\sqrt{b(y)}g(x-y)\omega(y)\dd y\\
&+ \int_{\mathbb{R}^2} \sqrt{b(x)b(y)}\nabla g(x-y)\omega(y) \dd y\\
&+ \int_{\mathbb{R}^2} \nabla_x S_b(x,y) \omega(y)\dd y\\
=:& T_1 + T_2 + T_3.
\end{align*}
Now,
\begin{align*}
|T_1| \leq& C_b|\nabla b(x)|\bigg(\int_{B(x,1)} |(\ln|x-y|)\omega(y)|\dd y \\
&+ \int_{\supp(\omega)\backslash B(x,1)} (|x|+|y|)|\omega(y)|\dd y \bigg) \\
\leq& C_{b}\norm{\omega}_{L^1((1+|x|)\dd x)\cap L^\infty}(1+|x|)^{-(3+\gamma)}
\end{align*}
by Assumption~\ref{assumption_b}. Hence $T_1 \in L^p$. For the second term, we have
\begin{align*}
T_2 = \sqrt{b(x)}\nabla g\ast(\sqrt{b}\omega)
\end{align*}
and therefore $T_2 \in L^p$ by Hardy-Littlewood-Sobolev inequality (see for example \cite[Theorem~1.7]{BahouriCheminDanchin}). For the third term,
\begin{align*}
|T_3| \leq \left(\int_{\mathbb{R}^2} |\omega|\right)\int_{\mathbb{R}^2} |\nabla_x S_b(x,y)|\frac{|\omega(y)|\dd y}{\int |\omega|}
\end{align*}
and thus by Jensen inequality
\begin{align*}
\norm{T_3}^p_{L^p} \leq \left(\int_{\mathbb{R}^2} |\omega|\right)^{p-1} \iint_{\mathbb{R}^2\times\mathbb{R}^2} |\nabla_x S_b(x,y)|^p |\omega(y)| \dd y\dd x.
\end{align*}
We have that
\begin{equation*}
\norm{\nabla S_b(\cdot,y)}_{L^p} \leq C_b(1+|y|)
\end{equation*}
by Claim (1) of Lemma~\ref{estimees_S_b_not_necessarily_symmetric}. Therefore
\begin{equation*}
\norm{T_3}^p_{L^p} \leq C_b \left(\int_{\mathbb{R}^2} |\omega|\right)^{p-1} \int_{\mathbb{R}^2}(1+|y|)^p|\omega(y)|\dd y
\end{equation*}
and it follows that $\nabla G_b[\omega] \in L^p$. By Proposition~\ref{proposition_well_posedness_elliptic} we  get that $G_b[\omega] = \psi$ up to a constant.
\end{proof}

We are only left to justify that there exists a symmetric solution of \eqref{definition_S_b}. Consider $\omega_1,\omega_2$ two smooth functions with average zero, then by Claim~\ref{link_solutions_elliptic_kernel}, we have
\begin{align*}
\iint_{\mathbb{R}^2\times\mathbb{R}^2} g_b(x,y) \omega_1(x) \omega_2(y) \dd x \dd y 
&= \int_{\mathbb{R}^2} (\psi_2(x) +C)\omega_1(x) \dd x \\
&= - \int_{\mathbb{R}^2} \psi_2(x)\div\left(\frac{1}{b}\nabla \psi_1\right)(x)\dd x
\end{align*}
where $\psi_i$ is the solution of 
\begin{align*}
-\div\left(\frac{1}{b}\nabla \psi_i\right) = \omega_i
\end{align*} 
given by Proposition~\ref{proposition_well_posedness_elliptic}. If $R > 0$, we have that
\begin{align*}
- \int_{B(0,R)} \psi_2(x)\div\left(\frac{1}{b}\nabla \psi_1\right)(x)\dd x
=& - \int_{\partial B(0,R)}\frac{1}{b}\psi_2 \nabla \psi_1\cdot \dd \vec{S} \\
&+ \int_{B(0,R)} \frac{1}{b}\nabla \psi_2 \cdot \nabla \psi_1.
\end{align*}
Using Proposition~\ref{theorem_asymptotic_velocity_field_lake}, we obtain
\begin{equation*}
\left|\int_{\partial B(0,R)}\frac{1}{b}\psi_2 \nabla \psi_1\cdot \dd \vec{S}\right| \leq 2\pi R\norm{b^{-1}}_{L^\infty} C(1+R^\delta)\frac{C}{R^2} \Tend{R}{+\infty} 0
\end{equation*}
and therefore 
\begin{equation*}
\iint_{\mathbb{R}^2\times\mathbb{R}^2} g_b(x,y) \omega_1(x) \omega_2(y) \dd x \dd y = \int_{\mathbb{R}^2} \frac{1}{b}\nabla \psi_2 \cdot \nabla \psi_1
\end{equation*}
which is a symmetric expression of $\psi_1$ and $\psi_2$. It follows that
\begin{align*}
\iint_{\mathbb{R}^2\times\mathbb{R}^2} g_b(x,y) \omega_1(x) \omega_2(y) \dd x \dd y = \iint_{\mathbb{R}^2\times\mathbb{R}^2} g_b(y,x) \omega_1(x) \omega_2(y) \dd x \dd y.
\end{align*}
Since $\sqrt{b(x)b(y)}g(x-y)$ is symmetric we get that
\begin{equation*}
\iint_{\mathbb{R}^2\times\mathbb{R}^2} S_b(x,y) \omega_1(x) \omega_2(y) \dd x \dd y = \iint_{\mathbb{R}^2\times\mathbb{R}^2} S_b(y,x) \omega_1(x) \omega_2(y) \dd x \dd y
\end{equation*}
for any $\omega_1,\omega_2$ smooth with compact suport and average zero. Let us define
\begin{align*}
A(x,y) := S_b(x,y) - S_b(y,x).
\end{align*}
Now we fix $\chi, \omega_1, \omega_2$ smooth functions with compact support such that $\displaystyle{\int_{\mathbb{R}^2} \omega_2 = 0}$ and $\displaystyle{\int_{\mathbb{R}^2}\chi = 1}$. Remark that we no longer assume that $\displaystyle{\int_{\mathbb{R}^2} \omega_1 = 0}$. We define
\begin{equation*}
A_2(x) := \int_{\mathbb{R}^2} A(x,y) \omega_2(y) \dd y.
\end{equation*}
We have
\begin{align*}
\int_{\mathbb{R}^2} A_2 \omega_1 
&= \int_{\mathbb{R}^2} A_2\left(\omega_1 - \left(\int_{\mathbb{R}^2} \omega_1\right)\chi\right) + \left(\int_{\mathbb{R}^2} \omega_1\right)\int_{\mathbb{R}^2} A_2 \chi \\
&= 0 + \left(\int_{\mathbb{R}^2} \omega_1\right) \int_{\mathbb{R}^2} A_2 \chi.
\end{align*}
Thus $A_2$ is constant so for every $x \in \mathbb{R}^2$,  
\begin{align*}
\int_{\mathbb{R}^2} \nabla_x A(x,y) \omega_2(y)\dd y = 0
\end{align*}
for all $\omega_2$ with mean zero and therefore $\nabla_x A(x,y) = U(x)$. It follows that $A(x,y) = c(x) + d(y)$. Since $A(x,y) = - A(y,x)$, we have $d = -c$. Now let us set $\widetilde{S_b}(x,y) := S_b(x,y) + c(y)$. We have:
\begin{align*}
\widetilde{S_b}(x,y) - \widetilde{S_b}(y,x) &= S_b(x,y) - S_b(y,x) + c(y) - c(x) \\
&= c(x) - c(y) + c(y) - c(x) \\
&= 0
\end{align*}
which proves that $\widetilde{S_b}$ a symmetric solution of \eqref{definition_S_b}. Up to adding a constant we can also assume that $\widetilde{S_b}(0,0) = 0$. 
\end{proof}

The symmetry of $S_b$ allows us to obtain more regularity estimates:
\begin{lemma}\label{estimees_S_b_symmetric}
Let $S_b$ be the symmetric solution of Equation \eqref{definition_S_b} given by Proposition~\ref{proposition_bs_law}, then 
\begin{enumerate}

\item $S_b$ is smooth on  $\mathbb{R}^2\times\mathbb{R}^2\backslash \{(x,x) \; ; \; x \in\mathbb{R}^2\}$.

\item $|S_b(x,y)| \leq C_b(1+|x|^2+|y|^2)$.

\end{enumerate}
\end{lemma}

\begin{proof}
For $0 < r < R$, we define $C(y,r,R) := B(y,R)\backslash B(y,r)$. We have that $S_b(\cdot,y)$ is a solution of
\begin{equation*}
\left\{
\begin{aligned}
& \div\left(\frac{1}{b}\nabla S_b(\cdot,y)\right) = g(\cdot-y)\sqrt{b(y)}\Delta \left(\frac{1}{\sqrt{b}}\right) \qquad \text{in} \; C(y,r,R) \\
&  S_b(\cdot,y) =  S_b(\cdot,y) \in \mathcal{C}^{0,s} \qquad \text{in} \; \partial C(y,r,R).
\end{aligned}\right.
\end{equation*}
Thus by elliptic regularity (see for example \cite[Theorem~6.13]{GilbargTrudinger}) we obtain that $S_b(\cdot,y) \in \mathcal{C}^{2,s}(\mathring{C}(y,r,R))$ for all $y \in \mathbb{R}^2$ and $0 < r < R$. By symmetry we get that $S_b$ is $\mathcal{C}^{2,s}$ on 
\begin{equation*}
\mathbb{R}^2\times\mathbb{R}^2\backslash \{(x,x) \; ; \; x \in\mathbb{R}^2\}.
\end{equation*} 
We can iterate the argument by writing the elliptic system satisfied by the derivatives of $S_b$ to show that $S_b$ is smooth on
\begin{equation*}
\mathbb{R}^2\times\mathbb{R}^2\backslash \{(x,x) \; ; \; x \in\mathbb{R}^2\}.
\end{equation*} 

The second claim is just a consequence of Lemma~\ref{estimees_S_b_not_necessarily_symmetric}, since
\begin{align*}
|S_b(x,y)| &\leq |S_b(0,0) - S_b(x,0)| + |S_b(x,0) - S_b(x,y)| \\
&\leq |S_b(0,0) - S_b(x,0)| + |S_b(0,x) - S_b(y,x)| \\
&\leq \norm{\nabla_x S_b(\cdot,0)}_{L^\infty}|x| + \norm{\nabla_x S_b(\cdot,x)}_{L^\infty}|y| \\
&\leq C_b|x| + C_b(1+|x|)|y| \\
&\leq C_b(1+|x|^2 + |y|^2).
\end{align*}
\end{proof}

We finish this subsection by giving a straightforward consequence of Proposition~\ref{theorem_asymptotic_velocity_field_lake} and \cite[Lemma~2.7]{DuerinckxFischer} which will be useful to deal with the regularisation of the dirac mass we will introduce in Subsection~\ref{subsection:24} and use in Sections~\ref{section:5} and \ref{section:6}. 
\begin{lemma}\label{nabla_G_b_bounded_H_minus_one}
$\mu \mapsto \nabla G_b[\mu]$ extends into a bounded operator from $\dot H^ {-1}$ to $L^2$.
\end{lemma}

\begin{proof}
Let $\mu$ be a smooth function with compact support and average zero. By Proposition~\ref{theorem_asymptotic_velocity_field_lake}, $\nabla G_b[\mu] \in L^2$ and therefore it follows by \cite[Lemma~2.7]{DuerinckxFischer} that
\begin{align*}
\norm{\nabla G_b[\mu]}_{L^2} \leq C_b \norm{\mu}_{\dot H^{-1}}
\end{align*}
and the lemma follows from the density of smooth functions with compact support and average zero in $\dot H^{-1}$.
\end{proof}

\subsection{Regularisations of the Coulomb kernel and the dirac mass}\label{subsection:24}

To study our modulated energy we will need to have suitable regularisations of $g$ and of the dirac mass $\delta_y$. For that purpose, let us first define $g^{(\eta)}$ for any $0 < \eta < 1$ as
\begin{equation}\label{definition_g_eta}
g^{(\eta)}(x) :=
\left\{
\begin{aligned}
&-\frac{1}{2\pi}\ln(\eta) \qquad &\text{if} \; |x| \leq \eta \\
&g(x) \qquad &\text{if} \; |x| \geq \eta
\end{aligned}\right.
\end{equation}
and we define $\delta_{y}^{(\eta)}$ as the uniform probability measure on the circle $\partial B(y,\eta)$. We also define 
\begin{equation}\label{definition_delta_tilde_q}
\widetilde{\delta}_y^{(\eta)} := m_b(y,\eta)\frac{\dd \delta_{y}^{(\eta)}}{\sqrt{b}}
\end{equation}
where
\begin{equation}\label{definition_m_b}
m_b(y,\eta) := \left(\int\frac{\dd \delta_y^{(\eta)}}{\sqrt{b}}\right)^{-1}.
\end{equation}
In the following proposition we state several properties related to these regularisations.

\begin{proposition}\label{proposition_regularisation}

For any $0 < \eta < 1$ and $y \in \mathbb{R}^2$, we have
\begin{equation}\label{egalite_convolution_g_eta}
\int g(x-z) \dd \delta_y^{(\eta)}(z) = g^{(\eta)}(x-y)
\end{equation}
and
\begin{equation}\label{estimee_m_b}
|m_b(y,\eta)-\sqrt{b(y)}| \leq C_b\eta.
\end{equation}
\end{proposition}

\begin{proof}
By a change of variable we may assume that $y = 0$. The function
\begin{equation*}
f(x) := \int_{\partial B(0,\eta)} g(x-z) \dd \delta_0^{(\eta)}(z) 
\end{equation*}
is locally bounded and satisfies $\Delta f = -\delta_0^{(\eta)} = \Delta g^{(\eta)}$. Now if $|x| \geq \eta$, we have 

\begin{align*}
\int_{\partial B(0,\eta)} g(x-z)\dd \delta_0^{(\eta)}(z)  - g^{(\eta)}(x) &= \int_{\partial B(0,\eta)}(g(x-z)-g(x))\dd \delta_0^{(\eta)}(z) \\
&= \int_{\partial B(0,\eta)}g\bigg(\frac{x}{|x|}-\frac{z}{|x|}\bigg)\dd \delta_0^{(\eta)}(z) \\
&\Tend{|x|}{\infty} \int_{\partial B(0,\eta)} -\frac
{1}{2\pi}\ln(1) = 0
\end{align*}
by dominated convergence theorem. Thus $f - g^{(\eta)}$ is a harmonic bounded function so it is constant. Since $f(z) = g(\eta) = g^{(\eta)}(z)$ for any $z$ of norm $\eta$, we get that $f = g^{(\eta)}$.

Let us now prove \eqref{estimee_m_b}:
\begin{equation*}
m_b(y,\eta)-\sqrt{b(y)} = m_b(y,\eta)\sqrt{b(y)}\left(\frac{1}{\sqrt{b(y)}}-\int\frac{ \dd \delta_y^{(\eta)}(z)}{\sqrt{b(z)}}\right)
\end{equation*}
and thus
\begin{equation*}
|m_b(y,\eta)-\sqrt{b(y)}|\leq C_b \eta
\end{equation*}
by Assumption~\ref{assumption_b}.
\end{proof}

\section{Point vortices}\label{section:3}

To prove Theorem~\ref{MFL_theorem} we will need to control the evolution of the interaction energy and of the moment of inertia. We recall that the moment of inertia is not conserved for the lake equations, nor for the point vortex system. Due to the self-interactions, the interaction energy $E_N$ is also not conserved.

The following proposition gives bounds on the interaction energy and on the moment of inertia and the global well-posedness of the lake point-vortex system \eqref{equation_pv}. 

\begin{proposition}\label{proposition_pv_well_posed}
Let $T > 0$ and $(q_1^0,...,q_N^0)$ be such that $q_i^0 \neq q_j^0$ if $i \neq j$. There exists a unique smooth solution of \eqref{equation_pv} on $[0,T]$. Moreover, we have the following estimates:
\begin{equation}\label{bound_E_N}
|E_N(t)| \leq e^{C_b(1+\alpha_N)t}(|E_N(0)|+I_N(0)+1)
\end{equation}
\begin{equation}\label{bound_I_N}
I_N(t) \leq e^{C_b(1+\alpha_N)t}(|E_N(0)|+I_N(0)+1).
\end{equation}
We also have similar estimates for the rescaled moment of inertia and for the interaction energy:
\begin{equation}\label{bound_E_N_resc}
|\overline{E_N}(t)| \leq e^{C_b(1+\alpha_N^{-1})t}(|\overline{E_N}(0)| + \overline{I_N}(0)+1)
\end{equation}
\begin{equation}\label{bound_I_N_resc}
\overline{I_N}(t) \leq e^{C_b(1+\alpha_N^{-1})t}(|\overline{E_N}(0)| + \overline{I_N}(0)+1).
\end{equation}
\end{proposition}

\begin{proof}
Since $b$ is regular (see Assumption~\ref{assumption_b}) and $S_b,g,\nabla g$ are regular outside of the diagonal (see Claim (1) of Lemma~\ref{estimees_S_b_symmetric}), System \eqref{equation_pv} is well-posed up to the first collision time by Cauchy-Lipschitz theorem. We will first prove the bounds on $E_N$ and $I_N$ and then deduce that there is no collision between the points (this is the classical strategy to prove that the Euler point vortex system is well-posed when all the vorticities are positive, as explained for example in \cite[Chapter~4.2]{MarchioroPulvirenti}). Let us assume that there is no collision up to some time $T^\ast \leq T$.

We first compute the time derivative of $E_N$. Since $g_b$ is symmetric, we have
\begin{align*}
\dot E_N =& \frac{1}{N^2}\sum_{i=1}^N\bigg(\underset{j\neq i}{\sum_{j=1}^N} \dot q_i \cdot \nabla_x g_b(q_i,q_j) + \dot q_j\nabla_y g_b(q_i,q_j)\bigg) \\
=& \frac{2}{N^2}\sum_{i=1}^N\bigg(-\alpha_N\frac{\nabla^\bot b(q_i)}{b(q_i)} -\frac{1}{Nb(q_i)}\underset{k\neq i}{\sum_{k=1}^N}\nabla_x^\bot g_b(q_i,q_k)\bigg) \cdot\underset{j\neq i}{\sum_{j=1}^N}\nabla_x g_b(q_i,q_j) \\
=& -\frac{2\alpha_N}{N^2}\sum_{i=1}^N\frac{\nabla^\bot b(q_i)}{b(q_i)}\cdot\bigg(\frac{\sqrt{b(q_i)}}{2\sqrt{b(q_j)}}g(q_i - q_j)\nabla b(q_j) \\
&+\sqrt{b(q_i)b(q_j)}\nabla g(q_i - q_j) +  \nabla_x S_b(q_i,q_j)\bigg)
\end{align*}
and thus we get that
\begin{equation}\label{derivative_E_N}
\dot E_N = -\frac{2\alpha_N}{N^2}\sum_{i=1}^N\frac{\nabla^\bot b(q_i)}{b(q_i)}
\cdot \bigg(\underset{j\neq i}{\sum_{j=1}^N}\sqrt{b(q_i)b(q_j)}\nabla g(q_i - q_j) + \nabla_x S_b(q_i,q_j)\bigg).
\end{equation}

Now let us bound the right-handside of the upper equality. Using Claim $(1)$ of Lemma~\ref{estimees_S_b_not_necessarily_symmetric} and Assumption~\ref{assumption_b}, we have
\begin{align*}
\left|\frac{\nabla^\bot b(q_i)}{b(q_i)}\cdot \nabla_x S_b(q_i,q_j)\right| &\leq C_b(1+|q_j|)
\end{align*}
and thus
\begin{equation}\label{bound_E_N_1}
\left|\frac{2\alpha_N}{N^2}\sum_{i=1}^N\frac{\nabla^\bot b(q_i)}{b(q_i)}
\cdot \underset{j\neq i}{\sum_{j=1}^N} \nabla_x S_b(q_i,q_j)\right| \leq C_b\alpha_N(1+I_N).
\end{equation}
Now remark that
\begin{align*}
\frac{2\alpha_N}{N^2}&\sum_{i=1}^N\frac{\nabla^\bot b(q_i)}{b(q_i)}
\cdot \bigg(\underset{j\neq i}{\sum_{j=1}^N}\sqrt{b(q_i)b(q_j)}\nabla g(q_i - q_j)\bigg) \\
&= \frac{\alpha_N}{N^2}\sum_{i=1}^N\underset{j\neq i}{\sum_{j=1}^N}\left(\sqrt{\frac{b(q_j)}{b(q_i)}}\nabla^\bot b(q_i) - \sqrt{\frac{b(q_i)}{b(q_j)}}\nabla^\bot b(q_j)\right)\cdot \nabla g(q_i - q_j).
\end{align*}
Moreover,
\begin{align*}
\sqrt{\frac{b(q_j)}{b(q_i)}}\nabla^\bot b(q_i) - \sqrt{\frac{b(q_i)}{b(q_j)}}\nabla^\bot b(q_j) =& \sqrt{\frac{b(q_j)}{b(q_i)}}(\nabla^\bot b(q_i) - \nabla^\bot b(q_j)) \\
&+ \frac{b(q_j) - b(q_i)}{\sqrt{b(q_i)b(q_j)}}\nabla^\bot b(q_j)
\end{align*}
and thus using the Lipschitz regularity of $b$ and $\nabla b$ (see Assumption~\ref{assumption_b}) and $|\nabla g(q_i-q_j)| = C|q_i-q_j|^{-1}$ we get that
\begin{equation}\label{bound_E_N_2}
\left|\frac{2\alpha_N}{N^2}\sum_{i=1}^N\frac{\nabla^\bot b(q_i)}{b(q_i)}
\cdot \bigg[\underset{j\neq i}{\sum_{j=1}^N}\sqrt{b(q_i)b(q_j)}\nabla g(q_i - q_j)\bigg]\right| \leq C_b\alpha_N.
\end{equation}
Combining inequalities \eqref{bound_E_N_1} and \eqref{bound_E_N_2} we get that
\begin{equation}\label{bound_E_N_dot}
|\dot E_N| \leq C_b(1+I_N)\alpha_N.
\end{equation}
Now we compute the time derivative of $I_N$:
\begin{align*}
\dot I_N
=& \frac{2}{N}\sum_{i=1}^N q_i \cdot \dot q_i \\
=& -\frac{2\alpha_N}{N}\sum_{i=1}^N q_i \cdot \frac{\nabla^\bot b(q_i)}{b(q_i)} \\
&-\frac{2}{N}\sum_{i=1}^N\underset{j\neq i}{\sum_{j=1}^N}\frac{\sqrt{b(q_j)}}{2b(q_i)\sqrt{b(q_i)}}g(q_i-q_j)q_i \cdot\nabla^\bot b(q_i) \\
&-\frac{2}{N}\sum_{i=1}^N\underset{j\neq i}{\sum_{j=1}^N}\frac{\sqrt{b(q_i)b(q_j)}}{b(q_i)}q_i \cdot\nabla^\bot g(q_i - q_j) \\
&-\frac{2}{N}\sum_{i=1}^N\underset{j\neq i}{\sum_{j=1}^N} q_i \cdot\nabla_x^\bot S_b(q_i,q_j) \\
=:& 2(T_1 + T_2 + T_3 + T_4).
\end{align*}
Using Assumption~\ref{assumption_b} we have
\begin{equation}\label{bound_IN1}
|T_1| \leq C_b\alpha_N.
\end{equation}
For the second term, using Assumption~\ref{assumption_b} we have
\begin{align*}
|T_2| &\leq \frac{C_b}{N^2}\sum_{i=1}^N \bigg(\underset{j\neq i}{\sum_{j=1}^N} |g(q_i - q_j)|\bigg) \\
&\leq \frac{C_b}{N^2}\sum_{i=1}^N\bigg(\underset{j\neq i}{\sum_{j=1}^N} g(q_i-q_j)\mathbf{1}_{|q_i-q_j| \leq 1}+|q_i|^2 + |q_j|^2\bigg)\\
&\leq C_b I_N + \frac{C_b}{N^2}\underset{|q_i - q_j| \leq 1}{\sum_{1 \leq i \neq j \leq N}}g(q_i-q_j).
\end{align*}
Now by Assumption~\ref{assumption_b}, we have that
\begin{align*}
\frac{1}{N^2}\underset{|q_i - q_j| \leq 1}{\sum_{1 \leq i \neq j \leq N}}g(q_i-q_j)
\leq& \frac{C_b}{N^2}\sum_{1 \leq i \neq j \leq N}\bigg(\sqrt{b(q_i)b(q_j)}g(q_i - q_j) \\
&+ S_b(q_i,q_j)\bigg) + \frac{C_b}{N^2}\underset{|q_i - q_j| \geq 1}{\sum_{1 \leq i \neq j \leq N}}|g(q_i-q_j)| \\
&+ \frac{C_b}{N^2}\sum_{1 \leq i \neq j \leq N} |S_b(q_i,q_j)| \\
\leq& C_b\bigg(E_N + \frac{1}{N^2}\underset{|q_i - q_j| \geq 1}{\sum_{1 \leq i \neq j \leq N}}|g(q_i-q_j)| \\
&+ \frac{1}{N^2}\sum_{1 \leq i \neq j \leq N} |S_b(q_i,q_j)|\bigg).
\end{align*}
Moreover,
\begin{align*}
\frac{C_b}{N^2}\underset{|q_i - q_j| \geq 1}{\sum_{1 \leq i \neq j \leq N}}|g(q_i-q_j)|
&\leq \frac{C_b}{N^2}\underset{|q_i - q_j| \geq 1}{\sum_{1 \leq i \neq j \leq N}}|q_i|^2 + |q_j|^2 \\
&\leq C_b I_N
\end{align*}
and using Claim $(2)$ of Lemma~\ref{estimees_S_b_symmetric},
\begin{align*}
\frac{1}{N^2}\sum_{1 \leq i \neq j \leq N} |S_b(q_i,q_j)| &\leq \frac{C_b}{N^2}\sum_{1 \leq i \neq j \leq N}(1+|q_i|^2+|q_j|^2) \leq C_b(1+I_N).
\end{align*}
Therefore
\begin{equation}\label{bound_IN2}
|T_2| \leq C_b(1+|E_N|+I_N).
\end{equation}
For the third term we write
\begin{align*}
T_3 =& -\frac{1}{N^2}\sum_{i=1}^N\underset{j\neq i}{\sum_{j=1}^N}\frac{\sqrt{b(q_j)}-\sqrt{b(q_i)}}{\sqrt{b(q_i)}}\nabla^\bot g(q_i - q_j) \cdot q_i \\
&-\frac{1}{2N^2}\sum_{i=1}^N\underset{j\neq i}{\sum_{j=1}^N}\nabla^\bot g(q_i - q_j) \cdot (q_i-q_j) \\
=& -\frac{1}{N^2}\sum_{i=1}^N\underset{j\neq i}{\sum_{j=1}^N}\frac{\sqrt{b(q_j)}-\sqrt{b(q_i)}}{\sqrt{b(q_i)}}\nabla^\bot g(q_i - q_j) \cdot q_i -0
\end{align*}
and thus using the Lipschitz regularity of $b$ (see Assumption~\ref{assumption_b}) we get
\begin{equation}\label{bound_IN3}
|T_3| \leq C_b(1+I_N).
\end{equation}
For the fourth term, using Claim (1) of Lemma~\ref{estimees_S_b_not_necessarily_symmetric} we get
\begin{equation}\label{bound_IN4}
\begin{aligned}
|T_4| &= \left|-\frac{1}{N^2}\sum_{i=1}^N\underset{j\neq i}{\sum_{j=1}^N} \frac{1}{b(q_i)}q_i\cdot\nabla_x^\bot S_b(q_i,q_j)\right| \\
&\leq C_b\frac{1}{N^2}\sum_{i=1}^N\sum_{j=1}^N |q_i|(1+|q_j|) \\
&\leq C_b(1+I_N).
\end{aligned}
\end{equation}
Combining with inequalities \eqref{bound_IN1}, \eqref{bound_IN2}, \eqref{bound_IN3} and  \eqref{bound_IN4} we get that 
\begin{equation}\label{bound_I_N_dot}
|\dot I_N| \leq C_b(1+ \alpha_N + |I_N| + |E_N|).
\end{equation}
Let us write $U_N := (E_N,I_N)$. By equations \eqref{bound_E_N_dot} and \eqref{bound_I_N_dot} we have
\begin{equation*}
|\dot U_N| \leq C_b(1+\alpha_N)(1+|U_N|)
\end{equation*}
therefore by Grönwall's lemma we have
\begin{equation*}
|U_N(t)| \leq e^{C_b(1+\alpha_N)t}(|U_N(0)|+1)-1
\end{equation*}
from which \eqref{bound_E_N} and \eqref{bound_I_N} follows.

Let us use these bounds to prove that there is no collision (and it will follow that System \eqref{equation_pv} is globally well-posed). If $i \neq j$, then
\begin{align*}
g(|q_i - q_j|) \leq& C_b\bigg(E_N + \frac{1}{N^2}\sum_{1 \leq k \neq l \leq N}{\sum}|S_b(q_k,q_l)| \\ 
&-\frac{1}{N^2}\underset{(k,l) \neq (i,j)}{\sum_{1 \leq k \neq l \leq N}}g(q_k - q_l)\bigg) \\
\leq& C_b\bigg(E_N + \frac{1}{N^2}\sum_{1 \leq k \neq l \leq N}(1+|q_k|^2 + |q_l|^2)\bigg).
\end{align*} 
where we used Claim $(2)$ of Lemma~\ref{estimees_S_b_symmetric} and $\ln|x-y| \leq |x| + |y|$. Thus by inequalities \eqref{bound_E_N} and \eqref{bound_I_N} we get
\begin{align*}
g(|q_i - q_j|) &\leq C_b(e^{C_b(1+\alpha_N)t}(|E_N(0)|+I_N(0)+1)+1)
\end{align*}
and therefore
\begin{equation*}
|q_i(t) - q_j(t)| \geq \exp\bigg(-2\pi C_b(e^{C_b(1+\alpha_N)t}(|E_N(0)|+I_N(0)+1)+1)\bigg) > 0.
\end{equation*}
It follows that there is no collision on $[0,T]$. The bounds on $\overline{E_N}$ and $\overline{I_N}$ follow directly from Inequalities \eqref{bound_E_N} and \eqref{bound_I_N} applied to $t = \alpha_N^{-1}\tau$.
\end{proof}

\section{Time derivatives of the modulated energies}\label{section:4}

The time derivatives of $\mathcal{F}_{b,N}$ and of $\overline{\mathcal{F}}_{b,N}$, defined in \eqref{definition_F_b_N} and \eqref{definition_F_b_N_rescaled}, are given by the two following propositions:
\begin{proposition}\label{time_derivative_F_N}
Let $\omega$ be a weak solution of \eqref{lake_equation_vorticity} in the sense of Definition~\ref{definition_weak_solution}, $(q_1,...,q_N)$ be solutions of \eqref{equation_pv}. We denote
\begin{equation*}
\omega_N = \frac{1}{N}\sum_{i=1}^N \delta_{q_i(t)}.
\end{equation*}

Assume that $\omega$ satisfies Assumption~\ref{assumption_omega}. Then $\mathcal{F}_{b,N}$ is Lipschitz and for almost every $t \in [0,T]$,
\begin{align*}
&\frac{\dd}{\dd t}\mathcal{F}_{b,N}(t)
= \\
&2\iint_{(\mathbb{R}^2\times\mathbb{R}^2)\backslash\Delta} \left(u(t,x) - \alpha\frac{\nabla^\bot b(x)}{b(x)}\right)\cdot \nabla_x g_b(x,y)\dd(\omega(t)-\omega_N(t))^{\otimes 2}(x,y) \\
&+ 2(\alpha_N - \alpha)\iint_{(\mathbb{R}^2\times\mathbb{R}^2)\backslash\Delta} \frac{\nabla^\bot b(x)}{b(x)}\cdot \nabla_x g_b(x,y) \dd \omega_N(t,x) \dd(\omega(t) - \omega_N(t))(y).
\end{align*}
\end{proposition}

\begin{proposition}\label{time_derivative_F_N_resc}
Let $(\overline{q_1},...,\overline{q_N})$ be solutions of \eqref{equation_pv_rescaled} and $\overline{\omega}$ be a solution of \eqref{transport_equation} in the sense of Definition~\ref{definition_weak_solution_transport}.

We denote
\begin{equation*}
\overline{\omega}_N = \frac{1}{N}\sum_{i=1}^N \delta_{\overline{q_i}(t)}.
\end{equation*}

Assume that $\overline{\omega}$ satisfies Assumption~\ref{assumption_omega}. Denote $v = \nabla G_b[\overline{\omega}]$. Then $\overline{\mathcal{F}}_{b,N}$ is Lipschitz and for almost every $t \in [0,T]$, we have
\begin{multline*}
\frac{\dd}{\dd t}\overline{\mathcal{F}}_{b,N}(t) = -2\iint_{(\mathbb{R}^2\times \mathbb{R}^2)\backslash \Delta} \frac{\nabla^\bot b(x)}{b(x)}\cdot \nabla_x g_b(x,y)\dd(\overline{\omega}(t)-\overline{\omega}_N(t))^{\otimes 2}(x,y) \\
+ \frac{2}{N^2\alpha_N} \sum_{i=1}^N\underset{j\neq i}{\sum_{j=1}^N} \frac{v(t,\overline{q_i})}{b(\overline{q_i})}\cdot\nabla_x g_b(\overline{q_i},\overline{q_j}).
\end{multline*}
\end{proposition}

\begin{proof}[Proof of Proposition~\ref{time_derivative_F_N}]
We split $\mathcal{F}_{b,N}$ in three terms:
\begin{align*}
\mathcal{F}_{b,N} =& \iint_{\mathbb{R}^2\times\mathbb{R}^2} g_b(x,y) \omega(t,x)\omega(t,y) \dd x \dd y  \\
&-\frac{2}{N}\sum_{i=1}^N \int_{\mathbb{R}^2} g_b(x,q_i)\omega(t,x)\dd x + E_N \\
=:& T_1 + T_2 + E_N.
\end{align*}

Let us compute the time derivative of $T_1$. For that purpose, we will need to regularize the kernel $g_b$. The regularisation we will use is given by the following Claim:
\begin{claim}\label{regularisation_g_b}
There exists a familly of smooth functions $(g_b^{\eta})_{0 < \eta < 1}$ such that:
\begin{itemize}
\item $|g_b^{\eta}(x,y)| \leq C_b(|g(x-y)|+1+|x|^2+|y|^2)$
\item $|\nabla_x g_b^{\eta}(x,y)|,|\nabla_y g_b^{\eta}(x,y)| \leq C_b(|x-y|^{-1}+1+|x|+|y|)$.
\item For any $(x,y) \in (\mathbb{R}^2)^2$ such that $x \neq y$, 
\begin{align*}
g_b^{\eta}(x,y) &\Tend{\eta}{0} g_b(x,y) \\
\nabla_x g_b^{\eta}(x,y) &\Tend{\eta}{0} \nabla_x g_b(x,y) \\
\nabla_y g_b^{\eta}(x,y) &\Tend{\eta}{0} \nabla_y g_b(x,y).
\end{align*}
\end{itemize}
\end{claim}

\begin{proof}[Proof of the claim]
We define
\begin{equation*}
g_b^\eta(x,y) = \sqrt{b(x)b(y)}g^\eta(x-y) + S_b^\eta(x,y)
\end{equation*}
where $g^\eta$ is a smooth function satisfying:
\begin{itemize}
\item $g^\eta(x) = g(x)$ for $|x| \geq \eta$, 
\item $|g^\eta(x)| \leq |g(x)|$,
\item $|\nabla g^\eta(x)| \leq C|x|^{-1}$.
\end{itemize}
that we can obtain by extending $\displaystyle{\ln\vert_{x\geq \eta}}$ in a smooth function on $\mathbb{R^+}$.
We define $S_b^\eta := S_b\ast\chi_\eta$ where $\chi_\eta$ is a mollifier on $\mathbb{R}^4$. Since $S_b$ is locally Lipschitz (see Lemma~\ref{estimees_S_b_not_necessarily_symmetric}), $S_b^\eta$ is smooth and we get from Claim (1) of Lemma~\ref{estimees_S_b_not_necessarily_symmetric} and Claim (2) of Lemma~\ref{estimees_S_b_symmetric} that
\begin{itemize}
\item $|S_b^\eta(x,y)| \leq C_b(1+|x|^2+|y|^2)$,
\item $|\nabla_x S_b^\eta(x,y)|,|\nabla_y S_b^\eta(x,y)| \leq C_b(1+|x|+|y|)$.
\end{itemize}
Since $S_b$ is locally Lipschitz, $S_b^\eta$ and $\nabla S_b^\eta$ converge locally uniformly to $S_b$ and $\nabla S_b$ (see for example \cite[Proposition~4.21]{Brezis}) and therefore we get the convergence of $g_b^\eta(x,y)$ and $\nabla g_b^\eta(x,y)$ to $g_b(x,y)$ and $\nabla g_b(x,y)$ for any $x \neq y$. 
\end{proof}

With this regularisation we can compute the time derivative of $T_1$:
\begin{claim}\label{expression_derivative_T1}
$T_1 \in W^{1,\infty}([0,T])$ and for almost every $t \in [0,T]$, we have 
\begin{equation*}
\frac{\dd T_1}{\dd t} = 2 \iint_{\mathbb{R}^2\times\mathbb{R}^2} \left(u(t,x)-\alpha\frac{\nabla^\bot b(x)}{b(x)}\right)\cdot\nabla_x g_b(x,y)\omega(t,x)\omega(t,y)\dd x \dd y.
\end{equation*}
\end{claim}

\begin{proof}[Proof of the claim]
For $0 \leq s,t \leq T$ and $0 < \eta < 1$ we have:
\begin{equation*}
T_1(t) - T_1(s) = \iint_{\mathbb{R}^2\times\mathbb{R}^2} g_b(x,y)(\omega(t,x)\omega(t,y) - \omega(s,x)\omega(s,y))\dd x \dd y.
\end{equation*}
Now for almost all $x$ and $y$ such that $x \neq y$,
\begin{multline*}
|g_b^{\eta}(x,y)||\omega(t,x)\omega(t,y) - \omega(s,x)\omega(s,y)| \\
\leq C_b(|g(x-y)|+1+|x|^2+|y|^2) |\omega(t,x)\omega(t,y) - \omega(s,x)\omega(s,y)|
\end{multline*}
and
\begin{multline*}
\iint_{\mathbb{R}^2\times\mathbb{R}^2}(|g(x-y)|+1+|x|^2+|y|^2)\\
\times |\omega(t,x)\omega(t,y) - \omega(s,x)\omega(s,y)|\dd x \dd y < +\infty
\end{multline*}
because $\omega \in L^\infty$ with compact support. Therefore by dominated convergence theorem we get that
\begin{equation}\label{limit_T_1_b_eta}
T_1(t) - T_1(s) = \underset{\eta \rightarrow 0}{\lim}\iint_{\mathbb{R}^2\times\mathbb{R}^2} g_b^\eta(x,y)(\omega(t,x)\omega(t,y) - \omega(s,x)\omega(s,y))\dd x \dd y.
\end{equation}
Since $g_b^\eta$ is smooth and $\omega$ has compact support, we can use \eqref{formulation_faible} to get that
\begin{multline*}
\int_{\mathbb{R}^2} g_b^\eta(x,y)(\omega(t,y)-\omega(s,y))\dd y = \\ \int_s^t \int_{\mathbb{R}^2} \nabla_y g_b^\eta(x,y)\cdot \left(u(\tau,y)-\alpha\frac{\nabla^\bot b(y)}{b(y)}\right)\omega(\tau,y) \dd y \dd \tau.
\end{multline*}
Let us write 
\begin{equation*}
\phi(t,x) := \int_{\mathbb{R}^2} g_b^\eta(x,y)\omega(t,y)\dd y.
\end{equation*}
Since $g_b^\eta$ is smooth we have that for any compact $K \subset \mathbb{R}^2$,
\begin{multline*}
(t,x) \mapsto \\ \int_{\mathbb{R}^2} \nabla_y g_b^\eta(x,y)\cdot\left(u(t,y)-\alpha\frac{\nabla^\bot b(y)}{b(y)}\right) \omega(t,y)\dd y \in L^\infty([0,T],\mathcal{C}^\infty(K))
\end{multline*}
and thus $\phi \in W^{1,\infty}([0,T],\mathcal{C}^\infty(K))$ and for almost every $t \in [0,T]$,
\begin{equation*}
\partial_t \phi(t,x) = \int_{\mathbb{R}^2} \nabla_y g_b^\eta(x,y)\left(u(\tau,y)-\alpha\frac{\nabla^\bot b(y)}{b(y)}\right) \omega(t,y)\dd y.
\end{equation*}
Therefore we can use $\phi$ as a test function in \eqref{formulation_faible} (remark that we defined \eqref{formulation_faible} for smooth functions only but by density we can extend it to functions which are only $W^{1,\infty}$ in time) and we get that
\begin{align*}
&\iint_{\mathbb{R}^2\times\mathbb{R}^2} g_b^\eta(x,y)(\omega(t,x)\omega(t,y) - \omega(s,x)\omega(s,y))\dd x \dd y \\
=& \int_s^t\iint_{\mathbb{R}^2\times\mathbb{R}^2} \nabla_y g_b^\eta(x,y)\cdot \left(u(\tau,y)-\alpha\frac{\nabla^\bot b(y)}{b(y)}\right) \omega(\tau,y)\omega(\tau,x)\dd y \dd x \dd \tau \\
&+\int_s^t\iint_{\mathbb{R}^2\times\mathbb{R}^2} \nabla_x g_b^\eta(x,y)\cdot \left(u(\tau,x)-\alpha\frac{\nabla^\bot b(x)}{b(x)}\right)\omega(\tau,x)\omega(\tau,y)\dd x \dd y \dd \tau.
\end{align*}
Now we have that for almost every $x$ and $y$ such that $x \neq y$ and almost every $\tau \in [0,T]$,
\begin{multline*}
|\nabla_x g_b^\eta(x,y)\cdot \left(u(\tau,x)-\alpha\frac{\nabla^\bot b(x)}{b(x)}\right) \omega(\tau,x)\omega(\tau,y)| \\
\leq C_b (|x-y|^{-1}+1+|x|^2+|y|^2)|\left|u(\tau,x)-\alpha\frac{\nabla^\bot b(x)}{b(x)}\right||\omega(\tau,y)||\omega(\tau,x)|
\end{multline*}
and
\begin{multline*}
\int_s^t\iint_{\mathbb{R}^2\times\mathbb{R}^2} (|x-y|^{-1}+1+|x|^2+|y|^2)\\
\times\left|u(\tau,x)-\alpha\frac{\nabla^\bot b(x)}{b(x)}\right||\omega(\tau,y)||\omega(\tau,x)| \dd x \dd y < +\infty.
\end{multline*}
Applying dominated convergence theorem we find that
\begin{multline*}
\int_s^t\iint_{\mathbb{R}^2\times\mathbb{R}^2} \nabla_x g_b^\eta(x,y)\cdot \left(u(\tau,x)-\alpha\frac{\nabla^\bot b(x)}{b(x)}\right)\omega(\tau,x)\omega(\tau,y)\dd x \dd y \dd \tau \\
\Tend{\eta}{0} \int_s^t\iint_{\mathbb{R}^2\times\mathbb{R}^2} \nabla_x g_b(x,y)\cdot \left(u(\tau,x)-\alpha\frac{\nabla^\bot b(x)}{b(x)}\right)\omega(\tau,x)\omega(\tau,y)\dd x \dd y \dd \tau. 
\end{multline*}
We can do the same for the first term to get that
\begin{multline*}
\int_s^t\iint_{\mathbb{R}^2\times\mathbb{R}^2} \nabla_y g_b^\eta(x,y)\cdot\left(u(\tau,y)-\alpha\frac{\nabla^\bot b(y)}{b(y)}\right) \omega(\tau,y)\omega(\tau,x)\dd y \dd x \dd \tau \\
\Tend{\eta}{0} \int_s^t\iint_{\mathbb{R}^2\times\mathbb{R}^2} \nabla_y g_b(x,y)\cdot \left(u(\tau,y)-\alpha\frac{\nabla^\bot b(y)}{b(y)}\right) \omega(\tau,y)\omega(\tau,x)\dd y \dd x \dd \tau.
\end{multline*}
Using that $\nabla_y g_b(x,y) = \nabla_x g_b(y,x)$ and  \eqref{limit_T_1_b_eta} we get that $T_1 \in W^{1,\infty}([0,T])$ and for almost every $t \in [0,T]$, we get Claim~\ref{expression_derivative_T1}.
\end{proof}

We know by Equation~\ref{derivative_E_N} that
\begin{align*}
\dot E_N &= -\frac{2\alpha_N}{N^2}\sum_{i=1}^N\frac{\nabla^\bot b(q_i)}{b(q_i)}\cdot \underset{j\neq i}{\sum_{j=1}^N} \nabla_x g_b(q_i,q_j)
\end{align*}
and therefore 
\begin{equation}\label{expression_derivative_EN}
\dot E_N = -2 \alpha_N\iint_{\mathbb{R}^2\times \mathbb{R}^2 \backslash \Delta} \frac{\nabla^\bot b(x)}{b(x)}\cdot \nabla_x g_b(x,y) \dd \omega_N(x) \dd \omega_N (y).
\end{equation}
Now we compute the derivative of the second term:
\begin{claim}\label{expression_derivative_T2}
$T_2$ is Lipschitz and for almost every $t \in [0,T]$, we have
\begin{align*}
\frac{\dd }{\dd t}T_2(t)& \\
=& -2\iint_{\mathbb{R}^2\times \mathbb{R}^2} \left(u(t,x)-\alpha\frac{\nabla^\bot b(x)}{b(x)}\right)\cdot \nabla_x g_b(x,y) \omega(t,x)\dd x\dd \omega_N(t,y) \\
&+ 2\alpha_N\iint_{\mathbb{R}^2\times \mathbb{R}^2} \frac{\nabla^\bot b(x)}{b(x)}\cdot \nabla_x g_b(x,y) \dd \omega_N(x)\omega(t,y)\dd y \\
&+ 2\iint_{(\mathbb{R}^2\times \mathbb{R}^2)\backslash \Delta} u(t,x)\cdot \nabla_x g_b(x,y)\dd \omega_N(x) \dd \omega_N(y).
\end{align*}
\end{claim}

\begin{proof}[Proof of the Claim]
If we use the regularisation $g_b^\eta$ we defined in Claim~\ref{regularisation_g_b}, Equation \eqref{formulation_faible} and if we let $\eta$ tends to zero as we did for the proof of Claim~\ref{expression_derivative_T1}, we can show that $T_2$ is Lipschitz and that for almost every $t \in [0,T]$, we have
\begin{align*}
\frac{\dd T_2}{\dd t} = T_{2,1} + T_{2,2}
\end{align*}
where
\begin{equation}\label{expression_derivative_T21}
\begin{aligned}
T_{2,1} &:= -\frac{2}{N}\sum_{i=1}^N \int_{\mathbb{R}^2} \left(u(t,x)-\alpha\frac{\nabla^\bot b(x)}{b(x)}\right)\cdot \nabla_x g_b(x,q_i)\omega(t,x)\dd x \\
&= -2\iint_{\mathbb{R}^2\times \mathbb{R}^2} \left(u(t,x)-\alpha\frac{\nabla^\bot b(x)}{b(x)}\right)\cdot \nabla_x g_b(x,y) \omega(t,x)\dd x\dd \omega_N(t,y)
\end{aligned}
\end{equation}
and
\begin{align*}
T_{2,2} :=&- \frac{2}{N} \sum_{i=1}^N\dot q_i \cdot \int_{\mathbb{R}^2} \nabla_y g_b(x,q_i) \omega(t,x)\dd x \\
=& -\frac{2}{N} \sum_{i=1}^N\dot q_i \cdot \int_{\mathbb{R}^2} \nabla_x g_b(q_i,x)\omega(t,x)\dd x \\
=& \bigg[\frac{2\alpha_N}{N} \sum_{i=1}^N \frac{\nabla^\bot b(q_i)}{b(q_i)}\cdot \int_{\mathbb{R}^2} \nabla_x g_b(q_i,x) \omega(t,x)\dd x\bigg] \\
&+ \bigg[\frac{2}{N^2}\sum_{i=1}^N\underset{j\neq i}{\sum_{j=1}^N}\frac{1}{b(q_i)}\nabla_x^\bot g_b(q_i,q_j)\cdot \int_{\mathbb{R}^2} \nabla_x g_b(q_i,x) \omega(t,x)\dd x \bigg]\\
=:& T_{2,2,1} + T_{2,2,2}.
\end{align*}
Now we have
\begin{equation}\label{expression_derivative_T221}
\begin{aligned}
T_{2,2,1} &= \frac{2\alpha_N}{N} \sum_{i=1}^N \frac{\nabla^\bot b(q_i)}{b(q_i)}\cdot \int_{\mathbb{R}^2} \nabla_x g_b(q_i,x) \omega(t,x)\dd x  \\
&= 2\alpha_N\iint_{\mathbb{R}^2\times\mathbb{R}^2} \frac{\nabla^\bot b(x)}{b(x)}\cdot \nabla_x g_b(x,y) \dd \omega_N(x)\omega(t,y)\dd y 
\end{aligned}
\end{equation}
and using $\displaystyle{\int_{\mathbb{R}^2} \nabla_x g_b(q_i,y)\omega(y)\dd y = b(q_i)u^\bot(t,q_i)}$ (see Proposition~\ref{proposition_bs_law}), we get
\begin{align*}
T_{2,2,2} =&\frac{2}{N^2} \sum_{i=1}^N u(t,q_i)\cdot\underset{j\neq i}{\sum_{j=1}^N}\nabla_x^\bot g_b(q_i,q_j)\\
=& 2\iint_{(\mathbb{R}^2\times \mathbb{R}^2)\backslash \Delta} u(t,x)\cdot \nabla_x g_b(x,y)\dd \omega_N(x) \dd \omega_N(y).
\end{align*}
Combining the upper equality with \eqref{expression_derivative_T21} and\eqref{expression_derivative_T221} we get the proof of Claim \eqref{expression_derivative_T2}.
\end{proof}

Now remark that
\begin{align*}
\iint_{\mathbb{R}^2\times\mathbb{R}^2} u(x)\cdot \nabla_x g_b(x,y)\dd \omega_N(x) \dd \omega(y)
= \int_{\mathbb{R}^2} u \cdot bu^\bot \dd \omega_N = 0.
\end{align*}
Thus combining Claim~\ref{expression_derivative_T1}, Equation \eqref{expression_derivative_EN} and Claim~\ref{expression_derivative_T2} we obtain Proposition~\ref{time_derivative_F_N}.
\end{proof}

We now compute the derivative of the rescaled modulated energy:
\begin{proof}[Proof of Proposition~\ref{time_derivative_F_N_resc}]
We split $\overline{\mathcal{F}}_{b,N}$ in three terms:
\begin{align*}
\overline{\mathcal{F}}_{b,N} =& \iint_{\mathbb{R}^2\times\mathbb{R}^2} g_b(x,y) \overline{\omega}(t,x)\overline{\omega}(t,y) \dd x \dd y  \\
&-\frac{2}{N}\sum_{i=1}^N \int_{\mathbb{R}^2} g_b(x,\overline{q_i})\overline{\omega}(t,x)\dd x  + \overline{E_N}\\
=:& T_1 + T_2 + \overline{E_N}.
\end{align*}
Let us compute the time derivative of the first term. Using the regularisation $g_b^\eta$ we defined in Claim~\ref{regularisation_g_b} and using \eqref{formulation_faible} and letting $\eta$ tends to zero as we did for the proof of Claim~\ref{expression_derivative_T1}, one can show that $T_1$ is Lipschitz and that for almost every $t \in [0,T]$, we have
\begin{equation}\label{expression_derivative_resc_T1}
\begin{aligned}
\frac{\dd T_1}{\dd t} 
&= -2 \iint_{\mathbb{R}^2\times\mathbb{R}^2} \frac{\nabla^\bot b(x)}{b(x)}\cdot\nabla_x g_b(x,y)\overline{\omega}(t,x)\overline{\omega}(t,y)\dd x \dd y.
\end{aligned}
\end{equation}

For the derivative of $\overline{E_N}$ we rescale Equation \eqref{expression_derivative_EN} to get
\begin{equation}\label{expression_derivative_resc_EN}
\frac{\dd}{\dd t}\overline{E_N} = -2\iint_{\mathbb{R}^2\times \mathbb{R}^2 \backslash \Delta}\frac{\nabla^\bot b(x)}{b(x)}\cdot \nabla_x g_b(x,y) \dd \overline{\omega}_N(x) \dd \overline{\omega}_N (y).
\end{equation}
Now let us compute the derivative of the second term:
\begin{claim}
\begin{equation}\label{expression_derivative_resc_T2}
\begin{aligned}
\frac{\dd T_2}{\dd t} 
=&  2\iint_{\mathbb{R}^2\times\mathbb{R}^2} \frac{\nabla^\bot b(x)}{b(x)} \cdot \nabla_x g_b(x,y) \overline{\omega}(t,x)\dd x\dd \overline{\omega}_N(t,y) \\
&+ 2\iint_{\mathbb{R}^2\times\mathbb{R}^2} \frac{\nabla^\bot b(x)}{b(x)}\cdot \nabla_x g_b(x,y)\dd \overline{\omega}_N(x)\dd \overline{\omega}(y) \\
&+ \frac{2}{N^2\alpha_N} \sum_{i=1}^N\underset{j\neq i}{\sum_{j=1}^N} \frac{v(t,\overline{q_i})}{b(\overline{q_i})}\cdot\nabla_x g_b(\overline{q_i},\overline{q_j}).
\end{aligned}
\end{equation}
\end{claim}

\begin{proof}[Proof of \eqref{expression_derivative_resc_T2}]
Using the regularisation $g_b^\eta$ we defined in Claim~\ref{regularisation_g_b} and using \eqref{formulation_faible} and letting $\eta$ tends to zero as we did for the proof of Claim~\ref{expression_derivative_T1}, one can show that $T_2$ is Lipschitz and that for almost every $t \in [0,T]$, we have
\begin{align*}
\frac{\dd T_2}{\dd t} = T_{2,1} + T_{2,2}
\end{align*}
where
\begin{equation*}
\begin{aligned}
T_{2,1} &:= \frac{2}{N}\sum_{i=1}^N \int_{\mathbb{R}^2} \frac{\nabla^\bot b(x)}{b(x)} \cdot \nabla_x g_b(x,q_i)\overline{\omega}(t,x)\dd x \\
&= 2\iint_{\mathbb{R}^2\times\mathbb{R}^2} \frac{\nabla^\bot b(x)}{b(x)} \cdot \nabla_x g_b(x,y) \overline{\omega}(t,x)\dd x\dd \overline{\omega}_N(t,y)
\end{aligned}
\end{equation*}
and
\begin{align*}
T_{2,2} :=&- \frac{2}{N} \sum_{i=1}^N \dot{\overline{q_i}} \cdot \int_{\mathbb{R}^2} \nabla_y g_b(x,q_i) \overline{\omega}(t,x)\dd x \\ 
=&-\frac{2}{N} \sum_{i=1}^N\dot{\overline{q_i}} \cdot v(t,\overline{q_i}) \\
=& \frac{2}{N}\sum_{i=1}^N v(t,\overline{q_i})\cdot\bigg[\frac{\nabla^\bot b(\overline{q_i})}{b(\overline{q_i})} 
+\frac{1}{N\alpha_N}\underset{j\neq i}{\sum_{j=1}^N}\frac{1}{b(\overline{q_i})}\nabla_x g_b(\overline{q_i},\overline{q_j})\bigg] \\
=& 2\iint_{(\mathbb{R}^2\times\mathbb{R}^2)\backslash\Delta} \frac{\nabla^\bot b(x)}{b(x)}\cdot \nabla_x g_b(x,y)\dd \overline{\omega}_N(x)\dd \overline{\omega}(y) \\
&+ \frac{2}{N^2\alpha_N} \sum_{i=1}^N\underset{j\neq i}{\sum_{j=1}^N} \frac{v(t,\overline{q_i})}{b(\overline{q_i})}\cdot\nabla_x g_b(\overline{q_i},\overline{q_j})
\end{align*}
and thus we have \eqref{expression_derivative_resc_T2}.
\end{proof}

Combining Equations \eqref{expression_derivative_resc_T1}, \eqref{expression_derivative_resc_EN} and \eqref{expression_derivative_resc_T2} we get \eqref{time_derivative_F_N_resc}.
\end{proof}

\section{Properties of the modulated energy}\label{section:5}

For $\displaystyle{0 < \eta < 1}$, we denote
\begin{equation*}
H_{N,\eta} := G_b\left[\frac{1}{N}\sum_{i=1}^N \widetilde{\delta}_{q_i}^{(\eta)} - \omega\right].
\end{equation*}
If $b=1$ this quantity is the electric potential introduced by Serfaty in \cite[Equation~(3.12)]{Serfaty} divided by $N$.

\begin{proposition}\label{proposition_pre_coerc}
Let $\omega \in \mathcal{P}(\mathbb{R}^2)\cap L^\infty(\mathbb{R}^2)$ with compact support and $q_1,...,q_N \in \mathbb{R}^2$ be such that $q_i \neq q_j$ if $i \neq j$. Then the following inequality holds:
\begin{multline*}
\int_{\mathbb{R}^2} \frac{1}{b}|\nabla H_{N,\eta}|^2 + \frac{C_b}{N^2}\sum_{1 \leq i \neq j \leq N}(g(q_i - q_j) - g^{(\eta)}(q_i-q_j)) \\
\leq \mathcal{F}_b(Q_N,\omega) + C_b\bigg(\frac{g(\eta)}{N} + I(Q_N)(\eta+N^{-1})
+ \norm{\omega}_{L^1((1+|x|)\dd x)\cap L^\infty}g(\eta)\eta\bigg)
\end{multline*}
where $g^{(\eta)}$ is defined by \eqref{definition_g_eta}.
\end{proposition}

From this proposition we see that even if it is not necessarily positive, the modulated energy is bounded from below by some negative power of $N$ (provided that $(I(Q_N))$ is bounded). We will also prove the three following corollaries:
\begin{corollary}\label{corollary_counting}
If $\omega$ and $Q_N$ satisfy the hypothesis of Proposition~\ref{proposition_pre_coerc}, then there exists $c > 0$ such that
\begin{align*}
\frac{c}{N^2}|\{(q_i,q_j) ; |q_i - q_j| \leq \epsilon\}| \leq& \mathcal{F}_b(Q_N,\omega) + C_b\bigg(\frac{g(\epsilon)}{N} + I(Q_N)(\epsilon+N^{-1}) \\
&+ \norm{\omega}_{L^1((1+|x|)\dd x)\cap L^\infty}g(\epsilon)\epsilon\bigg).
\end{align*}
\end{corollary}

\begin{corollary}\label{corollary_coercivity}
Let $\alpha \in (0,1)$ and $\xi$ be a test function (for example smooth with compact support or in the Schwartz space), then if $\omega$ and $Q_N$ satisfy the hypothesis of Proposition~\ref{proposition_pre_coerc} we have

\begin{align*}
\left|\int_{\mathbb{R}^2} \xi\left(\frac{1}{N}\sum_{i=1}^N \delta_{q_i} - \omega\right)\right| \leq& C_b|\xi|_{\mathcal{C}^{0,\alpha}}N^{-\alpha}
+ C_b \left(\int_{\mathbb{R}^2} \frac{1}{b} |\nabla \xi|^2\right)^\frac{1}{2}\bigg(\mathcal{F}_b(\omega,Q_N) \\
&+ \frac{\ln(N)}{N} + I(Q_N)N^{-1}\\
&+ \norm{\omega}_{L^1((1+|x|)\dd x)\cap L^\infty}\frac{\ln(N)}{N}\bigg)^\frac{1}{2}.
\end{align*}
In particular, there exists $\beta > 0$ such that for all $s < -1$,
\begin{align*}
\norm{\frac{1}{N}\sum_{i=1}^N \delta_{q_i} - \omega}_{H^s} 
\leq& C_b((1+I(Q_N)+\norm{\omega}_{L^1((1+|x|)\dd x)\cap L^\infty})N^{-\beta} \\ &+ \mathcal{F}_b(\omega,Q_N)).
\end{align*}
\end{corollary}

\begin{corollary}\label{corollary_weak_star_cv}
If $\omega$ and $Q_N$ satisfy the hypothesis of Proposition~\ref{proposition_pre_coerc} and if $(I(Q_N))$ is bounded, then the two following assertions are equivalent:
\begin{enumerate}

\item $\mathcal{F}_b(\omega,Q_N) \Tend{N}{+\infty} 0$.

\item $\displaystyle{\frac{1}{N}\sum_{i=1}^N \delta_{q_i}} \wstarcv{N}{+\infty} \omega$ for the weak-$\ast$ topology of probability measures and
\begin{equation*}
\frac{1}{N^2} \sum_{1 \leq i \neq j \leq N} g_b(q_i,q_j) \longrightarrow \iint_{\mathbb{R}^2\times \mathbb{R}^2} g_b(x,y)\omega(x)\omega(y)\dd x \dd y.
\end{equation*}

\end{enumerate}
\end{corollary}

Proposition~\ref{proposition_pre_coerc} and Corollaries~\ref{corollary_counting}, \ref{corollary_coercivity} and \ref{corollary_weak_star_cv} are analogues of other results obtained  in \cite{Duerinckx,NguyenRosenzweigSerfaty,Serfaty}. Proposition~\ref{proposition_pre_coerc} is an equivalent of \cite[Proposition~3.3]{Serfaty} or \cite[Proposition~2.2]{NguyenRosenzweigSerfaty} and the proof will follow the same steps: regularise the modulated energy and control the remainders. Some terms are very similar to the ones obtained in the Coulomb case whereas other terms are specific to the lake kernel and will be handled using the estimates proved in Section~\ref{section:2}.

Corollary~\ref{corollary_counting} is an equivalent of \cite[Corollary~2.3]{NguyenRosenzweigSerfaty} and Corollary~\ref{corollary_coercivity} is an equivalent of \cite[Proposition~3.6]{Serfaty}. Both can be deduced from Proposition~\ref{proposition_pre_coerc} in the same way \cite[Corollary~2.3]{NguyenRosenzweigSerfaty} and \cite[Proposition~3.6]{Serfaty}  are deduced from \cite[Proposition~3.3]{Serfaty} or \cite[Proposition~2.2]{NguyenRosenzweigSerfaty}.

Corollary~\ref{corollary_weak_star_cv} is an equivalent of \cite[Lemma~2.6]{Duerinckx} and its proof proceeds in the same way. Due to the bound we assumed on the moment of inertia, tightness issues will be easier to handle.

Let us begin by proving the main proposition of this section:
\begin{proof}[Proof of Proposition~\ref{proposition_pre_coerc}]
Let us regularise the modulated energy \eqref{definition_modulated_energy} using the regularisation of the dirac mass $\widetilde{\delta}$ defined in \eqref{definition_delta_tilde_q}. We have
\begin{align*}
&\mathcal{F}_b(Q_N,\omega) = \\
&\iint_{\mathbb{R}^2\times \mathbb{R}^2} g_b(x,y)\dd \left(\frac{1}{N}\sum_{i=1}^N \widetilde{\delta}_{q_i}^{(\eta)} - \omega\right)(x)\dd\left(\frac{1}{N}\sum_{i=1}^N \widetilde{\delta}_{q_i}^{(\eta)} - \omega\right)(y) \\
&+\frac{1}{N^2}\underset{1 \leq i \neq j \leq N}{\sum}\iint_{\mathbb{R}^2\times \mathbb{R}^2}  (\sqrt{b(q_i)b(q_j)}g(q_i - q_j) \\
&- \sqrt{b(x)b(y)}g(x-y))\dd \widetilde{\delta}_{q_i}^{(\eta)}(x)\dd \widetilde{\delta}_{q_j}^{(\eta)}(y) \\
&+ \frac{1}{N^2}\underset{1 \leq i \neq j \leq N}{\sum} \iint_{\mathbb{R}^2\times \mathbb{R}^2} (S_b(q_i,q_j) - S_b(x,y))\dd \widetilde{\delta}_{q_i}^{(\eta)}(x)\dd \widetilde{\delta}_{q_j}^{(\eta)}(y) \\
&- \frac{1}{N^2}\sum_{i=1}^N\iint_{\mathbb{R}^2\times \mathbb{R}^2}g_b(x,y)\dd\widetilde{\delta}_{q_i}^{(\eta)}(x)\dd \widetilde{\delta}_{q_i}^{(\eta)}(y) \\
&+ \frac{2}{N}\sum_{i=1}^N \iint_{\mathbb{R}^2\times \mathbb{R}^2} \big(\sqrt{b(x)b(y)}g(x-y) \\
&- \sqrt{b(x)b(q_i)}g(x-q_i)\big)\omega(x)\dd x \dd\widetilde{\delta}_{q_i}^{(\eta)}(y) \\
&+ \frac{2}{N}\sum_{i=1}^N \iint_{\mathbb{R}^2\times \mathbb{R}^2} (S_b(x,y) - S_b(x,q_i))\omega(x)\dd x \dd\widetilde{\delta}_{q_i}^{(\eta)}(y) \\
=:& T_1 + T_2 + T_3 + T_4 + T_5 + T_6.
\end{align*}

\begin{claim}\label{coerc_T1}
We have
\begin{equation*}
T_1 = \int_{\mathbb{R}^2} \frac{1}{b}|\nabla H_{N,\eta}|^2.
\end{equation*}
\end{claim}

\begin{proof}[Proof of the claim]
Let us first fix $\mu$ smooth with compact support and average zero and write $H_\mu = G_b[\mu]$. By Proposition~\ref{proposition_bs_law}, we have
\begin{align*}
\iint_{\mathbb{R}^2\times \mathbb{R}^2} g_b(x,y)\mu(x)\mu(y) \dd x \dd y  &= \int_{\mathbb{R}^2} H_{\mu}(x) \mu(x)\dd x \\
&= - \int_{\mathbb{R}^2} H_{\mu}(x) \div\left(\frac{1}{b}\nabla H_{\mu}\right)(x)\dd x.
\end{align*}
Let $R >0$, then integrating by parts we get 
\begin{align*}
-\int_{B(0,R)} H_{\mu} \div\left(\frac{1}{b}\nabla H_{\mu}\right)
= - \int_{\partial B(0,R)}\frac{1}{b}H_{\mu} \nabla H_{\mu}\cdot \dd \vec{S}
+ \int_{B(0,R)} \frac{1}{b}|\nabla H_{\mu}|^2.
\end{align*}
Using Proposition~\ref{theorem_asymptotic_velocity_field_lake} applied to $\omega = \mu$, $\displaystyle{u = -\frac{1}{b}\nabla^\bot H_\mu}$ and $\psi = H_\mu$, we have
\begin{equation*}
\left|\int_{\partial B(0,R)}\frac{1}{b}H_{\mu} \nabla H_{\mu}\cdot \dd \vec{S}\right| \leq \frac{C}{R^2}(1+R^\delta)R  \Tend{R}{+\infty} 0
\end{equation*}
and therefore
\begin{equation*}
\iint_{\mathbb{R}^2\times \mathbb{R}^2} g_b(x,y) \mu(x) \mu(y) \dd x \dd y = \int_{\mathbb{R}^2} \frac{1}{b}|\nabla H_{\mu}|^2.
\end{equation*}
Now consider a sequence $(\mu_k)$ of smooth functions with compact support and average zero converging to $\displaystyle{m := \frac{1}{N}\sum_{i=1}^N \widetilde{\delta}_{q_i}^{(\eta)} - \omega}$ in $\dot H^{-1}$, then by Lemma~\ref{nabla_G_b_bounded_H_minus_one},
\begin{align*}
\nabla H_{\mu_k} \Tend{k}{+\infty} \nabla H_{N,\eta} \; \text{in} \; L^2.
\end{align*}
and therefore
\begin{align*}
\int_{\mathbb{R}^2} \frac{1}{b}|\nabla H_{\mu_k}|^2 \Tend{k}{+\infty} \int_{\mathbb{R}^2} \frac{1}{b}|\nabla H_{N,\eta}|^2
\end{align*}
and 
\begin{align*}
\bigg|\iint_{\mathbb{R}^2\times \mathbb{R}^2} &g_b(x,y)\mu_k(x) \mu_k(y)\dd x \dd y - \iint_{\mathbb{R}^2\times \mathbb{R}^2} g_b(x,y)\dd m(x) \dd m(y)\bigg| \\
&= \bigg|\int_{\mathbb{R}^2}G_b[\mu_k-m]\dd \mu_k + \int_{\mathbb{R}^2}G_b[m]\dd(\mu_k-m)\bigg| \\
&\leq C\norm{\nabla G_b[\mu_k - m]}_{L^2}\norm{\mu_k}_{\dot H^{-1}} + C\norm{\nabla G_b[m]}_{L^2}\norm{\mu_k-m}_{\dot H^{-1}} \\
&\leq C\norm{\mu_k-m}_{\dot H^{-1}} 
\end{align*}
by Lemma~\ref{nabla_G_b_bounded_H_minus_one} so we get Claim~\ref{coerc_T1}.  
\end{proof}

Now let us bound the fourth term:
\begin{claim}\label{coerc_T4}
\begin{equation*}
|T_4| \leq \frac{C_b}{N}(g(\eta)+I(Q_N)).
\end{equation*}
\end{claim}
\begin{proof}
We write
\begin{align*}
T_4 =& -\frac{1}{N^2}\sum_{i=1}^N\iint_{\mathbb{R}^2\times \mathbb{R}^2}\sqrt{b(x)b(y)}g(x-y)\dd\widetilde{\delta}_{q_i}^{(\eta)}(x)\dd \widetilde{\delta}_{q_i}^{(\eta)}(y) \\
&-\frac{1}{N^2}\sum_{i=1}^N\iint_{\mathbb{R}^2\times \mathbb{R}^2} S_b(x,y)\dd\widetilde{\delta}_{q_i}^{(\eta)}(x)\dd \widetilde{\delta}_{q_i}^{(\eta)}(y) \\
=:& T_{4,1} + T_{4,2}.
\end{align*}
Using the definition of $\widetilde{\delta}_q$ \eqref{definition_delta_tilde_q}  and Equality \eqref{egalite_convolution_g_eta} we get
\begin{align*}
T_{4,1} &= -\frac{1}{N^2}\sum_{i=1}^N  m_b(q_i,\eta)^2 \iint_{\mathbb{R}^2\times \mathbb{R}^2} g(x-y) \dd \delta_{q_i}^{(\eta)}(x)\dd \delta_{q_i}^{(\eta)}(y) \\
&= -\frac{1}{N^2}\sum_{i=1}^Nm_b(q_i,\eta)^2\int_{\mathbb{R}^2} g^{(\eta)}(x-q_i) \dd \delta_{q_i}^{(\eta)}(x).
\end{align*}
Therefore,
\begin{equation*}
|T_{4,1}| \leq \frac{C_bg(\eta)}{N}.
\end{equation*}
Now by Claim $(2)$ of Lemma~\ref{estimees_S_b_symmetric}, we have
\begin{equation*}
|T_{4,2}| \leq \frac{C_b}{N^2}\sum_{i=1}^N(1+|q_i|^2).
\end{equation*}
We get that
\begin{equation*}
|T_4| \leq \frac{C_b}{N}(1+I(Q_N) + g(\eta)) \leq \frac{C_b}{N}(g(\eta)+I(Q_N)).
\end{equation*}
\end{proof}

Now we bound the third and the sixth term:
\begin{claim}\label{coerc_T36}
\begin{equation*}
|T_3| + |T_6| \leq C_b(\norm{\omega}_{L^1((1+|x|)\dd x)}+I(Q_N))\eta.
\end{equation*}
\end{claim}

\begin{proof}
For $x \in \partial B(q_i,\eta), y \in \partial B(q_j,\eta)$, we use Claim $(1)$ of Lemma~\ref{estimees_S_b_not_necessarily_symmetric} and the symmetry of $S_b$ to get  
\begin{align*}
|S_b(q_i,q_j) -  S_b(x,y)| &\leq |S_b(q_i,q_j) - S_b(x,q_j)| + |S_b(x,q_j) - S_b(x,y)| \\
&\leq C_b(1+|q_j|)\eta + C_b(1+|q_i|)\eta \\
&\leq C_b(1+|q_i| + |q_j|)\eta.
\end{align*}
Thus we can bound the third term:
\begin{equation}\label{coerc_T3}
|T_3| \leq C_b(1+I(Q_N))\eta.
\end{equation}
The sixth term can be bounded in the same way:
\begin{equation*}
|T_6| \leq \frac{C_b}{N}\sum_{i=1}^N \iint_{\mathbb{R}^2\times \mathbb{R}^2} (1+|x|+|q_i|)\eta\omega(x) \dd x \dd\widetilde{\delta}_{q_i}^{(\eta)}(y).
\end{equation*}
We get that
\begin{equation}\label{coerc_T6}
|T_6| \leq C_b(\norm{\omega}_{L^1((1+|x|)\dd x)}+I(Q_N))\eta.
\end{equation}
and combining \eqref{coerc_T3} with \eqref{coerc_T6} we get Claim~\ref{coerc_T36}. 
\end{proof}

Now let us bound the fifth term:
\begin{claim}\label{coerc_T5}
\begin{equation*}
|T_5| \leq C_b\norm{\omega}_{L^1\cap L^\infty}\eta g(\eta).
\end{equation*}
\end{claim}

\begin{proof}
Using Proposition~\ref{proposition_regularisation} we write $T_5$ as
\begin{align*}
T_5 =& \frac{2}{N}\sum_{i=1}^N \int_{\mathbb{R}^2}(m_b(q_i,\eta)g^{(\eta)}(x-q_i)-\sqrt{b(q_i)}g(x-q_i))\sqrt{b(x)}\omega(x)\dd x \\
=& \frac{2}{N}\sum_{i=1}^N (m_b(q_i,\eta) - \sqrt{b(q_i)})\int_{\mathbb{R}^2} g^{(\eta)}(x-q_i)\sqrt{b(x)}\omega(x)\dd x \\
&+ \frac{2}{N}\sum_{i=1}^N \sqrt{b(q_i)}\int_{\mathbb{R}^2} (g^{(\eta)}(x-q_i) - g(x-q_i))\sqrt{b(x)}\omega(x)\dd x.
\end{align*}
and thus by \eqref{estimee_m_b} and since $|g^{(\eta)}(x-q_i)| \leq C(g(\eta) + |x| + |q_i|)$ we have
\begin{align*}
|T_5| \leq& C_b\norm{\omega}_{L^1}\eta g(\eta) + C_b\norm{\omega}_{L^1(|x|\dd x)}\eta + C_b\norm{\omega}_{L^1}(1+I(Q_N))\eta  \\
&+C_b\norm{\omega}_{L^\infty}\int_{B(0,\eta)}|g^{(\eta)}(x) - g(x)|\dd x.
\end{align*}
We get that
\begin{equation*}
|T_5| \leq C_b\norm{\omega}_{L^1((1+|x|)\dd x)\cap L^\infty}\eta g(\eta) + (1+I(Q_N))\eta
\end{equation*}
since $\omega$ is a probability density.
\end{proof}

We are only remained to estimate from below the second term:
\begin{claim}\label{coerc_T2}
\begin{equation*}
T_2 \geq \frac{C_b}{N^2}\sum_{1 \leq i \neq j \leq N}(g(q_i - q_j) - g^{(\eta)}(q_i-q_j)) - C_b\eta g(\eta).
\end{equation*}
\end{claim}

\begin{proof}
We also split $T_2$ in two terms:
\begin{align*}
T_2 =& \frac{1}{N^2}\sum_{1 \leq i \neq j \leq N} \sqrt{b(q_i)b(q_j)}g(q_i - q_j) \\
&- m_b(q_i,\eta) m_b(q_j,\eta) \iint_{\mathbb{R}^2\times\mathbb{R}^2} g(x-y) \dd \delta_{q_i}^{(\eta)}(x)\dd \delta_{q_j}^{(\eta)}(y) \\
=& \frac{1}{N^2}\sum_{1 \leq i \neq j \leq N} \sqrt{b(q_i)b(q_j)}g(q_i - q_j) \\
&- m_b(q_i,\eta) m_b(q_j,\eta) \int_{\mathbb{R}^2} g^{(\eta)}(q_i-y) \dd \delta_{q_j}^{(\eta)}(y) \\
=&  \frac{1}{N^2}\sum_{1 \leq i \neq j \leq N}(\sqrt{b(q_i)b(q_j)} - m_b(q_i,\eta) m_b(q_j,\eta)) \\
&\times\int_{\mathbb{R}^2} g^{(\eta)}(q_i-y) \dd \delta_{q_j}^{(\eta)}(y) \\
&+ \frac{1}{N^2}\sum_{1 \leq i \neq j \leq N} \sqrt{b(q_i)b(q_j)}\left(g(q_i - q_j) -\int_{\mathbb{R}^2} g^{(\eta)}(q_i-y) \dd \delta_{q_j}^{(\eta)}(y)\right) \\
=& T_{2,1} + T_{2,2}.
\end{align*}
Writing 
\begin{align*}
\sqrt{b(q_i)b(q_j)} - m_b(q_i,\eta) m_b(q_j,\eta)
&= \sqrt{b(q_i)}(\sqrt{b(q_j)} - m_b(q_j,\eta)) \\
&+ m_b(q_j,\eta)(\sqrt{b(q_i)} - m_b(q_i,\eta))
\end{align*}
and using \eqref{estimee_m_b} we get that
\begin{equation}\label{coerc_T21}
|T_{2,1}| \leq  C_b\eta g(\eta).
\end{equation}
Now by \eqref{egalite_convolution_g_eta},
\begin{align*}
g&(q_i - q_j) -\int_{\mathbb{R}^2} g^{(\eta)}(q_i-y) \dd \delta_{q_j}^{(\eta)}(y) \\
&= g(q_i - q_j) - g^{(\eta)}(q_i-q_j) + \int_{\mathbb{R}^2} (g(q_i - y) - g^{(\eta)}(q_i-y))\dd \delta_{q_j}^{(\eta)}(y) \\
&\geq g(q_i - q_j) - g^{(\eta)}(q_i-q_j) + 0
\end{align*}
and thus
\begin{equation}\label{coerc_T22}
T_{2,2} \geq \frac{C_b}{N^2}\sum_{1 \leq i \neq j \leq N}(g(q_i - q_j) - g^{(\eta)}(q_i-q_j)).
\end{equation}
We get Claim~\ref{coerc_T2} combining Equations \eqref{coerc_T21} with \eqref{coerc_T22}.
\end{proof}

Combining Claims~\ref{coerc_T1}, \ref{coerc_T4}, \ref{coerc_T36}, \ref{coerc_T5} and \ref{coerc_T2} we get the proof of Proposition~\ref{proposition_pre_coerc}. 
\end{proof}

Now we prove the "counting close particles" Corollary:
\begin{proof}[Proof of Corollary~\ref{corollary_counting}]
The proof is exactly the same as the proof of \cite[Lemma~3.7]{Rosenzweig}. If $|q_i - q_j| \leq \epsilon$ then 
\begin{align*}
g(q_i - q_j) - g^{(2\epsilon)}(q_i - q_j) &= -\frac{1}{2\pi}\ln|q_i - q_j| + \frac{1}{2\pi}\ln(2\epsilon) \\
&\geq -\frac{1}{2\pi}\ln(\epsilon) + \frac{1}{2\pi}\ln(2\epsilon) = \frac{1}{2\pi}\ln(2) > 0.
\end{align*}
Thus, since $g - g^{(2\epsilon)} \geq 0$,
\begin{align*}
&\frac{1}{2\pi N^2}\ln(2)|\{(q_i,q_j) ; |q_i - q_j| \leq \epsilon\}| \\ \leq& \frac{1}{N^2}\underset{|q_i - q_j| \leq \epsilon}{\sum_{1 \leq i \neq j\leq N}} (g(q_i - q_j) - g^{(2\epsilon)}(q_i - q_j)) \\
\leq& \frac{1}{N^2}\sum_{1 \leq i \neq j\leq N} (g(q_i - q_j) - g^{(2\epsilon)}(q_i - q_j)) \\
\leq& \mathcal{F}_b(Q_N,\omega) \\
&+ C_b\bigg(\frac{g(\epsilon)}{N} + I(Q_N)(\epsilon+N^{-1}) + \norm{\omega}_{L^1((1+|x|)\dd x)\cap L^\infty}g(\epsilon)\epsilon\bigg).
\end{align*}
where we used Proposition~\ref{proposition_pre_coerc} in the last inequality.
\end{proof}

Now we prove the coercivity result:
\begin{proof}[Proof of Corollary~\ref{corollary_coercivity}]
We have
\begin{align*}
\int_{\mathbb{R}^2} \xi\left(\frac{1}{N}\sum_{i=1}^N \delta_{q_i} - \omega\right)
=& \frac{1}{N}\int_{\mathbb{R}^2} \xi \left(\sum_{i=1}^N\delta_{q_i} - \widetilde{\delta}_{q_i}^{(\eta)}\right) \\
&+ \int_{\mathbb{R}^2} \xi\left(\frac{1}{N}\sum_{i=1}^N \widetilde{\delta}_{q_i}^{(\eta)} - \omega\right) \\
=&: T_1 + T_2.
\end{align*}
Now,
\begin{align*}
T_1 &= \frac{1}{N}\sum_{i=1}^N \xi(q_i) - m_b(q_i,\eta)\int_{\partial B(q_i,\eta)} \frac{\xi(x)}{\sqrt{b(x)}}\dd \delta_{q_i}^{(\eta)}(x) \\
&= \frac{1}{N}\sum_{i=1}^N m_b(q_i,\eta)\int_{\partial B(q_i,\eta)} \frac{\xi(q_i) - \xi(x)}{\sqrt{b(x)}}\dd\delta_{q_i}^{(\eta)}(x).
\end{align*}
Thus
\begin{equation*}
|T_1| \leq C_b|\xi|_{\mathcal{C}^{0,\alpha}}\eta^\alpha.
\end{equation*}
Using a sequence $(\mu_k)$ of smooth functions with compact support and average $0$ converging to $\displaystyle{\frac{1}{N}\sum_{i=1}^N \widetilde{\delta}_{q_i}^{(\eta)} - \omega}$ as we have done for Claim~\ref{coerc_T1} we can show that
\begin{align*}
T_2 &= \int_{\mathbb{R}^2} \frac{1}{b} \nabla \xi \cdot \nabla H_{N,\eta} 
\end{align*}
and therefore
\begin{align*}
|T_2| \leq& \left(\int_{\mathbb{R}^2} \frac{1}{b} |\nabla \xi|^2\right)^\frac{1}{2}\left(\int_{\mathbb{R}^2} \frac{1}{b} |\nabla  H_{N,\eta}|^2\right)^\frac{1}{2} \\
\leq& C_b \left(\int_{\mathbb{R}^2} \frac{1}{b} |\nabla \xi|^2\right)^\frac{1}{2}\bigg(\mathcal{F}_b(Q_N,\omega)+ \frac{g(\eta)}{N} \\ 
&+ I(Q_N)(\eta+N^{-1}) + \norm{\omega}_{L^1((1+|x|)\dd x)\cap L^\infty}g(\eta)\eta\bigg)^\frac{1}{2}.
\end{align*}
 by Proposition~\ref{proposition_pre_coerc}. We conclude by taking $\eta = N^{-1}$. The bound on
\begin{equation*}
\norm{\frac{1}{N}\sum_{i=1}^N \delta_{q_i} - \omega}_{H^s}
\end{equation*} 
follows from Sobolev embeddings.
\end{proof}

We finish this section by proving the weak-$\ast$ convergence result:

\begin{proof}[Proof of Corollary~\ref{corollary_weak_star_cv}]

Let us denote $\displaystyle{\omega_N = \frac{1}{N}\sum_{i=1}^N\delta_{q_i}}$ and prove that $(\omega_N)$ is a tight sequence of probability measures. Let $R > 1$, then
\begin{equation}\label{bound_tightness}
\begin{aligned}
|\{i \in [1,N]\; ; \; |q_i| \geq R \}|R^2
&\leq \underset{|q_i| \geq R}{\sum_{i=1}^N}|q_i|^2 \\
&\leq N I(Q_N).
\end{aligned}
\end{equation}
Dividing by $NR^2$ both sides of the inequality we get
\begin{equation*}
\int_{B(0,R)^c} \dd \omega_N \leq I(Q_N)R^{-2}
\end{equation*}
and since $(I(Q_N))$ is bounded we get that $(\omega_N)$ is tight. We will now prove the following Claim:
\begin{claim}
Assume that $(\omega_N)$ converges to $\omega$ for the weak-$\ast$ topology of probability measures and that $(I(Q_N))$ is bounded. Then $\mathcal{F}_b(Q_N,\omega) \Tend{N}{+\infty} 0$ if and only if we have 
\begin{equation*}
\frac{1}{N^2} \sum_{1 \leq i \neq j \leq N} g_b(q_i,q_j) \longrightarrow \iint_{\mathbb{R}^2\times \mathbb{R}^2} g_b(x,y)\omega(x)\omega(y)\dd x \dd y.
\end{equation*}
\end{claim}

\begin{proof}
Let $\epsilon > 0$. We write the modulated energy as the sum of three terms:
\begin{equation}\label{split_mod_energy_3}
\begin{aligned}
\mathcal{F}_b(Q_N,\omega) =& -\iint_{\mathbb{R}^2\times\mathbb{R}^2} g_b(x,y) \omega(x)\omega(y)\dd x \dd y  \\
&+ \frac{1}{N^2}\sum_{1 \leq i \neq j \leq N} g_b(q_i,q_j) \\
&-2 \int_{\mathbb{R}^2} \psi(y)\dd(\omega_N-\omega)(y)
\end{aligned}
\end{equation}
where $\psi = G_b[\omega]$. Let $R \geq 1$ be such that $\supp(\omega) \subset B(0,R)$. We have
\begin{align*}
\int_{\mathbb{R}^2}\psi\dd (\omega-\omega_N) = -\int_{B(0,R)^c}\psi\dd\omega_N + \int_{B(0,R)}\psi\dd(\omega-\omega_N).
\end{align*}
We bound the first term as we did to obtain \eqref{bound_tightness}:
\begin{align*}
\left|\int_{B(0,R)^c}\psi\dd \omega_N\right| 
&\leq \frac{1}{N}\underset{|q_i| \geq R}{\sum_{i=1}^N}|\psi(q_i)| \\
&\leq \frac{C_b}{N}\underset{|q_i| \geq R}{\sum_{i=1}^N}(1+|q_i|^\delta) \\
&\leq C_b(R^{-2}I(Q_N) + R^{2-\delta} I(Q_N))
\end{align*}
for some $0 < \delta < 1$ (by Proposition~\ref{theorem_asymptotic_velocity_field_lake}). Therefore, 
\begin{equation*}
\left|\int_{B(0,R)^c}\psi\dd \omega_N\right| \leq \epsilon
\end{equation*}
if $R$ is big enough. Now let $\chi_{R,\beta}$ be a smooth function such that $0 \leq \chi \leq 1$, $\chi_{R,\beta}(x) = 1$ if $|x| \leq R$ and $\chi_{R,\beta}(x) = 0$ if $|x| \geq R+\beta$. 
Then
\begin{align*}
\int_{B(0,R)}\psi\dd (\omega-\omega_N) = \int \chi_{R,\beta}\psi\dd (\omega-\omega_N) - \int_{R \leq |x| \leq R+\beta} \chi_{R,\beta}\psi\dd (\omega-\omega_N)
\end{align*}
Choosing $\beta$ small enough we have
\begin{align*}
\left|\int_{R \leq |x| \leq R+\beta} \chi_{R,\beta}\psi\dd (\omega-\omega_N)\right| \leq \epsilon.
\end{align*} 
Now $\psi$ is continuous (see Lemma~\ref{lemma_link_solutions_duerinckx}) so by weak-$\ast$ convergence of $(\omega_N)$ to $\omega$ we get that
\begin{equation*}
\int\psi\chi_{R,\beta}\dd (\omega-\omega_N) \Tend{N}{+\infty} 0
\end{equation*}
and therefore
\begin{equation*}
\underset{N\rightarrow +\infty}{\limsup}\left|\int_{\mathbb{R}^2}\psi\dd (\omega-\omega_N)\right| \leq 2\epsilon.
\end{equation*}
for all $\epsilon > 0$, so we get
\begin{equation*}
\int_{\mathbb{R}^2}\psi\dd (\omega-\omega_N) \Tend{N}{+\infty} 0. 
\end{equation*}
Using \eqref{split_mod_energy_3} we get that $\mathcal{F}_b(Q_N,\omega) \Tend{N}{+\infty} 0$ if and only if we have 
\begin{equation*}
\frac{1}{N^2} \sum_{1 \leq i \neq j \leq N} g_b(q_i,q_j) \longrightarrow \iint_{\mathbb{R}^2\times \mathbb{R}^2} g_b(x,y)\omega(x)\omega(y)\dd x \dd y.
\end{equation*} 
\end{proof}

It follows directly from the Claim that $(2) \implies (1)$. Now if we have $(1)$, using Corollary~\ref{corollary_coercivity} we have convergence of $(\omega_N)$ to $\omega$ in any $H^{s}$ for any $s < -1$. It follows by Prokhorov's theorem $(\omega_N)$ converges to $\omega$ for the weak-$\ast$ topology of probability measures. By the Claim we also have convergence of the interaction energy and therefore $(1) \implies (2)$.
\end{proof}

\section{Proof of the main Proposition \ref{controle_terme_principal_gronwall}}\label{section:6}

Let us recall that for $q \in \mathbb{R}^2$, $Q_N = (q_1,...,q_N) \in (\mathbb{R}^2)^N$ and $\displaystyle{0 < \eta < 1}$, we have denoted
\begin{equation*}
I(Q_N) = \frac{1}{N}\sum_{i=1}^N |q_i|^2,
\end{equation*}
\begin{equation*}
\widetilde{\delta}_q^{(\eta)} = m_b(q,\eta)\frac{\dd \delta_{q}^{(\eta)}}{\sqrt{b}}
\end{equation*}
and
\begin{equation*}
m_b(q,\eta) = \left(\int_{\mathbb{R}^2}\frac{\dd \delta_{q}^{(\eta)}}{\sqrt{b}}\right)^{-1}
\end{equation*}
where $\displaystyle{\delta_{q}^{(\eta)}}$ is the uniform probability measure on the circle $\partial B(q,\eta)$.

In this Section, we prove the following result:
\begin{proposition}\label{controle_terme_principal_gronwall}
Let $Q_N = (q_1,...,q_N) \in (\mathbb{R}^2)^N$ such that $q_i \neq q_j$ if $i \neq j$, $u \in W^{1,\infty}(\mathbb{R}^2,\mathbb{R}^2)$ and $\omega \in \mathcal{P}(\mathbb{R}^2)\cap L^\infty(\mathbb{R}^2)$ with compact support such that $\nabla G_b[\omega]$ is continuous and bounded. There exists $\beta \in (0,1)$ (independent of $\omega$, $u$ and $Q_N$) such that
\begin{align*}
\bigg|\iint_{(\mathbb{R}^2\times\mathbb{R}^2)\backslash\Delta} &u(x)\cdot \nabla_x g_b(x,y)\dd\left(\frac{1}{N}\sum_{i=1}^N\delta_{q_i}-\omega\right)^{\otimes 2}(x,y)\bigg| \\
\leq& C_b\norm{u}_{W^{1,\infty}}|\mathcal{F}_b(Q_N,\omega)| \\
&+ C_b(1+\norm{u}_{W^{1,\infty}})\norm{\omega}_{L^1((1+|x|)\dd x)\cap L^\infty}(1+I(Q_N))N^{-\beta}.
\end{align*}

\end{proposition}

This proposition is an equivalent of \cite[Proposition~1.1]{Serfaty} or \cite[Proposition~4.1]{NguyenRosenzweigSerfaty} and the proof will follow the same steps: regularise the dirac masses, use the structure of the lake kernel to bound the regular part and control the remainders. Some terms are very similar to the ones obtained in the Coulomb case and we will use both the properties of our regularisation (see Subsection~\ref{subsection:24}) and some estimates proved in \cite{NguyenRosenzweigSerfaty} to bound them. As in the proof of Proposition~\ref{proposition_pre_coerc} some terms are specific to the lake kernel and we will use results of Section~\ref{section:2} to bound them.

\begin{proof}
Let us fix $\displaystyle{0 < \eta < \frac{1}{8}}$ and write
\begin{equation}\label{decomposition_main_proposition}
\begin{aligned}
&\iint_{(\mathbb{R}^2\times\mathbb{R}^2)\backslash\Delta} u(x)\cdot \nabla_x g_b(x,y)\dd\left(\frac{1}{N}\sum_{i=1}^N\delta_{q_i}-\omega\right)^{\otimes 2}(x,y) \\
=& \iint_{\mathbb{R}^2\times\mathbb{R}^2} u(x)\cdot \nabla_x g_b(x,y)\dd\left(\frac{1}{N}\sum_{i=1}^N\widetilde{\delta}^{(\eta)}_{q_i}-\omega\right)^{\otimes 2}(x,y) \\
&+\bigg(-\frac{1}{N}\sum_{i=1}^N\iint_{\mathbb{R}^2\times\mathbb{R}^2} u(x)\cdot \nabla_x g_b(x,y)\bigg[\dd \omega(x)\dd(\delta_{q_i}-\widetilde{\delta}^{(\eta)}_{q_i})(y) \\
&+ \dd(\delta_{q_i}-\widetilde{\delta}^{(\eta)}_{q_i})(x) \dd \omega(y)\bigg]\bigg) \\
&+\bigg(\frac{1}{N^2}\underset{1 \leq i,j \leq N}{\sum}\iint_{(\mathbb{R}^2\times\mathbb{R}^2)\backslash\Delta} u(x)\cdot \nabla_x g_b(x,y)\\
&[\dd \delta_{q_i}(x) \dd \delta_{q_j}(y) - \dd\widetilde{\delta}^{(\eta)}_{q_i}(x)\dd \widetilde{\delta}^{(\eta)}_{q_j}(y)]\bigg) \\
=:& T_1 + T_2 + T_3.
\end{aligned}
\end{equation}

Let us bound the first term. As in Section~\ref{section:5} we write 
\begin{equation*}
H_{N,\eta} := G_b\left[\frac{1}{N}\sum_{i=1}^N\widetilde{\delta}^{(\eta)}_{q_i}-\omega\right].
\end{equation*}

We claim:
\begin{claim}\label{claim_delort_type_limit_lake}
\begin{align*}
T_1 =& -\int_{\mathbb{R}^2} u(x) \cdot \nabla H_{N,\eta}(x) \nabla\left(\frac{1}{b}\right)\cdot\nabla H_{N,\eta}(x)\dd x \\
&+\int_{\mathbb{R}^2} \nabla\left(\frac{1}{2b}u\right):[H_{N,\eta},H_{N,\eta}] 
\end{align*}
\end{claim}

\begin{proof}[Proof of the Claim]
This claim is similar to \cite[Lemma~4.3]{Serfaty} and we proceed the same way: Let us first fix $\mu$ smooth with compact support and average zero and write $H_\mu = G_b[\mu]$. Then
\begin{align*}
\iint_{\mathbb{R}^2\times\mathbb{R}^2} u(x)\cdot \nabla_x g_b(x,y)&\dd \mu^{\otimes 2}(x,y) \\
=& - \int_{\mathbb{R}^2} u(x)\cdot \nabla H_\mu(x) \div\left(\frac{1}{b}\nabla H_\mu\right)(x)\dd x \\
=& -\int_{\mathbb{R}^2} u(x) \cdot \nabla H_\mu(x) \nabla\left(\frac{1}{b}\right)\cdot\nabla H_\mu(x)\dd x \\
&- \int_{\mathbb{R}^2}\frac{1}{b}u\cdot \nabla H_\mu\Delta H_\mu.
\end{align*}
For the second integral of the right handside we proceed as in \cite{Serfaty} and use the stress-energy tensor defined by \eqref{definition_SE} (for more details, see \cite[Equality~(1.25)]{Serfaty} and the associated references):
\begin{equation*}
\int_{\mathbb{R}^2}\frac{1}{b}u\cdot \nabla H_\mu\Delta H_\mu = \int_{\mathbb{R}^2} \frac{1}{2b}u\cdot \div([H_\mu,H_\mu]).
\end{equation*}
Integrating over a ball of radius $R$ and integrating by parts we get
\begin{align*}
\int_{B(0,R)}\frac{1}{2b}u\cdot  \div([H_\mu,H_\mu])
=& \int_{\partial B(0,R)} \frac{1}{2b}[H_\mu,H_\mu]u\cdot \dd \vec{S} \\
&- \int_{B(0,R)}\nabla\left(\frac{1}{2b}u\right):[H_\mu,H_\mu].
\end{align*}
Using Proposition~\ref{theorem_asymptotic_velocity_field_lake} (applied to $\omega = \mu$ and $\psi = H_\mu$) we have
\begin{equation*}
\left|\int_{\partial B(0,R)} \frac{1}{2b}[H_\mu,H_\mu]u\cdot \dd \vec{S}\right| \leq \frac{C_{b,\mu}\norm{u}_{L^\infty}}{R^4}R.
\end{equation*}
Letting $R \longrightarrow \infty$ we obtain
\begin{align*}
\int_{\mathbb{R}^2} \frac{1}{2b}u\cdot \div([H_{\mu},H_{\mu}]) &= -\int_{\mathbb{R}^2} \nabla\left(\frac{1}{2b}u\right):[H_{\mu},H_{\mu}].
\end{align*}
Now if $(\mu_k)$ is a sequence of smooth functions with compact support and average zero such that
\begin{equation*}
\mu_k - \bigg(\frac{1}{N}\sum_{i=1}^N\widetilde{\delta}^{(\eta)}_{q_i} - \omega\bigg) \Tend{N}{+\infty} 0 \; \text{in} \; \dot H^{-1}
\end{equation*}
then by Lemma~\ref{nabla_G_b_bounded_H_minus_one} we have
\begin{equation*}
\nabla H_{\mu_k} \Tend{N}{+\infty} \nabla H_{N,\eta}  \; \text{in} \; L^2
\end{equation*}
and therefore since $u \in W^{1,\infty}$ and since $[H_{\mu_k},H_{\mu_k}]$ (defined by Equation \eqref{definition_SE}) is quadratic in the derivatives of $H_{\mu_k}$ we get that
\begin{equation*}
-\int_{\mathbb{R}^2} u(x) \cdot \nabla H_{\mu_k}(x) \nabla\left(\frac{1}{b}\right)\cdot\nabla H_{\mu_k}(x)\dd x 
+\int_{\mathbb{R}^2} \nabla\left(\frac{1}{2b}u\right):[H_{\mu_k},H_{\mu_k}] 
\end{equation*}
converges to
\begin{equation*}
-\int_{\mathbb{R}^2} u(x) \cdot \nabla H_{N,\eta}(x) \nabla\left(\frac{1}{b}\right)\cdot\nabla H_{N,\eta}(x)\dd x 
+\int_{\mathbb{R}^2} \nabla\left(\frac{1}{2b}u\right):[H_{N,\eta},H_{N,\eta}] 
\end{equation*}
as $k \longrightarrow +\infty$. We are only left to justify that
\begin{multline*}
\iint_{\mathbb{R}^2\times\mathbb{R}^2} u(x)\cdot \nabla_x g_b(x,y)\dd \mu_k^{\otimes 2}(x,y) \\
\Tend{k}{+\infty} \iint_{\mathbb{R}^2\times\mathbb{R}^2} u(x)\cdot \nabla_x g_b(x,y)\dd\left(\frac{1}{N}\sum_{i=1}^N\widetilde{\delta}^{(\eta)}_{q_i}-\omega\right)^{\otimes 2}(x,y).
\end{multline*}
We define
\begin{align*}
m &= \frac{1}{N}\sum_{i=1}^N\widetilde{\delta}^{(\eta)}_{q_i}.
\end{align*}
Let us consider a sequence $(\nu_k)$ of smooth probability densities with support included in a ball $B(0,R)$ independent of $k$ (containing $\supp(m)$), such that $(\nu_k - m)$ converges to zero in $\dot H^{-1}$ and for the weak-$\ast$ topology of probability measures. If we set $\mu_k = \nu_k - \omega$, then  
\begin{equation*}
\mu_k - (m-\omega) \Tend{k}{+\infty}  0 \; \text{in} \; \dot H^{-1}.
\end{equation*} 
Now we write
\begin{align*}
&\iint_{\mathbb{R}^2\times\mathbb{R}^2} u(x)\cdot \nabla_x g_b(x,y)\dd \mu_k^{\otimes 2}(x,y) \\
&- \iint_{\mathbb{R}^2\times\mathbb{R}^2} u(x)\cdot \nabla_x g_b(x,y)\dd (m-\omega)^{\otimes 2}(x,y) \\
=& \iint_{\mathbb{R}^2\times\mathbb{R}^2} u(x)\cdot \nabla_x g_b(x,y)\omega(x)\dd x \dd(m-\nu_k)(y) \\
&+ \iint_{\mathbb{R}^2\times\mathbb{R}^2} u(x)\cdot \nabla_x g_b(x,y)\omega(y)\dd y\dd(m-\nu_k)(x) \\
&+ \iint_{\mathbb{R}^2\times\mathbb{R}^2} u(x)\cdot \nabla_x g_b(x,y)\dd(\nu_k\otimes\nu_k - m\otimes m)(x,y) \\
=:& I_1 + I_2 + I_3.
\end{align*}
We have
\begin{align*}
|I_1| = \bigg|\int_{\mathbb{R}^2}  u \cdot \nabla G_b[m-\nu_k]\omega \bigg|
&\leq \norm{u}_{L^\infty}\norm{\omega}_{L^2}\norm{\nabla G_b[m-\nu_k]}_{L^2} \\
&\leq C\norm{u}_{L^\infty}\norm{\omega}_{L^2}\norm{m-\nu_k}_{\dot H^{-1}}
\end{align*}
by Lemma~\ref{nabla_G_b_bounded_H_minus_one} and therefore $I_1 \Tend{k}{+\infty}  0$. Recall that $(m-\nu_k)$ converges to zero for the weak-$\ast$ topology of probability measures. Therefore 
\begin{align*}
I_2 = \int_{\mathbb{R}^2} u\cdot \nabla G_b[\omega] \dd(m-\nu_k) \Tend{k}{+\infty}  0
\end{align*}
since $u$ and $\nabla G_b[\omega]$ are continuous and bounded by assumption. Now we want to show that $I_3$ converges to zero. Remark that writing $\mu_k = \nu_k - \omega$ and proving that $I_1$ and $I_2$ converge to zero allowed us to restrict ourself to study the convergence of
\begin{equation*}
\iint_{\mathbb{R}^2\times\mathbb{R}^2} u(x)\cdot \nabla_x g_b(x,y)\dd \nu_k(x)\dd\nu_k(y)
\end{equation*}
for $\nu_k$ nonnegative (which will be crucial for using Delort's argument below).
We use the definition of $g_b$ \eqref{definition_g_b} to write
\begin{align*}
I_3 =& \iint_{\mathbb{R}^2\times\mathbb{R}^2}  u(x)\cdot \nabla \sqrt{b}(x)\sqrt{b(y)}g(x-y)\dd(\nu_k\otimes\nu_k - m\otimes m)(x,y) \\
&+ \iint_{\mathbb{R}^2\times\mathbb{R}^2}  \sqrt{b(x)b(y)} u(x)\cdot \nabla g(x-y)\dd(\nu_k\otimes\nu_k - m\otimes m)(x,y) \\
&+ \iint_{\mathbb{R}^2\times\mathbb{R}^2}  u(x)\cdot\nabla_x S_b(x,y) \dd(\nu_k\otimes\nu_k - m\otimes m)(x,y) \\
=:& I_{3,1} + I_{3,2} + I_{3,3}.
\end{align*}
We write
\begin{equation}\label{decomposition_I_3_1}
\begin{aligned}
I_{3,1} =& \iint_{\mathbb{R}^2\times\mathbb{R}^2}  u(x)\cdot \nabla \sqrt{b}(x)\sqrt{b(y)}g(x-y) \dd(\nu_k-m)(x)\dd \nu_k(y) \\
&+ \iint_{\mathbb{R}^2\times\mathbb{R}^2}  u(x)\cdot \nabla \sqrt{b}(x)\sqrt{b(y)}g(x-y) \dd m(x) \dd(\nu_k-m)(y) \\
=& \int_{\mathbb{R}^2} (u\cdot \nabla \sqrt{b})(g\ast [\sqrt{b}\nu_k])\dd(\nu_k-m) \\
&+ \int_{\mathbb{R}^2} (u\cdot \nabla \sqrt{b})(g\ast[\sqrt{b}(\nu_k-m)])\dd m. 
\end{aligned}
\end{equation}
Recall that $B(0,R)$ is a ball containing the supports of $m$ and $\nu_k$. Consider a smooth probability density $\rho$ with support in $B(0,R)$. We define 
\begin{align*}
\chi_k &= \left(\int_{\mathbb{R}^2} \sqrt{b}\nu_k\right)\rho, \\
\chi_\infty &= \left(\int_{\mathbb{R}^2} \sqrt{b}\dd m\right)\rho 
\end{align*}
and write
\begin{equation}\label{decomposition_claim_cv_H_1_loc}
\begin{aligned}
\nabla g\ast(\sqrt{b}(\nu_k - m)) =& \nabla g\ast(\sqrt{b}\nu_k - \chi_k+\chi_\infty - \sqrt{b}m) \\
&+ \left(\int_{\mathbb{R}^2} \sqrt{b}\nu_k-\int_{\mathbb{R}^2} \sqrt{b}\dd m\right)\nabla g\ast\rho. 
\end{aligned}
\end{equation}
Now
\begin{multline*}
\norm{\nabla g\ast(\sqrt{b}\nu_k - \chi_k+\chi_\infty - \sqrt{b}m)}^2_{L^2} \\
= C\int_{\mathbb{R}^2}\frac{1}{|\xi|^2}|\widehat{\sqrt{b}\nu_k}(\xi) - \widehat{\chi_k}(\xi)+\widehat{\chi_\infty}(\xi) - \widehat{\sqrt{b}m}(\xi)|^2\dd \xi. 
\end{multline*}
Remark that $\alpha_k=\sqrt{b}\nu_k - \chi_k+\chi_\infty - \sqrt{b}m$ is a Radon measure with support included in $B(0,R)$ such that $\widehat{\alpha_k}(0) = 0$. Therefore
\begin{align*}
\left|\int_{\mathbb{R}^2} e^{-ix\cdot\xi}\dd \alpha_k(x)\right| 
&= \left|\int_{\mathbb{R}^2}(e^{-ix\cdot\xi}-1)\dd \alpha_k(x)\right| \\
&=2\left|\int_{\mathbb{R}^2}\sin\left(\frac{x\cdot\xi}{2}\right)\dd \alpha_k(x)\right| \\
&\leq CR|\xi|\int_{\mathbb{R}^2}\dd |\alpha_k|(x) \\
&\leq C_{b,R}|\xi|.
\end{align*}
It follows that for $\epsilon > 0$
\begin{align*}
\int_{|\xi| \leq \epsilon}\frac{1}{|\xi|^2}|\widehat{\alpha_k}(\xi)|^2\dd \xi \leq C_{b,R}\epsilon^2.
\end{align*}
Moreover,
\begin{align*}
\int_{|\xi| \geq \epsilon}&\frac{1}{|\xi|^2}|\widehat{\sqrt{b}\nu_k}(\xi) - \widehat{\chi_k}(\xi)+\widehat{\chi_\infty}(\xi) - \widehat{\sqrt{b}m}(\xi)|^2\dd \xi \\
&\leq C_\epsilon\left(\int_{\mathbb{R}^2}|\widehat{\chi_k}(\xi)-\widehat{\chi_\infty}(\xi)|^2\dd \xi + \int_{\mathbb{R}^2}\frac{1}{1+|\xi|^2}|\widehat{\sqrt{b}\nu_k}(\xi) - \widehat{\sqrt{b}m}(\xi)|^2\right) \\
&\leq C_\epsilon\left(\norm{\chi_k - \chi_\infty}_{L^2} + \norm{\sqrt{b}\nu_k-\sqrt{b}m}_{H^{-1}}\right) \Tend{k}{+\infty} 0
\end{align*}
since $b$ is smooth. Therefore 
\begin{align}
\underset{k\rightarrow+\infty}{\limsup}\norm{\nabla g\ast(\sqrt{b}\nu_k - \chi_k+\chi_\infty - \sqrt{b}m)}^2_{L^2} \leq C_{b,R}\epsilon^2
\end{align}
for all $\epsilon > 0$ so
\begin{equation}\label{cv_grad_L_2_claim_cv_H_1_loc}
\nabla g\ast(\sqrt{b}\nu_k - \chi_k+\chi_\infty - \sqrt{b}m) \overset{L^2}{\Tend{k}{+\infty}} 0.
\end{equation}
By Hardy-Littlewood-Sobolev inequality (see for example \cite[Theorem~1.7]{BahouriCheminDanchin}), $\nabla g\ast\rho \in L^p$  for all $2 < p < +\infty$ so 
\begin{align*}
\left(\int_{\mathbb{R}^2} \sqrt{b}\nu_k-\int_{\mathbb{R}^2} \sqrt{b}m\right)\nabla g\ast\rho \Tend{k}{+\infty} 0 \; \text{in} \; L^2(B(0,R)).
\end{align*}
Combining the upper limit with \eqref{decomposition_claim_cv_H_1_loc} and  \eqref{cv_grad_L_2_claim_cv_H_1_loc} we get that
\begin{equation*}
\nabla g\ast(\sqrt{b}\nu_k) \Tend{k}{+\infty}  \nabla g\ast(\sqrt{b}m) \; \text{in} \; L^2(B(0,R)).
\end{equation*}
Now, by convolution inequality, we have
\begin{equation}\label{convergence_g_ast_nu_L_2}
\norm{g\ast[\sqrt{b}\nu_k]}_{L^2(B(0,R))} \leq C_b\norm{g}_{L^2(B(0,2R))}\norm{\nu_k}_{L^1} \leq C_b \norm{g}_{L^2(B(0,2R))}
\end{equation}
so $(g\ast[\sqrt{b}\nu_k])$ is bounded in $H^1(B(0,R))$ which is compactly embedded in $ L^2(B(0,R))$. Therefore by \eqref{cv_grad_L_2_claim_cv_H_1_loc}, up to extraction, $(g\ast[\sqrt{b}\nu_k])$ converges to $g\ast[\sqrt{b}m] + C$ where $C$ is a constant. If $x_0 \in B(0,R)$ is at a positive distance from the supports of $\nu_k$ and $m$ then $g(x_0-\cdot)$ is continuous on the supports of $\nu_k$ and $m$ and therefore
\begin{equation*}
g\ast[\sqrt{b}\nu_k](x_0) \Tend{k}{+\infty} g\ast[\sqrt{b}m](x_0)
\end{equation*}
by dominated convergence theorem. It follows that $C = 0$, thus
\begin{equation*}
g\ast[\sqrt{b}\nu_k] \Tend{k}{+\infty}  g\ast[\sqrt{b}m] \; \text{in} \; H^1(B(0,R)).
\end{equation*}
We recall that since $b$ is smooth,
\begin{align*}
\sqrt{b}\nu_k \Tend{k}{+\infty}  \sqrt{b}m \; \text{in} \; H^{-1}.
\end{align*}
Moreover, $m \in H^{-1}$ with compact support and $u\cdot\nabla \sqrt{b} \in W^{1,\infty}$ so it follows by Decomposition \ref{decomposition_I_3_1} that
\begin{align*}
I_{3,1} \Tend{k}{+\infty} 0.
\end{align*}

Since $\nabla g$ is antisymmetric we can write
\begin{align*}
I_{3,2} =& \frac{1}{2}\iint_{\mathbb{R}^2\times\mathbb{R}^2}H_u(x,y)\dd(\sqrt{b}\nu_k)(x)\dd(\sqrt{b}\nu_k)(y)\\
&- \frac{1}{2}\iint_{\mathbb{R}^2\times\mathbb{R}^2}H_u(x,y)\dd (\sqrt{b}m)(x)\dd(\sqrt{b}m)(y)
\end{align*}
where
\begin{align*}
H_u(x,y) = \frac{1}{2}(u(x)-u(y))\cdot\nabla g(x-y).
\end{align*}
We recall that $(\sqrt{b}\nu_k)$ is a sequence of nonnegative functions with supports in $B(0,R)$  converging to $\sqrt{b}m$ in $H^{-1}$ and for the weak-$\ast$ topology of measures with finite mass. Moreover, since $u$ is Lipschitz, $H_u$ is continuous outside of the diagonal and bounded. Therefore we can use Delort's argument (see \cite[Proposition~1.2.6]{Delort} or \cite[Inequalities (3.4) and (3.5)]{Schochet2}) to prove that 
\begin{align*}
I_{3,2} \Tend{k}{+\infty} 0.
\end{align*}
Finally we write
\begin{align*}
I_{3,3} =& \int_{\mathbb{R}^2}  u(x)\cdot\left(\int_{\mathbb{R}^2}\nabla_x S_b(x,y)\dd \nu_k(y)\right)  \dd(\nu_k-m)(x)\\
&+\int_{\mathbb{R}^2}u(x)\cdot\left(\int_{\mathbb{R}^2} \nabla_x S_b(x,y) \dd m(x)\right) \dd(\nu_k-m)(y).
\end{align*}
By Proposition~\ref{estimees_S_b_not_necessarily_symmetric} $u(x)\cdot\nabla_x S_b(x,y)$ is locally Hölder with respect to both variables and therefore since $\nu_k\otimes\nu_k - m\otimes m$ has compact support we have that $I_{3,3} \Tend{k}{+\infty} 0$.
\end{proof}

It follows from Claim \ref{claim_delort_type_limit_lake} that
\begin{align*}
|T_1| \leq C_b\norm{u}_{W^{1,\infty}}\int_{\mathbb{R}^2} |\nabla H_{N,\eta}|^2.
\end{align*}
Hence by Proposition~\ref{proposition_pre_coerc} we get
\begin{equation}\label{gronwall_bound_T1}
\begin{aligned}
|T_1| \leq& C_b\norm{u}_{W^{1,\infty}}\bigg(|\mathcal{F}_b(Q_N,\omega)| + \frac{g(\eta)}{N} + I(Q_N)(\eta+N^{-1}) \\&+ \norm{\omega}_{L^1((1+|x|)\dd x)\cap L^\infty}g(\eta)\eta\bigg).
\end{aligned}
\end{equation}

Now let us split $T_2$ in three parts:
\begin{align*}
T_2 =&
-\frac{1}{N}\sum_{i=1}^N\iint_{\mathbb{R}^2\times\mathbb{R}^2} \big(u(x)\cdot \nabla_x g_b(x,y) \\
&+ u(y)\cdot \nabla_x g_b(y,x)\big)\omega(x)\dd x\dd(\delta_{q_i}-\widetilde{\delta}^{(\eta)}_{q_i})(y) \\
=& -\frac{1}{N}\sum_{i=1}^N\iint_{\mathbb{R}^2\times\mathbb{R}^2}\big(u(x)\cdot \nabla \sqrt{b}(x) \sqrt{b(y)} \\
&+ u(y)\cdot \nabla \sqrt{b}(y) \sqrt{b(x)}\big)g(x-y)\omega(x)\dd x\dd(\delta_{q_i}-\widetilde{\delta}^{(\eta)}_{q_i})(y) \\
&- \frac{1}{N}\sum_{i=1}^N \iint_{\mathbb{R}^2\times\mathbb{R}^2} \sqrt{b(x)b(y)}(u(x) - u(y))\\
&\cdot \nabla g(x-y) \omega(x)\dd x\dd(\delta_{q_i}-\widetilde{\delta}^{(\eta)}_{q_i})(y) \\
&- \frac{1}{N}\sum_{i=1}^N\iint_{\mathbb{R}^2\times\mathbb{R}^2}\big(u(x)\cdot \nabla_x S_b(x,y) \\
&+ u(y)\cdot \nabla_x S_b(y,x)\big) \omega(x)\dd x\dd(\delta_{q_i}-\widetilde{\delta}^{(\eta)}_{q_i})(y) \\
=:& -(T_{2,1} + T_{2,2} + T_{2,3}).
\end{align*}

We will bound the three terms separately:
\begin{claim}\label{bound_T21}
There exists $0 < s < 1$ such that
\begin{equation*}
|T_{2,1}| \leq C_b\norm{u}_{W^{1,\infty}}\norm{\omega}_{L^1((1+|x|)\dd x)\cap L^\infty}(1+I(Q_N))\eta^s.
\end{equation*}
\end{claim}

\begin{proof}[Proof of the claim]
Since $\widetilde{\delta}^{(\eta)}_{q_i}$ is a probability measure, we can write
\begin{align*}
T_{2,1} =& \frac{1}{N}\sum_{i=1}^N\iint_{\mathbb{R}^2\times\mathbb{R}^2} \nabla \sqrt{b}(x)\cdot u(x) \omega(x)(\sqrt{b(q_i)}g(x-q_i) \\
&- \sqrt{b(y)}g(x-y))\dd \widetilde{\delta}^{(\eta)}_{q_i}(y) \dd x \\
&+ \frac{1}{N}\sum_{i=1}^N\iint_{\mathbb{R}^2\times\mathbb{R}^2} \sqrt{b(x)}\omega(x)(\nabla \sqrt{b}(q_i)\cdot u(q_i)g(x-q_i) \\
&- \nabla \sqrt{b}(y)\cdot u(y)g(x-y))\dd \widetilde{\delta}^{(\eta)}_{q_i}(y) \dd x \\
=& \frac{1}{N}\sum_{i=1}^N\iint_{\mathbb{R}^2\times\mathbb{R}^2} (\nabla \sqrt{b}\cdot u \omega)(x)(\sqrt{b(q_i)}-\sqrt{b(y)})g(x-y)\dd \widetilde{\delta}^{(\eta)}_{q_i}(y) \dd x\\
&+ \frac{1}{N}\sum_{i=1}^N\iint_{\mathbb{R}^2\times\mathbb{R}^2} (\nabla \sqrt{b}\cdot u \omega)(x)\sqrt{b(q_i)}\\
&\times(g(x-q_i) - g(x-y))\dd \widetilde{\delta}^{(\eta)}_{q_i}(y) \dd x\\
&+ \frac{1}{N}\sum_{i=1}^N\iint_{\mathbb{R}^2\times\mathbb{R}^2} \sqrt{b(x)}\omega(x)(\nabla \sqrt{b}(q_i)\cdot u(q_i) \\
&- \nabla \sqrt{b}(y)\cdot u(y))g(x-y)\dd\widetilde{\delta}^{(\eta)}_{q_i}(y) \dd x \\
&+ \frac{1}{N}\sum_{i=1}^N\iint_{\mathbb{R}^2\times\mathbb{R}^2} \sqrt{b(x)}\omega(x)\nabla \sqrt{b}(q_i)\cdot u(q_i)\\
&\times(g(x-q_i) -g(x-y))\dd \widetilde{\delta}^{(\eta)}_{q_i}(y)\dd x.
\end{align*}

For the first integral, we use the Lipschitz regularity of $\sqrt{b}$ to bound
\begin{multline*}
\bigg|\iint_{\mathbb{R}^2\times\mathbb{R}^2} (\nabla \sqrt{b}\cdot u \omega)(x) (\sqrt{b(q_i)}-\sqrt{b(y)})g(x-y)\dd \widetilde{\delta}^{(\eta)}_{q_i}(y) \dd x\bigg| \\
\leq C_b\eta \iint_{\mathbb{R}^2\times\mathbb{R}^2}   |(\nabla \sqrt{b}\cdot u \omega)(x)g(x-y)|\dd \widetilde{\delta}^{(\eta)}_{q_i}(y)\dd x.
\end{multline*}

Moreover for $y \in \partial B(q_i,\eta)$, we have
\begin{align*}
\int_{\mathbb{R}^2} &|(\nabla \sqrt{b}\cdot u \omega)(x)g(x-y)|\dd x \\
\leq&  \int_{B(y,1)} |(\nabla \sqrt{b}\cdot u \omega)(x)g(x-y)|\dd x \\
&+ \int_{B(y,1)^c} |(\nabla \sqrt{b}\cdot u \omega)(x) g(x-y)|\dd x \\
\leq& \norm{\nabla \sqrt{b}\cdot u \omega}_{L^\infty}\norm{g}_{L^1(B(0,1))} +\int_{B(y,1)^c} |(\nabla \sqrt{b}\cdot u \omega)(x)|(|x|+|y|)\dd x \\
\leq& C_b\norm{u}_{L^\infty}\norm{\omega}_{L^\infty}(1+|q_i|)
\end{align*}
since $b$ satisfies Assumption~\ref{assumption_b}. Therefore
\begin{multline*}
\bigg|\iint_{\mathbb{R}^2\times\mathbb{R}^2} (\nabla \sqrt{b}\cdot u \omega)(x) (\sqrt{b(q_i)}-\sqrt{b(y)})g(x-y)\dd \widetilde{\delta}^{(\eta)}_{q_i}(y) \dd x\bigg| \\
\leq C_b\norm{u}_{L^\infty}\norm{\omega}_{L^\infty}(1+|q_i|)\eta.
\end{multline*}
The third integral can be bounded in the same way:
\begin{align*}
\bigg|\iint_{\mathbb{R}^2\times\mathbb{R}^2} &\sqrt{b(x)}\omega(x)(\nabla \sqrt{b}(q_i)\cdot u(q_i)- \nabla \sqrt{b}(y)\cdot u(y))g(x-y)\dd \widetilde{\delta}^{(\eta)}_{q_i}(y) \dd x\bigg|  \\
\leq& C_b\norm{u}_{W^{1,\infty}}\eta \iint_{\mathbb{R}^2\times\mathbb{R}^2} |\sqrt{b(x)}\omega(x)g(x-y)|\dd \widetilde{\delta}^{(\eta)}_{q_i}(y) \dd x \\
\leq& C_b\norm{u}_{W^{1,\infty}}\norm{\omega}_{L^1((1+|x|)\dd x)\cap L^\infty}(1+|q_i|)\eta.
\end{align*}
Summing over $N$ we get that both the first and the third line can be bounded by
\begin{equation}\label{first_bound_T21}
C_b\norm{u}_{W^{1,\infty}}\norm{\omega}_{L^1((1+|x|)\dd x)\cap L^\infty}(1+I(Q_N))\eta.
\end{equation}
Now the second integral is equal to
\begin{align*}
\frac{1}{N}\sum_{i=1}^N\sqrt{b(q_i)}\int_{\mathbb{R}^2}(g\ast(\nabla \sqrt{b}\cdot u \omega)(q_i) - g\ast(\nabla \sqrt{b}\cdot u \omega)(y))\dd \delta_{q_i}^{(\eta)}(y) \\
\end{align*}
and thus by Morrey's inequality (see \cite[Theorem~9.12]{Brezis}) its absolute value can be bounded by
\begin{equation*}
C_{b,p}\eta^{1-\frac{2}{p}}\norm{\nabla g\ast(\nabla \sqrt{b}\cdot u \omega)}_{L^p}
\end{equation*}
for any finite $p > 2$. The fourth integral can be bounded in the same way by
\begin{equation*}
C_{b,p}\eta^{1-\frac{2}{p}}\norm{\nabla g\ast(\sqrt{b}\omega)}_{L^p}.
\end{equation*}
Using Hardy-Littlewood-Sobolev inequality (see for example \cite[Theorem~1.7]{BahouriCheminDanchin}) we have
\begin{equation}\label{second_bound_T21}
C_b\eta^{1-\frac{2}{p}}\norm{\nabla g\ast(\sqrt{b}\omega)}_{L^p} \leq C_b\eta^{1-\frac{2}{p}}\norm{\omega}_{L^{\frac{2p}{p+2}}}.
\end{equation}
Combining \eqref{first_bound_T21} and \eqref{second_bound_T21} we get that
\begin{equation*}
|T_{2,1}| \leq C_b\norm{u}_{W^{1,\infty}}\norm{\omega}_{L^1((1+|x|)\dd x)\cap L^\infty}(1+I(Q_N))\eta^s
\end{equation*}
for some $0 < s < 1$.
\end{proof}

Now we bound $T_{2,2}$:
\begin{claim}\label{bound_T22}
\begin{equation*}
|T_{2,2}| \leq C_b\norm{\nabla u}_{L^\infty}\norm{\omega}_{L^1\cap L^\infty}\eta.
\end{equation*}
\end{claim}

\begin{proof}[Proof of the claim]
Let us recall that
\begin{equation*}
\widetilde{\delta}^{(\eta)}_{q} = m_b(q,\eta)\frac{\dd \delta_{q}^{(\eta)}}{\sqrt{b}}
\end{equation*}
and thus 
\begin{align*}
&\iint_{\mathbb{R}^2\times\mathbb{R}^2} \sqrt{b(x)b(y)}(u(x) - u(y))\cdot \nabla g(x-y) \omega(x)\dd x \dd(\delta_{q_i}-\widetilde{\delta}^{(\eta)}_{q_i})(y) \\
=& m_b(q_i,\eta)\iint_{\mathbb{R}^2\times\mathbb{R}^2} \sqrt{b(x)}(u(x) - u(y))\cdot \nabla g(x-y) \omega(x)\dd x \dd(\delta_{q_i}-\delta_{q_i}^{(\eta)})(y) \\
&+ \left(1- \frac{m_b(q_i,\eta)}{\sqrt{b(q_i)}}\right)\int_{\mathbb{R}^2} \sqrt{b(x)b(q_i)}(u(x)-u(q_i))\cdot \nabla g(x-q_i)\omega(x)\dd x.
\end{align*}
The first integral is exactly the term defined in \cite[Equation~(4.10)]{NguyenRosenzweigSerfaty} with $s=0$ and $m=0$ (remark that we can choose $m=0$ since no extension procedure is needed for $s=0$ and $d=2$, for more details we refer to the introduction of \cite[Section~4]{NguyenRosenzweigSerfaty}). It can be bounded by the right hand side of \cite[Equation~(4.24)]{NguyenRosenzweigSerfaty} :
\begin{align*}
C\norm{\nabla u}_{L^\infty}\norm{|\nabla|^{-1}(\sqrt{b}\omega)}_{L^\infty}\eta
&\leq C_b\norm{\nabla u}_{L^\infty}\norm{|\nabla|^{-1}\omega}_{L^\infty}\eta \\
&\leq C_b\norm{\nabla u}_{L^\infty}\norm{\omega}_{L^1\cap L^\infty}\eta.
\end{align*}
A proof of the last inequality can be found for example in \cite[Lemma~1]{Iftimie}. Now by \eqref{estimee_m_b} and the Lipschitz regularity of $u$ we can bound the second line by
\begin{equation*}
C_b\norm{\nabla u}_{L^\infty}\norm{\omega}_{L^1}\eta.
\end{equation*}
Combining the two upper equations we get
\begin{equation*}
|T_{2,2}| \leq C_b\norm{\nabla u}_{L^\infty}\norm{\omega}_{L^1\cap L^\infty}\eta.
\end{equation*}
\end{proof}

\begin{claim}\label{bound_T23}
\begin{equation*}
|T_{2,3}| \leq C_{b,s}\norm{u}_{W^{1,\infty}}\norm{\omega}_{L^1((1+|x|)\dd x)}(1+I(Q_N))\eta^s.
\end{equation*}
\end{claim}

\begin{proof}[Proof of the claim]
We write $T_{2,3}$ as
\begin{align*}
T_{2,3} 
=& \frac{1}{N}\sum_{i=1}^N\bigg(\iint_{\mathbb{R}^2\times\mathbb{R}^2} \omega(x)u(x)\cdot(\nabla_x S_b(x,q_i) - \nabla_x S_b(x,y))\dd \widetilde{\delta}^{(\eta)}_{q_i}(y) \dd x \\
&+ \iint_{\mathbb{R}^2\times\mathbb{R}^2} \omega(x)(u(q_i) - u(y))\cdot \nabla_x S_b(q_i,x) \dd \widetilde{\delta}^{(\eta)}_{q_i}(y) \dd x \\
&+ \iint_{\mathbb{R}^2\times\mathbb{R}^2} \omega(x) u(y)\cdot(\nabla_x S_b(q_i,x) - \nabla_x S_b(y,x))\dd \widetilde{\delta}^{(\eta)}_{q_i}(y) \dd x\bigg).
\end{align*}
Using Claims $(1)$ and $(2)$ of Lemma~\ref{estimees_S_b_not_necessarily_symmetric}, we get that for some $0 < s < 1$,
\begin{align*}
|T_{2,3}| \leq& \frac{1}{N}\sum_{i=1}^N\bigg(C_{b,s}\norm{u}_{L^\infty}\norm{\omega}_{L^1}(1+|q_i|)\eta^s \\
&+ \norm{\nabla u}_{L^\infty}\norm{\omega}_{L^1((1+|x|)\dd x)}\eta \\
&+ \norm{u}_{L^\infty}\norm{\omega}_{L^1((1+|x|)\dd x)}\eta^s\bigg) \\
\leq& C_{b,s}\norm{u}_{W^{1,\infty}}\norm{\omega}_{L^1((1+|x|)\dd x)}(1+I(Q_N))\eta^s.
\end{align*}
\end{proof}

Combining Claims~\ref{bound_T21}, \ref{bound_T22} and \ref{bound_T23} we get that
\begin{equation}\label{gronwall_bound_T2}
|T_2| \leq C_{b,s}\norm{u}_{W^{1,\infty}}\norm{\omega}_{L^1((1+|x|)\dd x)\cap L^\infty}(1+I(Q_N))\eta^s.
\end{equation}

Now let us write $T_3$ as

\begin{align*}
T_3 
=& \frac{1}{N^2}\sum_{1\leq i,j \leq N} \iint_{(\mathbb{R}^2\times\mathbb{R}^2)\backslash\Delta} u(x)\cdot \nabla_x g_b(x,y) \\
&(\dd \delta_{q_i}(x) \dd \delta_{q_j}(y) - \dd\widetilde{\delta}^{(\eta)}_{q_i}(x)\dd \widetilde{\delta}^{(\eta)}_{q_j}(y)) \\
=&\frac{1}{N^2}\sum_{1\leq i,j \leq N} \iint_{(\mathbb{R}^2\times\mathbb{R}^2)\backslash\Delta} u(x)\cdot \nabla \sqrt{b}(x) \sqrt{b(y)}g(x-y) \\
&(\dd \delta_{q_i}(x) \dd \delta_{q_j}(y) - \dd\widetilde{\delta}^{(\eta)}_{q_i}(x)\dd \widetilde{\delta}^{(\eta)}_{q_j}(y)) \\
&+ \frac{1}{N^2}\sum_{1\leq i,j \leq N} \iint_{(\mathbb{R}^2\times\mathbb{R}^2)\backslash\Delta}\sqrt{b(x)b(y)}u(x)\cdot \nabla g(x-y) \\
&(\dd \delta_{q_i}(x) \dd \delta_{q_j}(y) - \dd\widetilde{\delta}^{(\eta)}_{q_i}(x)\dd \widetilde{\delta}^{(\eta)}_{q_j}(y)) \\
&+ \frac{1}{N^2}\sum_{1\leq i,j \leq N} \iint_{(\mathbb{R}^2\times\mathbb{R}^2)\backslash\Delta}u(x)\cdot \nabla_x S_b(x,y) \\
&(\dd \delta_{q_i}(x) \dd \delta_{q_j}(y) - \dd\widetilde{\delta}^{(\eta)}_{q_i}(x)\dd \widetilde{\delta}^{(\eta)}_{q_j}(y)) \\ 
=:& T_{3,1} + T_{3,2} + T_{3,3}.
\end{align*}

We bound the first term:
\begin{claim}\label{gronwall_bound_T31}
\begin{multline*}
|T_{3,1}| \leq C_b\norm{u}_{L^\infty}|\mathcal{F}_b(Q_N,\omega)|  + C_b\norm{u}_{W^{1,\infty}}\bigg(\frac{g(\eta)}{N} \\ 
+ I(Q_N)(\eta+N^{-1}) + \norm{\omega}_{L^1((1+|x|)\dd x)\cap L^\infty}g(\eta)\eta\bigg).
\end{multline*}
\end{claim}

\begin{proof}[Proof of the claim]
We write
\begin{align*}
T_{3,1} 
=& -\frac{1}{N^2} \sum_{i=1}^N  \iint_{\mathbb{R}^2\times\mathbb{R}^2}u(x)\cdot \nabla \sqrt{b}(x) \sqrt{b(y)}g(x-y) \dd \widetilde{\delta}^{(\eta)}_{q_i}(x)\dd \widetilde{\delta}^{(\eta)}_{q_i}(y)  \\
&+ \frac{1}{N^2}\sum_{1\leq i \neq j \leq N}\iint_{(\mathbb{R}^2\times\mathbb{R}^2)\backslash\Delta}u(x)\cdot \nabla \sqrt{b}(x) \sqrt{b(y)}g(x-y) \\
&(\dd \delta_{q_i}(x) \dd \delta_{q_j}(y) - \dd\widetilde{\delta}^{(\eta)}_{q_i}(x)\dd \widetilde{\delta}^{(\eta)}_{q_j}(y))  \\
=:& T_{3,1,1} + T_{3,1,2}.
\end{align*}
By the definition of $\widetilde{\delta}_{q_i}^{(\eta)}$ \eqref{definition_delta_tilde_q} we have 
\begin{align*}
T_{3,1,1} =& -\frac{1}{N^2}\sum_{i=1}^N m_b(q_i,\eta)^2\iint_{\mathbb{R}^2\times\mathbb{R}^2}u(x)\cdot \frac{\nabla \sqrt{b}(x)}{\sqrt{b(x)}}g(x-y)\dd \delta_{q_i}^{(\eta)}(x)\dd \delta_{q_i}^{(\eta)}(y) \\
=& -\frac{1}{N^2}\sum_{i=1}^N m_b(q_i,\eta)^2 \int_{\mathbb{R}^2} \frac{u(x)\cdot\nabla \sqrt{b}(x)}{\sqrt{b}(x)}g^{(\eta)}(x-q_i)\dd \delta_{q_i}^{(\eta)}(x)
\end{align*}
by Claim \eqref{egalite_convolution_g_eta}. It follows by Assumption \ref{assumption_b} that
\begin{equation}\label{boundT311}
|T_{3,1,1}| \leq \frac{C_b\norm{u}_{L^\infty}g(\eta)}{N}.
\end{equation} 
Now we write
\begin{align*}
T_{3,1,2} =& \frac{1}{N^2}\sum_{1\leq i \neq j \leq N}\bigg((u\cdot \nabla \sqrt{b})(q_i) \sqrt{b(q_j)}g(q_i-q_j) \\
&- \iint_{\mathbb{R}^2\times\mathbb{R}^2}(u\cdot \nabla \sqrt{b})(x)\sqrt{b(y)}g(x-y)\dd\widetilde{\delta}^{(\eta)}_{q_i}(x)\dd \widetilde{\delta}^{(\eta)}_{q_j}(y) \bigg)\\
=&  \frac{1}{N^2}\sum_{1\leq i \neq j \leq N}\bigg((u\cdot \nabla \sqrt{b})(q_i) \sqrt{b(q_j)}g(q_i-q_j) \\
&- m_b(q_j,\eta)\int_{\mathbb{R}^2}(u\cdot \nabla \sqrt{b})(x)g^{(\eta)}(x-q_j)\dd\widetilde{\delta}^{(\eta)}_{q_i}(x)\bigg)
\end{align*}
by the definition of $\widetilde{\delta}_{q_i}^{(\eta)}$ \eqref{definition_delta_tilde_q} and Claim \eqref{egalite_convolution_g_eta}. Now,
\begin{align*}
T_{3,1,2} =& \frac{1}{N^2}\sum_{1\leq i \neq j \leq N}(u\cdot \nabla \sqrt{b})(q_i) \sqrt{b(q_j)}(g(q_i-q_j)-g^{(\eta)}(q_i-q_j)) \\
&+ \frac{1}{N^2}\sum_{1\leq i \neq j \leq N}(u\cdot \nabla \sqrt{b})(q_i) \sqrt{b(q_j)}\\
&\times\int_{\mathbb{R}^2}(g^{(\eta)}(q_i-q_j)-g^{(\eta)}(x-q_j))\dd\delta^{(\eta)}_{q_i}(x) \\
&+ \frac{1}{N^2}\sum_{1\leq i \neq j \leq N}(u\cdot \nabla \sqrt{b})(q_i) \sqrt{b(q_j)}\int_{\mathbb{R}^2}g^{(\eta)}(x-q_j)\dd(\delta^{(\eta)}_{q_i}-\widetilde{\delta}^{(\eta)}_{q_i})(x) \\
&+ \frac{1}{N^2}\sum_{1\leq i \neq j \leq N}\sqrt{b(q_j)}\\
&\times\int_{\mathbb{R}^2}((u\cdot \nabla \sqrt{b})(q_i) - (u\cdot \nabla \sqrt{b})(x))g^{(\eta)}(x-q_j)\dd\widetilde{\delta}^{(\eta)}_{q_i}(x) \\
&+ \frac{1}{N^2}\sum_{1\leq i \neq j \leq N}(\sqrt{b(q_j)}-m_b(q_j,\eta))\\
&\times\int_{\mathbb{R}^2}(u\cdot \nabla \sqrt{b})(x)g^{(\eta)}(x-q_j)\dd\widetilde{\delta}^{(\eta)}_{q_i}(x) \\
=&: S_1 + S_2 + S_3 + S_4 + S_5.
\end{align*}
Since $g - g^{(\eta)}$ is nonnegative we can bound
\begin{equation}\label{bound_S_1}
\begin{aligned}
|S_1| \leq& C_b\norm{u}_{L^\infty}\frac{1}{N^2} \sum_{1\leq i \neq j \leq N}(g(q_i-q_j)-g^{(\eta)}(q_i-q_j)) \\
\leq&  C_b\norm{u}_{L^\infty}|\mathcal{F}_b(Q_N,\omega)|  + C_b\norm{u}_{L^\infty}\bigg(\frac{g(\eta)}{N} \\ 
&+ I(Q_N)(\eta+N^{-1}) + \norm{\omega}_{L^1((1+|x|)\dd x)\cap L^\infty}g(\eta)\eta\bigg)
\end{aligned}
\end{equation}
by Proposition \ref{proposition_pre_coerc}. 
Now remark that if $|q_i - q_j| \geq 2\eta$ and $x \in \partial B(q_i,\eta)$, \begin{align*}
|q_j - x| \geq |q_i-q_j| - |q_i-x| \geq 2\eta - \eta \geq \eta
\end{align*} 
and it follows by Claim \eqref{egalite_convolution_g_eta} that 
\begin{align*}
\int_{\mathbb{R}^2}g^{(\eta)}(x-q_j)\dd\delta^{(\eta)}_{q_i}(x) &= \int_{\mathbb{R}^2}g(x-q_j)\dd\delta^{(\eta)}_{q_i}(x) \\
&= g^{(\eta)}(q_i-q_j).
\end{align*}
Hence we can write
\begin{equation*}
S_2 = \frac{1}{N^2}\underset{|q_i - q_j| \leq 2\eta}{\sum_{1\leq i \neq j \leq N}}(u\cdot \nabla \sqrt{b})(q_i) \sqrt{b(q_j)}\int_{\mathbb{R}^2}(g^{(\eta)}(q_i-q_j)-g^{(\eta)}(x-q_j))\dd\delta^{(\eta)}_{q_i}(x).
\end{equation*}
Notice that if $|q_i - q_j| \leq 2\eta$ and $x \in \partial B(q_i,\eta)$, then
\begin{equation*}
|g^{(\eta)}(q_i-q_j)-g^{(\eta)}(x-q_j)| \leq \norm{\nabla g^{(\eta)}}_{L^\infty}\eta = C\eta^{-1}\eta \leq C.
\end{equation*}
Therefore,
\begin{equation}\label{bound_S_2}
\begin{aligned}
|S_2| \leq& \frac{C_b \norm{u}_{L^\infty}}{N^2} |\{(q_i,q_j) ; |q_i - q_j| \leq 2\eta\}| \\
\leq& C_b \norm{u}_{L^\infty}|\mathcal{F}_b(Q_N,\omega)| + C_b \norm{u}_{L^\infty}\bigg(\frac{g(2\eta)}{N} \\
&+ I(Q_N)(2\eta+N^{-1}) + \norm{\omega}_{L^1((1+|x|)\dd x)\cap L^\infty}g(2\eta)2\eta\bigg)
\end{aligned}
\end{equation}
by Corollary \ref{corollary_counting} applied to $\epsilon = 2\eta$.

By definition of $\widetilde{\delta}_{q_i}^{(\eta)}$ \eqref{definition_delta_tilde_q} we can write
\begin{align*}
S_3 =& \frac{1}{N^2}\sum_{1\leq i \neq j \leq N}(u\cdot \nabla \sqrt{b})(q_i) \sqrt{b(q_j)} \\
&\times\int_{\mathbb{R}^2}g^{(\eta)}(x-q_j)\left(1-\frac{m_b(q_i,\eta)}{\sqrt{b(x)}}\right)\dd\delta^{(\eta)}_{q_i}(x)
\end{align*}
and therefore
\begin{equation*}
|S_3| \leq \frac{C_b\norm{u}_{L^\infty}g(\eta)}{N}\sum_{i=1}^N\int_{\mathbb{R}^2}\left|\frac{m_b(q_i,\eta)}{\sqrt{b(x)}}-1\right|\dd\delta^{(\eta)}_{q_i}(x).
\end{equation*}
For $x \in \partial B(q_i,\eta)$, we have
\begin{align*}
\left|\frac{m_b(q_i,\eta)}{\sqrt{b(x)}}-1\right| &\leq C_b\left|m_b(q_i,\eta)^{-1}-\frac{1}{\sqrt{b(x)}}\right| \\
&\leq C_b\left|\int_{\mathbb{R}^2}\frac{\dd \delta_{q_i}^{(\eta)}(y)}{\sqrt{b(y)}}-\frac{1}{\sqrt{b(x)}}\right| \\
&\leq C_b\eta
\end{align*}
since $b$ is Lipschitz by Assumption \ref{assumption_b}. It follows that
\begin{equation}\label{bound_S_3}
|S_3| \leq C_b\norm{u}_{L^\infty}g(\eta)\eta.
\end{equation}
Now by regularity of $u$, $b$ and Proposition \ref{proposition_regularisation}, we have
\begin{align*}
|S_4| + |S_5| \leq C_b\norm{u}_{W^{1,\infty}}\eta g(\eta).
\end{align*} 
Combining the upper inequality with \eqref{boundT311}, \eqref{bound_S_1}, \eqref{bound_S_2} and \eqref{bound_S_3} we obtain Claim~\ref{gronwall_bound_T31}.
\end{proof}

For the third term we have the following bound:
\begin{claim}\label{gronwall_bound_T33}
For $s$ small enough, we have 
\begin{equation*}
|T_{3,3}| \leq C_{b,s}\norm{u}_{W^{1,\infty}}(1+I(Q_N))\eta^s.
\end{equation*}
\end{claim}

\begin{proof}
We write
\begin{align*}
T_{3,3} 
=& \frac{1}{N^2}\sum_{1 \leq i,j \leq N} u(q_i)\cdot \nabla_x S_b(q_i,q_j) \\
&- \frac{1}{N^2}\sum_{1 \leq i,j \leq N}\iint_{\mathbb{R}^2\times\mathbb{R}^2}u(x) \cdot \nabla_x S_b(x,y)\dd \widetilde{\delta}_{q_i}^{(\eta)}(x)\dd \widetilde{\delta}_{q_j}^{(\eta)}(y) \\
=& \frac{1}{N^2}\sum_{1 \leq i,j \leq N} \iint_{\mathbb{R}^2\times\mathbb{R}^2} \bigg(u(q_i)\cdot \nabla_x S_b(q_i,q_j) - u(q_i)\cdot \nabla_x S_b(q_i,y) \\
&+ u(q_i)\cdot \nabla_x S_b(q_i,y) - u(q_i)\cdot \nabla_x S_b(x,y) \\
&+ u(q_i) \cdot \nabla_x S_b(x,y) - u(x) \cdot \nabla_x S_b(x,y)\bigg)\dd \widetilde{\delta}_{q_i}^{(\eta)}(x)\dd \widetilde{\delta}_{q_j}^{(\eta)}(y).
\end{align*}
Therefore,
\begin{align*}
|T_{3,3}| \leq& \frac{1}{N^2}\sum_{1 \leq i,j \leq N} \norm{u}_{L^\infty}|\nabla_x S_b(q_i,\cdot)|_{\mathcal{C}^{0,s}(B(q_j,1))}\eta^s \\
&+ \frac{1}{N^2}\sum_{1 \leq i,j \leq N} \norm{u}_{L^\infty}\eta^s\int_{\mathbb{R}^2}|\nabla_x S_b(\cdot,y)|_{\mathcal{C}^{0,s}}\dd \widetilde{\delta}_{q_j}^{(\eta)}(y) \\
&+ \frac{1}{N^2}\sum_{1 \leq i,j \leq N}\iint_{\mathbb{R}^2\times\mathbb{R}^2} \norm{u}_{W^{1,\infty}}\eta|\nabla_x S_b(x,y)|\dd \widetilde{\delta}_{q_i}^{(\eta)}(x)\dd \widetilde{\delta}_{q_j}^{(\eta)}(y).
\end{align*}
By Proposition~\ref{estimees_S_b_not_necessarily_symmetric}, for $s$ small enough we have
\begin{align*}
|T_{3,3}| \leq& \frac{C_{b,s}}{N^2}\sum_{1 \leq i,j \leq N} \norm{u}_{L^\infty}|(1+|q_j|)\eta^s \\
&+ \frac{C_{b,s}}{N^2}\sum_{1 \leq i,j \leq N} \norm{u}_{L^\infty}\eta^s(1+|q_j|) \\
&+ \frac{C_b}{N^2}\sum_{1 \leq i,j \leq N} \norm{u}_{W^{1,\infty}}\eta(1+|q_j|) \\
\leq& C_{b,s}\norm{u}_{W^{1,\infty}}(1+I(Q_N))\eta^s.
\end{align*}
\end{proof}

We are only remained to bound $T_{3,2}$:
\begin{claim}\label{gronwall_bound_T32}
For $\epsilon > 2\eta$ small enough, we have
\begin{equation*}
\begin{aligned}
|T_{3,2}| \leq& \frac{C_b}{N}\norm{\nabla u}_{L^\infty} + \frac{C_b\eta\norm{\nabla u}_{L^\infty}}{\epsilon} + C_b\norm{\nabla u}_{L^\infty}\bigg(|\mathcal{F}_b(Q_N,\omega)| +\frac{g(\epsilon)}{N} +\eta \\
&+ I(Q_N)(\epsilon+N^{-1}) + \norm{\omega}_{L^1((1+|x|)\dd x)\cap L^\infty}g(\epsilon)\epsilon\bigg).
\end{aligned}
\end{equation*}
\end{claim}

\begin{proof}
Let us denote
\begin{equation*}
k_u(x,y) = (u(x) - u(y))\cdot \nabla g(x-y)
\end{equation*}
and remark that
\begin{equation}\label{bound_ku}
|k_u(x,y)| \leq C\norm{\nabla u}_{L^\infty}.
\end{equation}
Since $\nabla g$ is antisymmetric we can write $T_{3,2}$ as
\begin{multline*}
T_{3,2} = \\ \frac{1}{2N^2}\sum_{1\leq i,j \leq N} \iint_{(\mathbb{R}^2\times\mathbb{R}^2)\backslash\Delta}\sqrt{b(x)b(y)}k_u(x,y)\dd(\delta_{q_i} + \widetilde{\delta}^{(\eta)}_{q_i})(x)\dd(\delta_{q_j} - \widetilde{\delta}^{(\eta)}_{q_j})(y).
\end{multline*}
Using the definition of $\widetilde{\delta}^{(\eta)}_{q_i}$ \eqref{definition_delta_tilde_q} we can write
\begin{align*}
\dd(\delta_{q_i} + \widetilde{\delta}^{(\eta)}_{q_i})(x)&\dd(\delta_{q_j} - \widetilde{\delta}^{(\eta)}_{q_j})(y) \\
=& \dd\bigg(\delta_{q_i} + \frac{m_b(q_i,\eta)}{\sqrt{b}}\delta_{q_i}^{(\eta)}\bigg)(x)\dd\bigg(\delta_{q_j} - \frac{m_b(q_j,\eta)}{\sqrt{b}}\delta_{q_j}^{(\eta)}\bigg)(y) \\
=& \frac{m_b(q_i,\eta)m_b(q_j,\eta)}{\sqrt{b(x)b(y)}}\dd(\delta_{q_i} + \delta_{q_i}^{(\eta)})(x)\dd(\delta_{q_j} - \delta_{q_j}^{(\eta)})(y) \\
&+ \left(1-\frac{m_b(q_i,\eta)m_b(q_j,\eta)}{\sqrt{b(q_i)b(q_j)}}\right)\dd\delta_{q_i}(x)\dd\delta_{q_j}(y) \\
&+ \frac{m_b(q_i,\eta)}{\sqrt{b(q_i)}}\left(1-\frac{m_b(q_j,\eta)}{\sqrt{b(q_j)}}\right)\dd\delta_{q_i}^{(\eta)}(x)\dd\delta_{q_j}(y) \\\\
&+ \frac{m_b(q_j,\eta)}{\sqrt{b(y)}}\left(\frac{m_b(q_i,\eta)}{\sqrt{b(q_i)}}-1\right)\dd\delta_{q_i}(x)\dd\delta_{q_j}^{(\eta)}(y).
\end{align*}
We will use some inequalities proved in \cite{NguyenRosenzweigSerfaty} and Corollary~\ref{corollary_counting} to control the first line, but let us begin by controling the three last remainders. Using the bound \eqref{bound_ku} and \eqref{estimee_m_b} we can bound
\begin{align*}
T_{3,2,2} :=& \frac{1}{2N^2}\sum_{1\leq i,j \leq N}\iint_{(\mathbb{R}^2\times\mathbb{R}^2)\backslash\Delta} \sqrt{b(x)b(y)}k_u(x,y) \\
&\bigg(\left(1-\frac{m_b(q_i,\eta)m_b(q_j,\eta)}{\sqrt{b(q_i)b(q_j)}}\right)\dd\delta_{q_i}(x)\dd\delta_{q_j}(y) \\
&+ \frac{m_b(q_i,\eta)}{\sqrt{b(q_i)}}\left(1-\frac{m_b(q_j,\eta)}{\sqrt{b(q_j)}}\right)\dd\delta_{q_i}^{(\eta)}(x)\dd\delta_{q_j}(y) \\
&+ \frac{m_b(q_j,\eta)}{\sqrt{b(y)}}\left(\frac{m_b(q_i,\eta)}{\sqrt{b(q_i)}}-1\right)\dd\delta_{q_i}(x)\dd\delta_{q_j}^{(\eta)}(x)\bigg).
\end{align*}
by
\begin{equation}\label{gronwall_bound_T322}
|T_{3,2,2}| \leq C_b\norm{\nabla u}_{L^\infty}\eta.
\end{equation}
We are remained to bound
\begin{align*}
T_{3,2,1} :=& \frac{1}{2N^2}\sum_{1\leq i,j \leq N} m_b(q_i,\eta)m_b(q_j,\eta) \\
&\times\iint_{(\mathbb{R}^2\times\mathbb{R}^2)\backslash\Delta} k_u(x,y) \dd(\delta_{q_i} + \delta_{q_i}^{(\eta)})(x)\dd(\delta_{q_j} - \delta_{q_j}^{(\eta)}(y).
\end{align*}
Using decomposition (4.26) and inequalities (4.27), (4.28) and (4.31) of \cite{NguyenRosenzweigSerfaty} with $s=0$ and $m=0$ (remark that we can choose $m=0$ since no extension procedure is needed for $s=0$ and $d=2$, for more details we refer to the introduction of \cite[Section~4]{NguyenRosenzweigSerfaty}), we get that for any small parameter $\epsilon > 2\eta$,
\begin{equation*}
|T_{3,2,1}| \leq  \frac{C_b}{N}\norm{\nabla u}_{L^\infty} + \frac{C_b\norm{\nabla u}_{L^\infty}}{N^2}|\{(q_i,q_j) ; |q_i - q_j| \leq \epsilon\}| + \frac{C\eta\norm{\nabla u}_{L^\infty}}{\epsilon}.
\end{equation*}
Using Corollary~\ref{corollary_counting}, we get that
\begin{equation}\label{gronwall_bound_T321}
\begin{aligned}
T_{3,2,1} \leq& \frac{C_b}{N}\norm{\nabla u}_{L^\infty} + \frac{C_b\eta\norm{\nabla u}_{L^\infty}}{\epsilon} + C_b\norm{\nabla u}_{L^\infty}\bigg(|\mathcal{F}_b(Q_N,\omega)| \\
&+\frac{g(\epsilon)}{N} + I(Q_N)(\epsilon+N^{-1}) + \norm{\omega}_{L^1((1+|x|)\dd x)\cap L^\infty}g(\epsilon)\epsilon\bigg).
\end{aligned}
\end{equation}
And we get Claim \ref{gronwall_bound_T32} by combining \eqref{gronwall_bound_T321} with \eqref{gronwall_bound_T322}.
\end{proof}

We finish the proof of Proposition \ref{controle_terme_principal_gronwall} using Decomposition \eqref{decomposition_main_proposition},
Inequalities \eqref{gronwall_bound_T1}, \eqref{gronwall_bound_T2} and Claims~\ref{gronwall_bound_T31}, \ref{gronwall_bound_T33} and \ref{gronwall_bound_T32}. That gives
\begin{align*}
\bigg|\iint_{(\mathbb{R}^2\times\mathbb{R}^2)\backslash\Delta} &u(x)\cdot \nabla_x g_b(x,y)\dd\left(\frac{1}{N}\sum_{i=1}^N\delta_{q_i}-\omega\right)^{\otimes 2}(x,y)\bigg| \\
\leq& C_b\norm{u}_{W^{1,\infty}}\bigg(|\mathcal{F}_b(Q_N,\omega)| + \frac{g(\eta)}{N} + I(Q_N)(\eta+N^{-1}) \\ 
&+ \norm{\omega}_{L^1((1+|x|)\dd x)\cap L^\infty}g(\eta)\eta\bigg)\\
&+ C_{b,s}\norm{u}_{W^{1,\infty}}\norm{\omega}_{L^1((1+|x|)\dd x)\cap L^\infty}(1+I(Q_N))\eta^s\\
&+ C_b\norm{u}_{L^\infty}\mathcal{F}_b(Q_N,\omega)  + C_b\norm{u}_{W^{1,\infty}}\bigg(\frac{g(\eta)}{N} \\ 
&+ I(Q_N)(\eta+N^{-1}) + \norm{\omega}_{L^1((1+|x|)\dd x)\cap L^\infty}g(\eta)\eta\bigg)\\
&+ C_{b,s}\norm{u}_{W^{1,\infty}}(1+I(Q_N))\eta^s \\
&+\frac{C_b}{N}\norm{\nabla u}_{L^\infty} + \frac{C_b\eta\norm{\nabla u}_{L^\infty}}{\epsilon} + C_b\norm{\nabla u}_{L^\infty}\bigg(|\mathcal{F}_b(Q_N,\omega)| \\
&+\frac{g(\epsilon)}{N} + \eta + I(Q_N)(\epsilon+N^{-1}) + \norm{\omega}_{L^1((1+|x|)\dd x)\cap L^\infty}g(\epsilon)\epsilon\bigg).
\end{align*}
Choosing $\epsilon = N^{-1}$ and $\eta = N^{-2}$, and since $\norm{\omega}_{L^1((1+|x|)\dd x)\cap L^\infty}$ is bounded by below (because $\omega$ is a probability density) we get that
\begin{align*}
\bigg|\iint_{(\mathbb{R}^2\times\mathbb{R}^2)\backslash\Delta} &u(x)\cdot \nabla_x g_b(x,y)\dd\left(\frac{1}{N}\sum_{i=1}^N\delta_{q_i}-\omega\right)^{\otimes 2}(x,y)\bigg| \\
\leq& C_b\norm{u}_{W^{1,\infty}}|\mathcal{F}_b(Q_N,\omega)| \\
&+ C_b(1+\norm{u}_{W^{1,\infty}})\norm{\omega}_{L^1((1+|x|)\dd x)\cap L^\infty}(1+I(Q_N))N^{-\beta}
\end{align*}
for some $0 < \beta < 1$.
\end{proof}

\section{Mean-field limit}\label{section:7}

In this section we prove the mean-field limit Theorem~\ref{MFL_theorem}. For this purpose let us first prove the following estimates:

\begin{theorem}
If $\omega$ is a weak solution of \eqref{lake_equation_vorticity} with initial datum $\omega_0$ (in the sense of Definition \ref{definition_weak_solution}) that satisfies Assumption~\ref{assumption_omega} and if $I_N(0)$ is bounded, there exists a constant 
\begin{equation*}
A := A\left(b,T,\norm{u}_{L^\infty([0,T],W^{1,\infty})},\norm{\omega}_{L^\infty([0,T],L^1((1+|x|)\dd x)\cap L^\infty)},\underset{N}{\sup} \; I_N(0)\right)
\end{equation*}
such that for every $t \in [0,T]$,
\begin{equation}\label{gronwall_bound_alpha_regime}
|\mathcal{F}_{b,N}(t)| \leq A(|\mathcal{F}_{b,N}(0)| + (1+|E_N(0)|)(N^{-\beta} + |\alpha_N - \alpha|)).
\end{equation}
If $\overline{\omega}$ is a weak solution of \eqref{transport_equation} with initial datum $\omega_0$ (in the sense of Definition \ref{definition_weak_solution_transport}) that satisfies Assumption~\ref{assumption_omega} and if $\overline{I_N}(0)$ is bounded, there exists a constant 
\begin{align*}
B :=& B\bigg(b, T,\norm{\overline{\omega}}_{L^\infty([0,T],L^1((1+|x|)\dd x)\cap L^\infty)}, \\
&\norm{\nabla g*\overline{\omega}}_{L^\infty([0,T],W^{1,\infty})},\underset{N}{\sup} \;\overline{I_N}(0)\bigg)
\end{align*}
such that for every $t \in [0,T]$,
\begin{equation}\label{gronwall_bound_rescaled_regime}
|\overline{\mathcal{F}}_{b,N}(t)|  \leq B(|\overline{\mathcal{F}}_{b,N}(0)| + (1+|\overline{E_N}(0)|)(N^{-\beta} + \alpha_N^{-1})).
\end{equation}
\end{theorem}

\begin{proof}
By Proposition~\ref{time_derivative_F_N}, we have that for almost every $t \in [0,T]$,
\begin{align*}
&\frac{\dd}{\dd t}\mathcal{F}_{b,N}(t) \\
=& 2\iint_{(\mathbb{R}^2\times\mathbb{R}^2)\backslash\Delta} \left(u(t,x) - \alpha\frac{\nabla^\bot b(x)}{b(x)}\right)\cdot \nabla_x g_b(x,y)\dd(\omega(t)-\omega_N(t))^{\otimes 2}(x,y) \\
&+ 2(\alpha_N - \alpha)\iint_{(\mathbb{R}^2\times\mathbb{R}^2)\backslash\Delta} \frac{\nabla^\bot b(x)}{b(x)}\cdot \nabla_x g_b(x,y) \dd \omega_N(t,x) \dd(\omega(t) - \omega_N(t))(y) \\
=:& L_1 + 2(\alpha_N - \alpha)L_2.
\end{align*}
Using Proposition~\ref{controle_terme_principal_gronwall}, we have
\begin{align*}
|L_1| \leq& C_b\norm{u-\alpha\frac{\nabla^\bot b}{b}}_{L^\infty([0,T],W^{1,\infty})}|\mathcal{F}_b(Q_N,\omega)| \\
&+ C_b\left(1+\norm{u-\alpha\frac{\nabla^\bot b}{b}}_{L^\infty([0,T],W^{1,\infty})}\right)\\
&\times\norm{\omega}_{L^\infty([0,T],L^1((1+|x|)\dd x)\cap L^\infty)}
(1+I_N(t))N^{-\beta}.
\end{align*}
By Proposition~\ref{proposition_pv_well_posed}, we have
\begin{align*}
I_N(t) &\leq C_{b,T}(1+|E_N(0)| + I_N(0))
\end{align*}
since $(\alpha_N)$ is bounded (here we consider the case $\alpha_N \Tend{N}{+\infty} \alpha$). Therefore,
\begin{equation}\label{gronwall_alpha_regime_L1}
\begin{aligned}
|L_1| \leq& C_b\big(1+\norm{u}_{L^\infty([0,T],W^{1,\infty})}\big)|\mathcal{F}_b(Q_N,\omega)| + C_{b,T}\left(1+\norm{u}_{L^\infty([0,T],W^{1,\infty})}\right)\\
&\times\norm{\omega}_{L^\infty([0,T],L^1((1+|x|)\dd x)\cap L^\infty)}
(1+I_N(0)+|E_N(0)|)N^{-\beta}.
\end{aligned}
\end{equation}
Now
\begin{align*}
L_2 =& \frac{1}{N}\sum_{i=1}^N\frac{\nabla^\bot b(q_i)}{b(q_i)} \\
&\cdot\bigg[\int_{\mathbb{R}^2\backslash\{q_i\}} \sqrt{b(q_i)b(y)}\nabla g(q_i - y)\dd\bigg(\omega(t)-\frac{1}{N}\sum_{j=1}^N\delta_{q_j(t)}\bigg) \\
&+ \int_{\mathbb{R}^2\backslash\{q_i\}} \nabla_x S_b(q_i,y)\dd\bigg(\omega(t)-\frac{1}{N}\sum_{j=1}^N\delta_{q_j(t)}\bigg)\bigg] \\
=:& L_{2,1} + L_{2,2} + L_{2,3}
\end{align*}
with
\begin{equation}\label{gronwall_alpha_regime_L21}
\begin{aligned}
|L_{2,1}| &= \bigg|\frac{1}{N}\sum_{i=1}^N\frac{\nabla^\bot b(q_i)}{\sqrt{b(q_i)}}\cdot\int_{\mathbb{R}^2} \nabla g(q_i - y)\sqrt{b(y)}\omega(t,y)\dd y\bigg| \\
&\leq C_b\norm{\omega}_{L^\infty([0,T],L^1\cap L^\infty)}
\end{aligned}
\end{equation}
(for the last inequality see for example  \cite[Lemma~1]{Iftimie}). For the second term
\begin{align*}
L_{2,2} &= -\frac{1}{N^2}\sum_{i=1}^N\underset{j \neq i}{\sum_{j=1}^N} \sqrt{b(q_i)b(q_i)}\frac{\nabla^\bot b(q_i)}{b(q_i)}\cdot \nabla g(q_i-q_j).
\end{align*}
We can bound it as in \eqref{bound_E_N_2} to get 
\begin{equation}\label{gronwall_alpha_regime_L22}
|L_{2,2}| \leq C_b.
\end{equation}
For the last term, we use Claim $(1)$ of Lemma~\ref{estimees_S_b_not_necessarily_symmetric} to get
\begin{align*}
|L_{2,3}| &= \bigg|\frac{1}{N}\sum_{i=1}^N\frac{\nabla^\bot b(q_i)}{b(q_i)}\cdot\int_{\mathbb{R}^2} \nabla_x S_b(q_i,y)\dd\bigg(\omega(t)-\frac{1}{N}\underset{j\neq i}{\sum_{j=1}^N}\delta_{q_j(t)}\bigg)(y)\bigg| \\
&\leq C_b \int_{\mathbb{R}^2} (1+|y|)\dd\bigg(\omega(t)+\frac{1}{N}\underset{j\neq i}{\sum_{j=1}^N}\delta_{q_j(t)}\bigg)(y) \\
&\leq C_b(\norm{\omega}_{L^\infty([0,T],L^1((1+|x|)\dd x))} + I_N(t)) \\
&\leq C_{b,T}(\norm{\omega}_{L^\infty([0,T],L^1((1+|x|)\dd x))} + 1+I_N(0) + |E_N(0)|)
\end{align*}
by Proposition~\ref{proposition_pv_well_posed}. Combining the upper inequality with \eqref{gronwall_alpha_regime_L1}, \eqref{gronwall_alpha_regime_L21} and \eqref{gronwall_alpha_regime_L22} we get that for almost every $t \in [0,T]$,
\begin{align*}
&\left|\frac{\dd}{\dd t}\mathcal{F}_{b,N}(t)\right| \\
\leq& C_b(1+\norm{u}_{L^\infty([0,T],W^{1,\infty})})|\mathcal{F}_{b,N}(t)| + C_b\left(1+\norm{u}_{L^\infty([0,T],W^{1,\infty})}\right)\\
&\times\norm{\omega}_{L^\infty([0,T],L^1((1+|x|)\dd x)\cap L^\infty)}
(1+I_N(0)+|E_N(0)|)N^{-\beta} \\
&+ C_{b,T}|\alpha_N - \alpha|\bigg(\norm{\omega}_{L^\infty([0,T],L^1((1+|x|)\dd x)\cap L^\infty)} +1+ I_N(0) + |E_N(0)|\bigg).
\end{align*}
Therefore there exists a constant $A$ depending only on the quantities $b$, $T$, $\norm{u}_{L^\infty([0,T],W^{1,\infty})}$, $\norm{\omega}_{L^\infty([0,T],L^1((1+|x|)\dd x)\cap L^\infty)}$ and $I_N(0)$ (which is uniformly bounded in $N$ by assumption) such that for almost every $t \in [0,T]$,
\begin{equation*}
\left|\frac{\dd}{\dd t}\mathcal{F}_{b,N}(t)\right| \leq A(|\mathcal{F}_{b,N}(t)|+(1+|E_N(0)|)(N^{-\beta}+|\alpha_N - \alpha|)).
\end{equation*}
By Grönwall's lemma (up to taking another constant $A$ depending on the same quantities) we get \eqref{gronwall_bound_alpha_regime}.

Now let us study the rescaled regime where $\alpha_N \Tend{N}{+\infty} +\infty$. By Proposition~\ref{time_derivative_F_N_resc} we have
\begin{align*}
\frac{\dd}{\dd t}\overline{\mathcal{F}}_{b,N}(t) =& -2\iint_{(\mathbb{R}^2\times \mathbb{R}^2)\backslash \Delta} \frac{\nabla^\bot b(x)}{b(x)}\cdot \nabla_x g_b(x,y)\dd(\overline{\omega}(t)-\overline{\omega}_N(t))^{\otimes 2}(x,y) \\
&+ \frac{2}{N^2\alpha_N} \sum_{i=1}^N\underset{j\neq i}{\sum_{j=1}^N} \frac{v(t,\overline{q_i})}{b(\overline{q_i})}\cdot\nabla_x g_b(\overline{q_i},\overline{q_j}) \\
=:& L_1 + L_2.
\end{align*}
The first term can be bounded by Proposition~\ref{controle_terme_principal_gronwall}:
\begin{equation}\label{gronwall_rescaled_regime_L1}
\begin{aligned}
|L_1| \leq&  C_b\norm{\frac{\nabla b}{b}}_{W^{1,\infty}}|\overline{\mathcal{F}}_{b,N}(t)| + C_b\bigg(1+\norm{\frac{\nabla b}{b}}_{W^{1,\infty}}\bigg)\\
&\times \norm{\overline{\omega}}_{L^\infty([0,T],L^1((1+|x|)\dd x)\cap L^\infty)}(1+I(\overline{Q_N}))N^{-\beta} \\
\leq& C_b|\overline{\mathcal{F}}_{b,N}(t)| + C_{b,T}\norm{\overline{\omega}}_{L^\infty([0,T],L^1((1+|x|)\dd x)\cap L^\infty)}\\
&\times(1 +\overline{I_N}(0)+|\overline{E_N}(0)|)N^{-\beta}
\end{aligned}
\end{equation}
where we used Proposition~\ref{proposition_pv_well_posed} in the last inequality. We split the second line in three terms:
\begin{align*}
L_2 =& \frac{2}{N^2\alpha_N} \sum_{i=1}^N\underset{j\neq i}{\sum_{j=1}^N} \frac{v(t,\overline{q_i})}{b(\overline{q_i})}\cdot \frac{\nabla b(\overline{q_i})}{2\sqrt{b(\overline{q_i})}}\sqrt{b(\overline{q_j})}g(\overline{q_i} - \overline{q_j}) \\
&+ \frac{2}{N^2\alpha_N} \sum_{i=1}^N\underset{j\neq i}{\sum_{j=1}^N} \frac{v(t,\overline{q_i})}{b(\overline{q_i})}\cdot\nabla g(\overline{q_i} - \overline{q_j})\sqrt{b(\overline{q_i})b(\overline{q_i})} \\
&+\frac{2}{N^2\alpha_N}\sum_{i=1}^N\underset{j\neq i}{\sum_{j=1}^N} \frac{v(t,\overline{q_i})}{b(\overline{q_i})}\cdot \nabla_x S_b(\overline{q_i},\overline{q_j}) \\
=:& L_{2,1} + L_{2,2} + L_{2,3}.
\end{align*}
We can bound the first term by
\begin{align*}
|L_{2,1}| \leq \frac{C_b}{N^2\alpha_N}\norm{v}_{L^\infty}\sum_{i=1}^N\underset{j\neq i}{\sum_{j=1}^N}|g(\overline{q_i}-\overline{q_j})|
\end{align*}
and applying Lemma~\ref{lemma_link_solutions_duerinckx} we get
\begin{equation*}
\norm{v}_{L^\infty} = \norm{\nabla G_b[\overline{\omega}]}_{L^\infty} \leq C_b\norm{\overline{\omega}}_{L^\infty([0,T],L^1\cap L^\infty)}. 
\end{equation*}
We can bound $\displaystyle{\sum_{i=1}^N\underset{j\neq i}{\sum_{j=1}^N}|g(\overline{q_i}-\overline{q_j})|}$ as we did for Inequality \eqref{bound_IN2} to get
\begin{equation*}
|L_{2,1}| \leq C_b\norm{\overline{\omega}}_{L^\infty([0,T],L^1\cap L^\infty)}(1+|\overline{E_N}|+\overline{I_N})\alpha_N^{-1}.
\end{equation*}
The second term $L_{2,2}$ can be bounded as in \eqref{bound_E_N_2} to get
\begin{equation*}
|L_{2,2}| \leq C_b(1+\norm{v}_{L^\infty([0,T],W^{1,\infty})})\alpha_N^{-1}
\end{equation*}
and the last term can be bounded directly using Claim (1) of Lemma~\ref{estimees_S_b_not_necessarily_symmetric}:
\begin{equation*}
|L_{2,3}| \leq C_b\norm{v}_{L^\infty}(1+\overline{I_N})\alpha_N^{-1} \leq C_b\norm{\overline{\omega}}_{L^\infty([0,T],L^1\cap L^\infty)}(1+\overline{I_N})\alpha_N^{-1}.
\end{equation*}
Combining these three inequalities with \eqref{gronwall_rescaled_regime_L1} and using Proposition~\ref{proposition_pv_well_posed} to bound $\overline{I_N}$ we get that for almost every $t \in [0,T]$,
\begin{align*}
\left|\frac{\dd}{\dd t}\overline{\mathcal{F}}_{b,N}(t)\right| \leq&
 C_b|\overline{\mathcal{F}}_{b,N}(t)| + C_b(\norm{\overline{\omega}}_{L^\infty([0,T],L^1((1+|x|)\dd x)\cap L^\infty)}\\
 &+\norm{v}_{L^\infty([0,T],W^{1,\infty})})\\
 &\times(1+|\overline{I_N}(0)|+|\overline{E_N}(0)|)(N^{-\beta} + \alpha_N^{-1}).
\end{align*}
And therefore there exists a constant $B$ depending only on the quantities $b$, $T$,$\norm{\omega}_{L^\infty([0,T],L^1((1+|x|)\dd x)\cap L^\infty)}$ and $\overline{I_N}(0)$ (which is uniformly bounded in $N$ by assumption) such that for almost every $t \in [0,T]$,
\begin{align*}
\left|\frac{\dd}{\dd t}\overline{\mathcal{F}}_{b,N}(t)\right| \leq&
B(|\overline{\mathcal{F}}_{b,N}(t)| + (1+|\overline{E_N}(0)|)(N^{-\beta} + \alpha_N^{-1})).
\end{align*}
By Grönwall's lemma (up to taking another constant $B$ depending on the same quantities) we get \eqref{gronwall_bound_rescaled_regime}.
\end{proof}

\begin{proof}[Proof of Theorem~\ref{MFL_theorem}]
By Corollary \ref{corollary_weak_star_cv}, weak-$\ast$ convergence and convergence of the interaction energy gives that $(\mathcal{F}_{b,N}(0))$ and $(\overline{\mathcal{F}_{b,N}}(0))$ converge to zero. Using convergence of the interaction energy we also get that $|E_N(0)|$ and $|\overline{E_N}(0)|$ are bounded. Thus by Inequalities \eqref{gronwall_bound_alpha_regime} and \eqref{gronwall_bound_rescaled_regime} we get that for any $t \in [0,T]$ $(\mathcal{F}_{b,N}(t))$ and $(\overline{\mathcal{F}_{b,N}}(t))$ converge to zero and the theorem follows by Corollary~\ref{corollary_weak_star_cv}.
\end{proof}

\end{document}